\documentclass[11pt]{amsart}

\usepackage{epigamath}

\usepackage[english]{babel}

\numberwithin{equation}{section}

\usepackage{enumerate}
\usepackage{enumitem}
\usepackage{tikz}
\usetikzlibrary{shapes}

\newtheorem{teo}{Theorem}[section]
\newtheorem{cor}[teo]{Corollary}
\newtheorem{lem}[teo]{Lemma}
\newtheorem{prop}[teo]{Proposition}

\theoremstyle{definition}

\newtheorem{defi}[teo]{Definition}
\newtheorem{ex}[teo]{Example}

\newtheorem{rem}[teo]{Remark}

\newcommand{\Aut}{\operatorname{Aut}}

\newcommand{\SL}{\operatorname{SL}}

\renewcommand{\dim}{\operatorname{dim}}

\newcommand{\Cbb}{{\mathbb C}}
\newcommand{\Qbb}{{\mathbb Q}}
\newcommand{\Zbb}{{\mathbb Z}}

\newcommand{\Pbb}{{\mathbb P}}
\newcommand{\Nbb}{{\mathbb N}}

\def\chadom{
\draw[color=gray!60] (1,0) -- (1,4);
\draw[color=gray!60] (2,0) -- (2,4);
\draw[color=gray!60] (3,0) -- (3,4);

\draw[color=gray!60] (0,1) -- (4,1);
\draw[color=gray!60] (0,2) -- (4,2);
\draw[color=gray!60] (0,3) -- (4,3);
\node at (1.7,1.55) [color=gray] {0};
}
\def\chadombis{
\draw[color=gray!60] (1,0) -- (1,4);
\draw[color=gray!60] (2,0) -- (2,4);
\draw[color=gray!60] (3,0) -- (3,4);
\draw[color=gray!60] (0,1) -- (4,1);
\draw[color=gray!60] (0,2) -- (4,2);
\draw[color=gray!60] (0,3) -- (4,3);
\node at (0.7,0.5) [color=gray] {0};
}

\EpigaVolumeYear{4}{2020} \EpigaArticleNr{4} \ReceivedOn{January 17,
2019}
\InFinalFormOn{February 27, 2020}
\AcceptedOn{March 10, 2020}

\title{Smooth projective horospherical varieties of Picard group $\Zbb^2$}
\titlemark{Smooth horospherical varieties}

\author{Boris Pasquier}
\address{Laboratoire de Math\'ematiques Appliqu\'ees de Poitiers, CNRS, Univ. Poitiers.}
\email{boris.pasquier@univ-poitiers.fr }

\authormark{Boris Pasquier}

\AbstractInEnglish{We classify all smooth projective horospherical varieties of Picard group $\Zbb^2$ and we give a first description of their geometry via the Log Minimal Model Program.}



\TitleInFrench{Vari\'et\'es horosph\'eriques projectives lisses de groupe de Picard $\Zbb^2$}

\AbstractInFrench{Nous classifions toutes les vari\'et\'es horosph\'eriques projectives lisses de groupe de Picard $\Zbb^2$ et nous donnons une premi\`ere description de leur g\'eom\'etrie \emph{via} le programme des mod\`eles minimaux logarithmiques.}

\acknowledgement{The author is supported by the ANR Project FIBALGA ANR-18-CE40-0003-01.}

\begin{document}

%%%%%%%%%%%%%%%%%%%%%%%%%%%%%%%
% Title page
%%%%%%%%%%%%%%%%%%%%%%%%%%%%%%%

%\removeabove{}
%\removebetween{}
%\removebelow{}

\maketitle

\begin{prelims}

\DisplayAbstractInEnglish

\bigskip

\DisplayKeyWords

\medskip

\DisplayMSCclass

\bigskip

\languagesection{Fran\c{c}ais}

\bigskip

\DisplayTitleInFrench

\medskip

\DisplayAbstractInFrench

\end{prelims}

%%%%%%%%%%%%%%%%%%%%%
% Table of Contents
%%%%%%%%%%%%%%%%%%%%%

\newpage

\setcounter{tocdepth}{2}

\tableofcontents

%%%%%%%%%%%%%%%%%%%%%
% Content begins here
%%%%%%%%%%%%%%%%%%%%%

\section{Introduction}
	
In this paper, varieties are irreducible algebraic varieties over $\Cbb$ and groups are linear algebraic groups over $\Cbb$. And we study varieties that belong to the family of horospherical varieties. Let us first introduce this family.
	
\subsection{About horospherical varieties}
	
Horospherical varieties are ones of the most studied normal $G$-varieties (i.e. varieties endowed with an algebraic action of a group $G$)  including flag varieties (i.e. rational projective homogeneous spaces) and toric varieties.
	
	Recall that toric varieties are normal $T$-varieties where $T$ is a torus and such that, in particular:
	\begin{itemize}
	\item[$\ast$] they have an open $T$-orbit;
	\item[$\ast$] the ring of regular functions of any $T$-stable affine open subset is a multiplicity free $T$-module;
	\item[$\ast$] they are classified in terms of fans.
	\end{itemize}
	There is a natural way to generalized toric varieties to normal $G$-varieties, for any connected reductive algebraic group $G$, with similar properties. And it gives the family of spherical varieties such that in particular:
	\begin{itemize}
	\item[$\ast$] they have an open $G$-orbit;
	\item[$\ast$] the $G$-modules associated to the varieties, for example global sections of $G$-linearized line bundles, are multiplicity free $G$-modules;
	\item[$\ast$] they are classified in terms of colored fans.
	\end{itemize}
	
	The colored fans of spherical varieties depend on data, called spherical data, defined from the open $G$-orbit. The spherical data can differ a lot from a spherical homogeneous space to another. This makes it difficult to study the geometry of all spherical varieties. That is why we often focus on remarkable subfamilies, as the family of horospherical varieties where the open $G$-orbit is a torus fibration over a flag variety. For horospherical varieties, the spherical data is quite simple, so that the combinatorial objects (colored fans,...) are a nice mix of combinatorial objects coming from toric varieties (fans, polytopes,...) and from flag varieties (root systems).

\vskip \baselineskip
We give a (non-exhaustive) list of recent results about horospherical varieties, related to the results and proofs of the paper.
	\begin{itemize}
	\item[$\ast$] There exists a smoothness criterion (really easier to apply than the general one existing for spherical varieties) \cite{these}.
	\item[$\ast$] Fano horospherical varieties are classified in terms of some types of polytopes \cite{Fanohoro}; this result was generalized to spherical varieties \cite{Gagliardi-Hofscheier1}.
	\item[$\ast$] Smooth projective horospherical varieties of Picard group $\Zbb$ are classified \cite{2orbits} and studied in several works: \cite{locallyrigid}, \cite{LiQifeng}, \cite{Kim}, \cite{GPPS},... Note that the only smooth projective toric varieties of Picard group $\Zbb$ are the projective spaces, and that for horospherical varieties we obtained non-homogeneous varieties: 5 families of two-orbit varieties (including two infinite families). 
	\item[$\ast$] The  Minimal Model Program (MMP) \cite{MMPhoro} and the Log MMP \cite{LogMMPhoro} for horospherical varieties can be constructively described in terms of one-parameter families of polytopes.
	\end{itemize}

\subsection{Results of the paper}	

We classify and give a first study of the geometry of smooth projective horospherical varieties of Picard group $\Zbb^2$. For toric varieties, these are only decomposable projective bundles over projective spaces \cite{Kleinschmidt}. But for horospherical varieties, there are many other cases. 

Indeed, in addition to homogeneous spaces, products of two varieties and decomposable projective bundles over projective spaces, we distinguish several other types of such horospherical varieties.  We classify them in this paper, in particular  by studying their Log MMP.

\vskip\baselineskip	
To state as nicely as possible the classification of smooth projective horospherical varieties of Picard group $\Zbb^2$, we extend the notion of simple roots to the groups $\Cbb^*$ and $\{1\}$. We first briefly recall the case of simple groups (in this paper, a simple group has positive semi-simple rank).

If $G$ is a simply connected simple group, we fix a maximal torus contained in a Borel subgroup $B$ of $G$, then it defines a root system and in particular a set of simple roots. To each simple root $\alpha$ are associated a fundamental weight denoted by $\varpi_\alpha$ and a fundamental $G$-module denoted by $V(\varpi_\alpha)$. More generally, if $\chi$ is a dominant weight (a non-negative sum of fundamental weights) we denote by $V(\chi)$ the $G$-module associated to $\chi$: it is the unique irreducible $G$-module that contains a unique $B$-stable line where $B$ acts with weight $\chi$. A non-zero element of the $B$-stable line of $V(\chi)$ is called a highest weight vector (of weight $\chi$) and the stabilizer of the $B$-stable line of $V(\chi)$ is denoted by $P(\chi)$ (it is a parabolic subgroup of $G$ containing $B$).

In this paper, if $G=\Cbb^*$,  we call the identity automorphism of $\Cbb^*$ the {\it simple root} of $G$; we denote it by $\alpha$, and  we set $\varpi_{\alpha}=\alpha$. Then the natural $\Cbb^*$-module $\Cbb$ is denoted by $V(\varpi_\alpha)$ where $\alpha$ is the simple root of $\Cbb^*$. And for any $n\in\Zbb$, $V(n\varpi_\alpha)$ is the $\Cbb^*$-module $\Cbb$ where $\Cbb^*$ acts with weight $n\varpi_\alpha$; in particular, any character of $\Cbb^*$ is dominant.
Moreover, if $G=\{1\}$, we call the trivial morphism from $G$ to $\Cbb^*$ the {\it simple root} of $G$; we denote it by $\alpha$, and we set $\varpi_{\alpha}=0$. In these two cases a highest weight vector is any non-zero vector.

Suppose now that $G$ is a product $G_0\times\cdots\times G_t$ of simply connected simple groups,  $\Cbb^*$ and $\{1\}$. A simple root of $G$ is a simple root of some $G_i$ and it is said to be {\it trivial} if $G_i$ is equal to $\Cbb^*$ or $\{1\}$.
Moreover if $\chi_0,\dots,\chi_t$ are respectively dominant weights of $G_0,\dots,G_t$,  the $G$-module associated to $\chi=\chi_0+\cdots+\chi_t$ is the tensor product $V(\chi_0)\otimes\cdots\otimes V(\chi_t)$ and a highest weight vector of this $G$-module is a decomposable tensor product of highest weight vectors.

\vskip\baselineskip
In Definition~\ref{def:varieties}, we define two types of projective horospherical varieties $\mathbb{X}^1$ and $\mathbb{X}^2$ with Picard group $\Zbb^2$. We describe them explicitly as the closure of some $G$-orbit of a sum of highest weight vectors in the projectivization of a $G$-module, with the convention above. These varieties depend on the group $G$, on a simple root $\beta$,  on a tuple $\underline{\alpha}$ of, possibly trivial, simple roots of $G$ and on a tuple $\underline{a}$ of positive integers.

We can now state the two main results of this paper.
	
\begin{teo}\label{th:main}
	
Let $X$ be a smooth projective horospherical variety with Picard group $\Zbb^2$. Suppose that $X$ is not the product of two varieties. Then $X$ is isomorphic to one of the following horospherical varieties (which we still denote by $X$).

In all cases, $G$ is a product of simply connected simple groups, $\Cbb^*$ and $\{1\}$.

\begin{enumerate}[label=Case (\arabic*): ,start=0]
	\item $G$ is simple and $X$ is a homogeneous variety $G/P$ where $P$ is the intersection of two maximal (proper) parabolic subgroups of $G$ containing the same Borel subgroup.
	
	\item $X$ is one of the variety $\mathbb{X}^1(G,\beta,\underline{\alpha},\underline{a})$ as in Definition~\ref{def:varieties} with one of the restricted conditions (a), (b) or (c) described in Definition~\ref{def:RC1}.
	
	\item $X$ is a variety $\mathbb{X}^2(G,\underline{\alpha},\underline{a})$ as in Definition~\ref{def:varieties} with one of the restricted conditions (a), (b) or (c) described in Definition~\ref{def:RC2}.
\end{enumerate}
\end{teo}

\begin{rem}
\leavevmode
\begin{enumerate}
\item In Theorem~\ref{th:main}, the decomposable projective bundles over projective spaces are some very particular varieties $\mathbb{X}^1(G,\beta,\underline{\alpha},\underline{a})$ with restricted conditions (b) or (c). (See Remark~\ref{rem:decprojbund} for the complete description.)
\item The restricted conditions are useful for two reasons: to get $X$ smooth (and not only locally factorial) and to delete isomorphic cases.
\end{enumerate}
\end{rem}

In Theorem~\ref{th:main}, isomorphisms are not $G$-equivariant isomorphisms. Indeed the acting group is not necessarily the same for both varieties, so we cannot even consider $G$-equivariant isomorphisms. Note that in the paper, if not precised, isomorphisms are not supposed to be $G$-equivariant. Nevertheless, all contractions appearing in the (Log) MMP from a given horospherical $G$-varieties are automatically $G$-equivariant.

	The horospherical varieties given in Theorem~\ref{th:main} are all distinct, i.e., pairwise not isomorphic. This is a consequence of  the following result.
	
	\begin{teo}\label{th:main2}
	Let $X$ be  one of the varieties described in Theorem~\ref{th:main}.  Then  ``the'' Log MMP from $X$ gives the following in each case, respectively with the restricted conditions (a), (b) or (c). 
	
	\begin{enumerate}[label=Case (\arabic*): ,start=0]
	\item There are two Mori fibrations from $X$, respectively onto $Y$ and $Z$, with (general) fibers respectively not isomorphic to $Z$ and $Y$.

	\item 
	\begin{enumerate}
	\item A ``first'' Log MMP consists of a Mori fibration from $X$ to $G/P(\varpi_\beta)$ with general fibers not isomorphic to a projective space (but isomorphic to another homogeneous variety or to a two-orbit variety)  and a ``second'' one consists of a flip from $X$ followed by a fibration.
	\item  A ``first'' Log MMP consists of a Mori fibration from $X$ to $G/P(\varpi_\beta)$ with general fibers isomorphic to a projective space and  a ``second'' one consists of a finite sequence (possibly empty) of flips from $X$ followed by a fibration. Moreover, the fibers of this latter fibration are not all isomorphic.
	\item A ``first'' Log MMP consists of a Mori fibration from $X$ to $G/P(\varpi_\beta)$ with general fibers isomorphic to a projective space and a ``second'' one consists of a finite sequence (possibly empty) of flips from $X$ followed by a divisorial contraction.
	\end{enumerate}
	\item A ``first'' Log MMP consists of a fibration $\psi$ to a two-orbit variety, the general fiber $F_\psi$ of $\psi$ and a ``second'' Log MMP are described as follows.
	\begin{enumerate}
	\item  $F_\psi$ is not isomorphic to a projective space (but isomorphic to another homogeneous variety or to a two-orbit variety) and a ``second'' Log MMP consists of a flip from $X$ followed by a fibration.
	\item  $F_\psi$ is isomorphic to a projective space and a ``second'' Log MMP consists of a finite sequence (not empty) of flips from $X$ followed by a fibration.
	\item  $F_\psi$ is isomorphic to a projective space and a ``second'' Log MMP consists of a finite sequence (may be empty) of flips from $X$ followed by a divisorial contraction.
	\end{enumerate}
	\end{enumerate}
Moreover, in all cases, up to reordering and up to symmetries of Dynkin diagrams, the data $G$ (as a product of simply connected simple groups, $\Cbb^*$ and $\{1\}$), $\beta$, $\underline{\alpha}$ and $\underline{a}$ are invariants of the ``two canonical ways'' to realize the Log MMP from $X$ (and then invariants of $X$).
	\end{teo}
	
\begin{rem}
In the paper (Proposition~\ref{prop:NefCone}), we prove that for any smooth projective horospherical variety $X$ with Picard group $\Zbb^2$, the nef cone of $X$ is generated by the two elements of a basis of $\operatorname{Pic}(X)$, then this gives us two canonical ways to choose the log pair to compute Log MMP from $X$ (see Section~\ref{sec:MMP} for more details). Also, in Cases (1) and (2), one of the ``two canonical'' Log MMP is ``naturally'' defined (see Remark~\ref{rem:debutMMP}) and only consists of a fibration. 
\end{rem}
\begin{rem}
In Case (1b), if the sequence of flips is empty, we get two fibrations from $X$. They could be both onto homogeneous varieties. But one and only one of these fibrations has all its fibers isomorphic to each other (by Proposition~\ref{prop:dimfibres1}, items 3 and 4 with $l=k$). On the contrary, in Case (0), each  fibration has all their fibers isomorphic to each other.
\end{rem}
	
The paper is organized as follows. We first recall in Section~2 the results on horospherical varieties that we use in the paper. Then, in Section~3, we easily describe a first (but not optimal) combinatorial classification, containing many repetitions, and we give a first geometric description of all these latter cases that permits to define the two types of varieties $\mathbb{X}^1$ and $\mathbb{X}^2$. In Section~4,  we first define the restricted conditions used in the statement of Theorem~\ref{th:main}, and we prove the theorem. Then, in Section~5, we prove Theorem~\ref{th:main2}, by studying the Log MMP of all varieties of Theorem~\ref{th:main}.

	\section{Some known results on horospherical varieties}\label{sec:notation}
	
	\subsection{First definitions, first properties of divisors, and smoothness criterion}
	
	In this section, we present the classification of horospherical varieties in terms of colored fans  Then we give the criteria for divisors to be Cartier, globally generated, and ample. And we state the smoothness criterion. All are generalizations of the theory of toric varieties (without colors).
	
\vskip\baselineskip
Let $G$ be a connected reductive group. Fix a maximal torus $T$ and a Borel subgroup $B$ containing $T$. Denote by $U$ the unipotent radical of $B$, by $\mathcal{S}$ the set of simple roots of $(G,B,T)$, by $\mathfrak{X}(T)$ the lattice of characters of $T$ (or $B$) and by $\mathfrak{X}(T)^+\subset \mathfrak{X}(T)$ the monoid of dominant characters.
	
	For any lattice $L$ we denote by $L_\Qbb$ the $\Qbb$-vector space $L\otimes_\Zbb\Qbb$.
	
	\begin{defi}
	A {\it horospherical variety} $X$ is a normal $G$-variety with an open orbit isomorphic to $G/H$ where $H$ is a subgroup of $G$ containing $U$.
	
	Then $G/H$ is a torus fibration over the flag variety $G/P$ where $P$ is the parabolic subgroup of $G$ containing $B$ defined as the normalizer of $H$ in $G$. The dimension of the torus is called the rank of $G/H$ or the {\it rank} of $X$ and is denoted by $n$.
	
	We denote by $M$ the sublattice of $\mathfrak{X}(T)$ consisting of characters of $P$ whose restrictions to $H$ are trivial. Its dual lattice is denoted by $N$. (The lattices $M$ and $N$ are of rank~$n$.)
	
Let $\mathcal{R}$ be the subset of $\mathcal{S}$ consisting of simple roots  that are not simple roots of $P$ (i.e., simple roots associated to fundamental weights some multiples of which are characters of $P$).

For any simple root $\alpha\in \mathcal{R}$, the restriction of the coroot $\alpha^\vee$ to $M$ is a point of $N$, which we denote by $\alpha^\vee_M$. We denote by $\sigma$ the map $\alpha\longmapsto\alpha^\vee_M$ from $\mathcal{R}$ to $N$. 
	\end{defi}

\begin{defi}\label{defi:fan}
\leavevmode
\begin{enumerate}
\item A {\it colored cone} of $N_\Qbb$ is a pair $(\mathcal{C}, \mathcal{F})$
 where $\mathcal{C}$ is a convex cone of $N_\Qbb$ and $\mathcal{F}$ is a subset of $\mathcal{R}$ (called the set of colors of the colored cone), such that
\begin{enumerate}[label= (\roman*)]
\item $\mathcal{C}$ is generated by finitely many elements of $N$ and contains $\{\alpha^\vee_M\,\mid\,\alpha\in\mathcal{F}\}$,
\item $\mathcal{C}$ does not contain any line and
$\mathcal{F}$ does not contain any $\alpha$ such that $\alpha^\vee_M$ is zero.
\end{enumerate}
\item A {\it colored face} of a colored cone $(\mathcal{C}, \mathcal{F})$ is a pair $(\mathcal{C}', \mathcal{F}')$ such that $\mathcal{C}'$ is a face of $\mathcal{C}$ and $\mathcal{F}'$ is the set of $\alpha\in\mathcal{F}$ satisfying $\alpha^\vee_M\in\mathcal{C}'$.

\item A {\it colored fan} is a finite set $\mathbb{F}$ of colored cones such that
\begin{enumerate}[label= (\roman*)]
\item any colored face of a colored cone of $\mathbb{F}$ is in $\mathbb{F}$, and
\item any element of $N_\Qbb$ is in the relative interior of at most one colored cone of $\mathbb{F}$.

\end{enumerate}
\end{enumerate}
\end{defi}

The main result of Luna-Vust Theory of spherical embeddings is the following classification result (see for example \cite{Knop89}).

\begin{teo}[D.~Luna-T.~Vust]\label{th:Luna-Vust}
There is an explicit one-to-one correspondence between $G$-isomorphism classes of horospherical $G$-varieties with open orbit $G/H$ and  colored fans.

Complete $G/H$-embeddings correspond to complete fans, i.e., to fans such that $N_\Qbb$ is the union of the first components of their colored cones.
\end{teo}

If $G=(\Cbb^*)^n$ and $H=\{1\}$, we recover the well-known classification of toric varieties in terms of fans.

If $X$ is a $G/H$-embedding, we denote by $\mathbb{F}_X$ the colored fan of $X$ in $N_\Qbb$ and we denote by $\mathcal{F}_X$ the subset $\cup_{(\mathcal{C},\mathcal{F})\in\mathbb{F}_X}\mathcal{F}$ of $\mathcal{R}$, called the {\it set of colors} of $X$.

\vskip\baselineskip
From now on, $X$ is a complete horospherical variety as above. Let us recall now the characterization of Cartier, $\Qbb$-Cartier, globally generated and ample divisors of horospherical varieties, due to M.~Brion in the more general case of spherical varieties (\cite{briondiv}).

	First, we describe the $B$-stable prime divisors of $X$. We denote by $X_1,\dots,X_m$ the $G$-stable prime divisors of $X$. The valuations of $\Cbb(X)$ defined by the order of zeros and poles along these divisors define primitive elements of $N$, denoted by $x_1,\dots,x_m$ respectively.

	And the $B$-stable but not $G$-stable prime divisors of $X$ are the closures in $X$ of $B$-stable prime divisors of $G/H$, which are the inverse images by the torus fibration $G/H\longrightarrow G/P$ of the 
Schubert divisors of the flag variety $G/P$. The Schubert divisors of $G/P$ can be naturally indexed by the subset of simple roots $\mathcal{R}$. Hence, we denote the $B$-stable but not $G$-stable prime divisors of $X$ by $D_\alpha$ with $\alpha\in \mathcal{R}$ (note that $\sigma(\alpha)$ is the element of $N$ defined by the valuation of $\Cbb(X)$ defined by the zeros and poles along the divisor $D_\alpha$).

\begin{teo}[cf. Section 3.3 in \cite{briondiv}] \label{th:divcrit}  
Any divisor of $X$ is linearly equivalent to a linear combination of $X_1,\dots,X_m$ and $D_\alpha$ with $\alpha\in \mathcal{R}$.
Now, let $D=\sum_{i=1}^m a_i X_i +\sum_{\alpha\in \mathcal{R}} a_\alpha D_\alpha$ be a $\Qbb$-divisor of $X$.
\begin{enumerate}
\item $D$ is $\Qbb$-Cartier if and only if there exists a piecewise linear function $h_D:\,N_\Qbb\longrightarrow\Qbb$, linear on each colored cone of $\mathbb{F}_X$, such that for any $i\in\{1,\dots,m\}$, $h_D(x_i)=a_i$ and for any $\alpha\in\mathcal{F}_X$, $h_D(\alpha^\vee_M)=a_\alpha$.

And $D$ is linearly equivalent to $0$ if and only if $h_D$ is linear on $N_\Qbb$.

Moreover, if $D$ is a divisor, $D$ is Cartier if and only if it is $\Qbb$-Cartier and the linear functions defined as above can be identified with elements of $M$.

\item Suppose that $D$ is $\Qbb$-Cartier. Then $D$ is globally generated (resp. ample) if and only if the piecewise linear function $h_D$ is convex (resp. strictly convex) and for any $\alpha\in \mathcal{R}\backslash\mathcal{F}_X$, we have $h_D(\alpha^\vee_M)\leq a_\alpha$ (resp. $h_D(\alpha^\vee_M)< a_\alpha$).

 Suppose that $D$ is a $\Qbb$-Cartier $\Qbb$-divisor. We define the {\bf pseudo-moment polytope} of $(X,D)$ to be the polytope $\tilde{Q}_D$ in $M_\Qbb$ given by the following inequalities, where $\chi\in M_\Qbb$:  $(h_D)+\chi\geq 0$ and for any $\alpha\in\mathcal{R}\backslash\mathcal{F}_X$, $a_\alpha+\chi(\alpha^\vee_M)\geq 0$.

Let $v^0:=\sum_{\alpha\in \mathcal{R}}a_\alpha\varpi_\alpha$, we  define the {\bf moment polytope} of $(X,D)$ to be the polytope $Q_D:=v^0+\tilde{Q}_D$.

\item Suppose that $D$ is a Cartier divisor. 
Note that the weight of the canonical section of $D$ is $v^0$. Then the $G$-module $H^0(X,D)$ is the direct sum (with multiplicities one) of the irreducible $G$-modules of highest weights $\chi+v^0$ with $\chi$ in $\tilde{Q}_D\cap M$.    

\end{enumerate}
\end{teo}

From now on, a divisor of a horospherical variety is always supposed to be $B$-stable, i.e., of the form $\sum_{i=1}^m a_i X_i +\sum_{\alpha\in \mathcal{R}} a_\alpha D_\alpha$.

	\begin{teo}[cf. Theorem~0.3 \cite{these}]\label{th:veryample}
	Let $X$ be a projective horospherical variety and let $D$ be an ample Cartier divisor of $X$. Suppose that $X$ is smooth. Then $D$ is very ample.
	\end{teo}

	Since $H\supset U$ and the unique $U$-stable lines of irreducible $G$-modules are the lines generated by highest weight vectors, we deduce from Theorems~\ref{th:divcrit} and \ref{th:veryample} the following result. 
	
	\begin{cor}\label{cor:smoothembedd}
	Let $X$ be a smooth projective horospherical variety and let $D$ be an ample Cartier divisor of $X$. Then $X$ is isomorphic to the closure of the $G$-orbit of a sum of highest weight vectors in $\Pbb(\oplus_{\chi\in\tilde{Q}_D\cap M}V(\chi+v^0)).$
	\end{cor}
	
	We should have $V(\chi+v^0)^*$ instead of $V(\chi+v^0)$, but the corollary is still true as stated above, see \cite[Remark~2.13]{MMPhoro}.
	
\vskip\baselineskip
From Theorem~\ref{th:divcrit}, we can also deduce a locally factoriality criterion.
	
	\begin{cor}\label{cor:locfactcrit}
	A horospherical variety $X$ is locally factorial if and only if for any colored cone $(\mathcal{C},\mathcal{F})$ of $\mathbb{F}_X$,  $\mathcal{C}$ is generated by part of a basis of $N$ and the map $\sigma:\,\alpha\longmapsto \alpha^\vee_M$  induces an injective map from $\mathcal{F}$ to this basis.
	
	In particular if $X$ is locally factorial, the Picard number of $X$ is given by the following formula $$\rho_X=m+|\mathcal{R}|-n=(|\mathbb{F}_X(1)|-n) + |\mathcal{R}\backslash\mathcal{F}_X|,$$ where $\mathbb{F}_X(1)$ is the set of edges (one-dimensional colored cones) of $\mathbb{F}_X$.
	\end{cor}
	 
	Note that the characterizations of Cartier, $\Qbb$-Cartier, globally generated and ample divisors can be also applied without the completeness assumption. In particular, Corollary~\ref{cor:locfactcrit} also does not need the completeness assumption.
	
To formulate the smoothness criterion we need to give the following definition.
	
	\begin{defi}(\cite[Def. 2.4]{these})
	Let $\mathcal{R}_1$ and $\mathcal{R}_2$ be two disjoint subsets of $\mathcal{S}$. Let $\Gamma_{\mathcal{R}_1\cup\mathcal{R}_2}$ be the maximal subgraph of the Dynkin diagram of $G$ whose vertices are in $ \mathcal{R}_1\cup\mathcal{R}_2$.
	
	The  {\it pair} $(\mathcal{R}_1,\mathcal{R}_2)$ is said to be {\it smooth} if, for any connected component $\Gamma$ of $\Gamma_{\mathcal{R}_1\cup\mathcal{R}_2}$,
	\begin{enumerate}
	\item there is at most one vertex of $\Gamma$ in $\mathcal{R}_2$  and,
	\item if $\alpha\in\mathcal{R}_2$ is a vertex of $\Gamma$, then $\Gamma$ is of type $A$ or $C$ and $\alpha$ is a short extremal simple root of $\Gamma$.
\end{enumerate}	  
	\end{defi}

\begin{teo}[cf. Theorem~2.6 of \cite{these}]\label{th:smooth}
Let $X$ be a locally factorial horospherical variety. Then $X$ is smooth if and only if for any colored cone $(\mathcal{C},\mathcal{F})$ of $\mathbb{F}_X$, the pair $(\mathcal{S}\backslash\mathcal{R},\mathcal{F})$ is smooth.
\end{teo}
	
	\begin{cor}[cf. Proposition~2.17 of \cite{these}]\label{cor:smooth}
	Let $X$ be a smooth horospherical variety. Any $G$-stable subvariety of $X$ is a smooth horospherical variety.
	\end{cor}
	
	\begin{rem}
	If $X$ is a toric variety, Theorem~\ref{th:smooth} is trivial because locally factorial toric varieties are smooth, or because for any colored cone $(\mathcal{C},\mathcal{F})$ of $\mathbb{F}_X$, the pair $(\mathcal{S}\backslash\mathcal{R},\mathcal{F})$ is necessarily $(\emptyset,\emptyset)$ (indeed the root system or the Dynkin diagram of a torus is empty).
	\end{rem}
	\subsection{Log MMP via moment polytopes}\label{sec:recallMMP}
	
	The MMP \cite{MMPhoro} and Log MMP \cite{LogMMPhoro} of horospherical varieties can be completely computed and described by studying one-parameter families of polytopes. In this subsection, we recall the main results of this theory, as briefly as we can, in order to use them in Section~\ref{sec:MMP}.
	
	From the previous section, to any horospherical variety $X$, are associated a parabolic subgroup $P$ and a sublattice $M$ of $\mathfrak{X}(P)$; and moreover, any ample $B$-stable $\Qbb$-Cartier $\Qbb$-divisor $D$ defines a pseudo-moment polytope $\tilde{Q}$ and a moment polytope $Q$. In fact, the map $(X,D)\longmapsto (P,M,Q,\tilde{Q})$ classifies polarized projective horospherical varieties in terms of quadruples $(P,M,Q,\tilde{Q})$. 
	
	\begin{defi} A {\it quadruple} $(P,M,Q,\tilde{Q})$ is called  {\it admissible}  if it satisfies the following:
\begin{itemize} 
\item $P$ is a parabolic subgroup of $G$ containing $B$, $M$ is a sublattice of $\mathfrak{X}(P)$, $Q$ is a polytope of $\mathfrak{X}(P)_\Qbb$ included in $\mathfrak{X}(P)^+_\Qbb$ and $\tilde{Q}$ is a polytope of $M_\Qbb$;
\item there exists (a unique) $v^0\in\mathfrak{X}(P)_\Qbb$ such that $Q=v^0+\tilde{Q}$;
\item the polytope $\tilde{Q}$ is of maximal dimension in $M_\Qbb$ (i.e., its interior in $M_\Qbb$ is not empty);
\item the polytope $Q$ intersects the interior of $\mathfrak{X}(P)^+_\Qbb$.
\end{itemize}	
\end{defi}

\begin{ex}\label{ex:quadruple1}
Suppose that $\mathfrak{X}(P)=\Zbb\varpi_1\oplus\Zbb\varpi_2$ and $M=\Zbb\varpi_2$, then $Q$ and $\tilde{Q}$ are vertical segments of the same length, $\tilde{Q}$ is in $\Qbb\varpi_2$ and $Q$ is in $\Qbb_{\geq 0}\varpi_1\oplus \Qbb_{\geq 0}\varpi_2$ (but not in $\Qbb\varpi_2$). In Figure~\ref{fig:quadruple1}, we draw three possible pairs $(Q,\tilde{Q})$ to get three admissible quadruples $(P,M,Q,\tilde{Q})$ respectively corresponding to polarized varieties $(X,D_1)$, $(X,D_2)$ and $(X',D')$, with $D_1\neq D_2$ and $X\not\simeq X'$.

\begin{figure}

\begin{center}
\begin{tikzpicture}[scale=0.5]

\draw[thick,color=gray ] (0,0) -- (4,0);
\draw[thick,color=gray ] (0,0) -- (0,4);
\draw[densely dotted,color=gray] (0,0) -- (-4,0);
\draw[densely dotted,color=gray] (0,0) -- (0,-4);
\node[color=gray ] at (-0.5,-0.5) {0};
\node[color=gray ] at (1,-0.5) {$\varpi_1$};
\node[color=gray ] at (-0.7,1) {$\varpi_2$};
\node[color=gray ] at (1,0) {$\bullet$};
\node[color=gray ] at (0,1) {$\bullet$};
\draw[very thick] (2,2) -- (2,4);
\node at (2,2) {$\bullet$};
\node at (2,4) {$\bullet$};
\node at (-0.5,-1.5) {$\tilde{Q}$};
\node at (2.5,3) {$Q$};
\draw[very thick] (0,-2.5) -- (0,-0.5);
\node at (0,-2.5) {$\bullet$};
\node at (0,-0.5) {$\bullet$};
\end{tikzpicture}
\begin{tikzpicture}[scale=0.5]

\draw[thick,color=gray ] (0,0) -- (4,0);
\draw[thick,color=gray ] (0,0) -- (0,4);
\draw[densely dotted,color=gray] (0,0) -- (-4,0);
\draw[densely dotted,color=gray] (0,0) -- (0,-4);
\node[color=gray ] at (-0.5,-0.5) {0};
\node[color=gray ] at (1,-0.5) {$\varpi_1$};
\node[color=gray ] at (-0.7,1) {$\varpi_2$};
\node[color=gray ] at (1,0) {$\bullet$};
\node[color=gray ] at (0,1) {$\bullet$};
\draw[very thick] (1,1) -- (1,2);
\node at (1,1) {$\bullet$};
\node at (1,2) {$\bullet$};
\node at (-0.5,2.5) {$\tilde{Q}$};
\node at (1.5,1.5) {$Q$};
\draw[very thick] (0,2) -- (0,3);
\node at (0,2) {$\bullet$};
\node at (0,3) {$\bullet$};
\end{tikzpicture}
\begin{tikzpicture}[scale=0.5]

\draw[thick,color=gray ] (0,0) -- (4,0);
\draw[thick,color=gray ] (0,0) -- (0,4);
\draw[densely dotted,color=gray] (0,0) -- (-4,0);
\draw[densely dotted,color=gray] (0,0) -- (0,-4);
\node[color=gray ] at (-0.5,-0.5) {0};
\node[color=gray ] at (1,-0.5) {$\varpi_1$};
\node[color=gray ] at (-0.7,1) {$\varpi_2$};
\node[color=gray ] at (1,0) {$\bullet$};
\node[color=gray ] at (0,1) {$\bullet$};
\draw[very thick] (2,0) -- (2,2);
\node at (2,0) {$\bullet$};
\node at (2,2) {$\bullet$};
\node at (-0.5,-1.5) {$\tilde{Q'}$};
\node at (2.5,1) {$Q'$};
\draw[very thick] (0,-2.5) -- (0,-0.5);
\node at (0,-2.5) {$\bullet$};
\node at (0,-0.5) {$\bullet$};
\end{tikzpicture}
\caption{Some (pseudo-)moment polytopes}
\label{fig:quadruple1}
\end{center}
\end{figure}
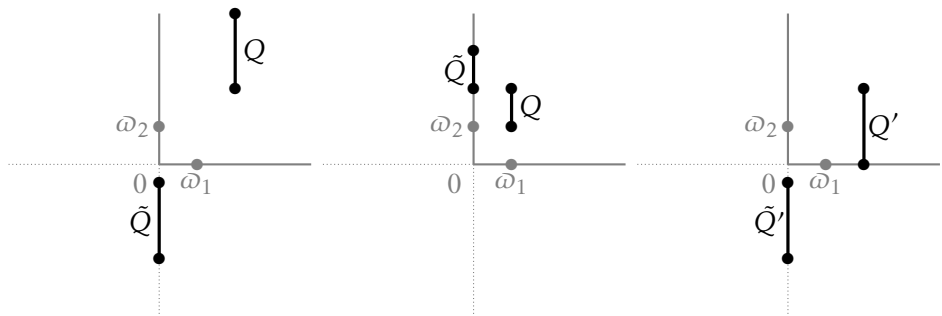
\end{ex}

\begin{prop}[Corollary 2.10 of \cite{LogMMPhoro} together with Propositions 2.10 and 2.11 \cite{MMPhoro}]\label{prop:classifpoly}
\leavevmode
\begin{enumerate}
\item The map $(X,D)\longmapsto (P,M,Q,\tilde{Q})$ is a bijection from the set of isomorphism classes of polarized projective horospherical varieties to the set of  admissible quadruples. 
 \item It induces a bijection between the set of $G$-orbits in $X$ and the set of non-empty faces of $Q$ (or $\tilde{Q}$), preserving the natural orders of both sets. Also, the $G$-orbit in $X$ associated to a non-empty face $F=v^0+\tilde{F}$ of $Q$ is isomorphic to a horospherical homogeneous space corresponding to $(P_F,M_F)$ where $P_F$ is the minimal parabolic subgroup of $G$ containing $P$ and $M_F$ is the  maximal sublattice of $M$ such that $(P_F,M_F,F,\tilde{F})$ is an admissible quadruple. Moreover $(P_F,M_F,F,\tilde{F})$ is the quadruple associated to the (horospherical) closure in $X$ of the $G$-orbit associated to $F$ (polarized by some $D_F$ we do not need to explicit here).
\end{enumerate}
\end{prop}

\begin{ex}
Consider the moment polytopes of Example~\ref{ex:quadruple1}. And suppose that $D_1$, $D_2$ and $D'$ are very ample (otherwise it would be enough to consider multiples of the divisors and of the polytopes).

Then $X$ is the closure of  $G\cdot[v_{2\varpi_1+2\varpi_2}+v_{2\varpi_1+3\varpi_2}+v_{2\varpi_1+4\varpi_2}]$ in
$$\Pbb(V(2\varpi_1+2\varpi_2)\oplus V(2\varpi_1+3\varpi_2)\oplus V(2\varpi_1+4\varpi_2))$$
but also the closure of $G\cdot[v_{\varpi_1+\varpi_2}+v_{\varpi_1+1\varpi_2}]$ in $\Pbb(V(\varpi_1+\varpi_2)\oplus V(\varpi_1+2\varpi_2))$. In the first case for example, one can easily check that there are exactly two (closed) $G$-orbits in addition to the open one in $X$; moreover, they are $G\cdot[v_{2\varpi_1+2\varpi_2}]\simeq G/(P(\varpi_1)\cap P(\varpi_2))$ and $G\cdot[v_{2\varpi_1+4\varpi_2}]\simeq G/(P(\varpi_1)\cap P(\varpi_2))$, and they correspond to the two vertices of the segment $Q$. Here, for both closed $G$-orbits, $P_F=P$ and $M_F=\{0\}$.

Similarly, $X'$ is the closure of $G\cdot[v_{2\varpi_1}+v_{2\varpi_1+\varpi_2}+v_{2\varpi_1+2\varpi_2}]$ in $\Pbb(V(2\varpi_1)\oplus V(2\varpi_1+\varpi_2)\oplus V(2\varpi_1+2\varpi_2))$. There are exactly two (closed) $G$-orbits in addition to the open one in $X'$, that is $G\cdot[v_{2\varpi_1}]\simeq G/P(\varpi_1)$ and $G\cdot[v_{2\varpi_1+2\varpi_2}]\simeq G/(P(\varpi_1)\cap P(\varpi_2))$. Here,  we still have $M_F=\{0\}$ for both closed $G$-orbits and $P_F=P$ for the second closed $G$-orbit, but $P_F\neq P$ for the first one ($\mathfrak{X}(P_F)=\Zbb\varpi_1$).
\end{ex}

From Proposition~\ref{prop:classifpoly}, we easily get the following result.
	\begin{cor}\label{cor:classifpoly}
	Let $(X,D)$ be a polarized projective horospherical variety and $(P,M,Q,\tilde{Q})$  be the corresponding admissible quadruple. Let $F$ be a non-empty face of $Q$ (or $\tilde{Q}$) and $\Omega$ be the corresponding $G$-orbit in $X$.
	 Then $$\dim(\Omega)=\dim(G/P_F)+\operatorname{rank}(M_F)=\dim(G/P_F)+\dim(F).$$
	\end{cor}

We can also describe $G$-equivariant morphisms between horospherical $G$-varieties, in terms of moment polytopes \cite[2.4]{MMPhoro}. We summarize, very briefly, this description here. 

Without loss of generality, we can reduce to dominant $G$-equivariant morphisms, i.e. $G$-equivariant morphisms from a $G/H$-embbedding to a $G/H'$-embbedding where $H\subset H'$, i.e., $G$-equivariant morphisms  that extend the projection $G/H\longrightarrow G/H'$. In that case, we have $P\subset P'$ and $M'\subset M$. We keep the same notations as above for the data associated to $G/H$ and we use the same notations with prime for the data associated to $G/H'$.

Let $X$ be a projective $G/H$-embedding corresponding to an admissible quadruple $(P,M,Q,\tilde{Q})$ and let $X'$ be a projective $G/H'$-embedding corresponding to an admissible quadruple $(P',M',Q',\tilde{Q'})$.  Then the projection $G/H\longrightarrow G/H'$ extends to a $G$-equivariant morphism from $X$ to  $X'$ if and only if for any non-empty face $F$ of $Q$, the set of facets (or the corresponding halfspaces in $M_\Qbb$) and the set of walls of $\mathfrak{X}(P)_\Qbb^+$ that contain $F$ define naturally a non-empty face $F'$ of $Q'$. Moreover in that case the $G$-orbit of $X$ corresponding to $F$ is sent to the $G$-orbit of $X'$ corresponding to $F'$.

\begin{ex}
Consider the varieties $X$ and $X'$ of Example~\ref{ex:quadruple1}. Each vertex of $Q$, which is a facet, naturally correspond to a vertex of $Q'$. But, the vertex $2\varpi_1$ of $Q'$ is contained in a wall of $\mathfrak{X}(P)_\Qbb^+$ and will correspond to the empty face of $Q$. Then, here, there exists a $G$-equivariant morphism $\phi$ from $X$ to $X'$ but there is no such morphism from $X'$ to $X$. Moreover, $\phi$ is an isomorphism outside one closed $G$-orbit where $\phi$ is the projection $G/(P(\varpi_1)\cap P(\varpi_2))\longrightarrow G/P(\varpi_2)$.

To complete this example, consider some $G/H$ of rank~2 such that $P$ has a unique fundamental weight $\varpi$. In Figure~\ref{fig:quadruple2} we draw  3 moments polytopes of $G/H$ and another moment polytope of a horospherical homogeneous space $G/H'$ of rank~1 with $P'=G$ (in fact $G/H'\simeq \Cbb^*$ and the segment corresponds to the variety $\Pbb^1$).  We also draw all $G$-equivariant morphisms between the corresponding varieties. Note that this picture is similar to Figure~\ref{figure6} with moment polytopes instead of pseudo-moment polytopes.

We also emphasis some vertices and some edges to illustrate images of $G$-orbits. More precisely, if we focus at the $G$-orbits distinguished by a $\bullet$, $\phi_0$ restricts to the projection $G/P(\varpi)\longrightarrow \mbox{pt}$. If we focus at the $G$-orbits distinguished by a non-dashed rectangle, $\phi_0^+$ restricts to the fibration $\Pbb^1\longrightarrow \mbox{pt}$ and $\phi_1$ restricts to the identity morphism $\Pbb^1\longrightarrow\Pbb^1$. If we focus at the $G$-orbits distinguished by a dashed rectangle, $\phi_0$ and $\phi_0^+$ restrict to identity morphisms and $\phi_1$ restricts to a fibration to a point.

 \begin{figure}
\begin{center}

\begin{tikzpicture}

\node [rectangle] (a) at (0,0) {
    \begin{tikzpicture}[scale=0.5]

\draw[very thick] (2,2) -- (2,3) -- (3,2) -- (2,2);
\node at (2,2) {$\bullet$};
\draw[dashed] (1.8,3.2) -- (2.2,3.2) -- (2.2,1.8) --(1.8,1.8) --  (1.8,3.2);
\draw[densely dotted] (0,3) -- (3,0) ;
\node[color=gray] at (3.8,0.5) {$\varpi=0$};
    \end{tikzpicture}
};

\node [rectangle] (b) at (2,-3.25) {
    \begin{tikzpicture}[scale=0.5]
     
\draw[very thick] (2,2) -- (2,3) -- (3,2) -- (2,2);
\node at (2,2) {$\bullet$};
\node at (2,2) {$\square$};
\draw[dashed] (1.8,3.2) -- (2.2,3.2) -- (2.2,1.8) --(1.8,1.8) --  (1.8,3.2);
\draw[densely dotted] (0.5,3.5) -- (3.5,0.5) ;
\node[color=gray] at (4.8,0.5) {$\varpi=0$};
    \end{tikzpicture}
};

\node [rectangle] (c) at (4,0) {
    \begin{tikzpicture}[scale=0.5]
     
\draw[very thick] (2,2.5) -- (2,3.5) -- (3.5,2) -- (2.5,2) -- (2,2.5);
\draw[ thin] (1.8,2.5) -- (2,2.7) -- (2.7,2) --(2.5,1.8) --  (1.8,2.5);
\draw[dashed] (1.8,3.7) -- (2.2,3.7) -- (2.2,2.3) --(1.8,2.3) --  (1.8,3.7);

\draw[densely dotted] (0.5,4) -- (4,0.5) ;
\node[color=gray] at (4.8,1) {$\varpi=0$};
    \end{tikzpicture}
};
\node [rectangle] (d) at (6,-3.25) {
    \begin{tikzpicture}[scale=0.5]
    
\draw[very thick]  (2,3) -- (3,2) ;
\draw[ thin] (1.8,3) -- (2,3.2) -- (3.2,2) --(3,1.8) --  (1.8,3);
\draw[dashed] (1.7,2.7) -- (2.3,2.7) -- (2.3,3.3) --(1.7,3.3) --  (1.7,2.7);
\draw[densely dotted] (1,4) -- (4,1) ;
\node[color=gray] at (5.2,1) {$\varpi=0$};
    \end{tikzpicture}
};

\draw (a) [->]  to node[above right] {$\phi_0$} (b);
\draw (c) [->]  to node[above left] {$\phi_0^+$} (b);
\draw (c) [->]  to node[above right] {$\phi_1$} (d);
\end{tikzpicture}

\caption{Moment polytopes and $G$-equivariant morphisms}

\label{fig:quadruple2}
\end{center}
\end{figure}
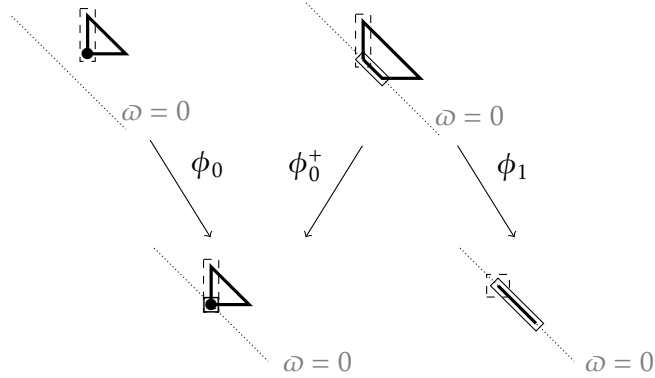
\end{ex}
	
	Now we can state the description of the Log MMP for horospherical varieties in terms of moment polytopes.

 First we fix a basis of $M$ (and consider the dual basis for $N$). Also we choose an order in the set $\{x_1,\dots,x_m\}\cup\{\alpha^\vee_M\,\mid\,\alpha\in\mathcal{R}\}$. 
Then we define a matrix $\mathcal{A}$ of size $(m+|\mathcal{R}|)\times n$ whose rows are the coordinates of the vectors of $\{x_1,\dots,x_m\}\cup\{\alpha^\vee_M\,\mid\,\alpha\in\mathcal{R}\}$ in the chosen basis. 

\begin{teo}[cf. Theorem~1.3 and Section~3 in \cite{LogMMPhoro}]\label{th:recallMMP}
	
	 Let $X$ be a $\Qbb$-factorial projective horospherical variety and let $\Delta$ be a $B$-stable $\Qbb$-divisor of $X$. 
	Then for any (general) choice of an ample $B$-stable $\Qbb$-Cartier $\Qbb$-divisor $D$ of $X$, a Log MMP from the pair $(X,\Delta)$ is described by the following one-parameter families of polytopes
	
	$$\tilde{Q}^\epsilon:=\{x\in M_\Qbb\,\mid\, \mathcal{A}x\geq \mathcal{B}+\epsilon \mathcal{C}\}\,\mbox{ and }\,Q^\epsilon:= v^\epsilon +\tilde{Q}^\epsilon$$ where 
	 $\mathcal{B}$, $\mathcal{C}$ and $v^\epsilon=v^0+\epsilon v^1$ are such that, for any $\epsilon\geq 0$ small enough, $\tilde{Q}^\epsilon$ and $Q^\epsilon$ are respectively the pseudo-moment and moment polytope of $(X,D+\epsilon(K_X+\Delta))$. 
	\end{teo}

Note that the matrices $\mathcal{A}$, $\mathcal{B}$ and $\mathcal{C}$ can be easily computed. Indeed,  $\mathcal{A}$ is given by the primitive elements of the rays of the colored fan of $X$ and the images of the colors of $G/H$; the coefficients of $\mathcal{B}$ are the opposites of the coefficients of $D$; and the coefficients of $\mathcal{C}$ are the opposites of the coefficients of $K_X+\Delta$. Also, the coefficients of $v^0$ and $v^1$ correspond to the coefficients of the $D_\alpha$'s for $D$ and $K_X+\Delta$  respectively.

\vskip\baselineskip
We can rewrite the conclusion of Theorem~\ref{th:recallMMP} more precisely as the existence of rational numbers \begin{multline*}0:=\epsilon_{0,0}<\cdots<\epsilon_{0,k_0}<\epsilon_{0,k_0+1}=\epsilon_{1,0}<\cdots\\ \cdots<\epsilon_{1,k_1}<\epsilon_{1,k_1+1}=\epsilon_{2,0}<\cdots<\epsilon_{p,k_p}<\epsilon_{p,k_p+1}=\epsilon_{max}
	\end{multline*} 
	(with $p\geq 1$, and for any $i\in\{0,\dots,p\}$, $k_i\geq 0$) such that, $(P,M,Q^\epsilon,\tilde{Q}^\epsilon)$ is an admissible quadruple if and only if $\epsilon\in[0,\epsilon_{max}[$, and for $\epsilon,\,\eta\in[0,\epsilon_{max}[$ the following three assertions are equivalent:
	\begin{itemize}
	\item $X^\epsilon$ is isomorphic to $X^\eta$ (where $X^\epsilon$ and $X^\eta$ are the varieties associated to the admissible quadruples $(P,M,Q^\epsilon,\tilde{Q}^\epsilon)$ and $(P,M,Q^\eta,\tilde{Q}^\eta)$ respectively);
	\item the faces of $Q^\epsilon$ (or  $\tilde{Q}^\epsilon$) and $Q^\eta$ (or  $\tilde{Q}^\eta$) are ``the same'', in the following sense: up to deleting inequalities corresponding to some $x_j$ with $j\in\{1,\dots,m\}$ but without changing $\tilde{Q}^\epsilon$ and $\tilde{Q}^\eta$, we have that for any set $I$ of rows, the face of $\tilde{Q}^\epsilon$ corresponding to $I$ (defined by replacing inequalities by equalities for the rows in $I$) is non empty if and only the  face of $\tilde{Q}^\eta$ corresponding to $I$ is non empty;
	\item there exists $i\in\{0,\dots,p\}$ such that $\epsilon$ and $\eta$ are both in $[\epsilon_{i,0},\epsilon_{i,1}[$, or both in $]\epsilon_{i,k},\epsilon_{i,k+1}[$ with $k\in\{1,\dots,k_i\}$, or both equal to $\epsilon_{i,k}$ with $k\in\{1,\dots,k_i\}$.
	\end{itemize}

	Moreover, for any $i\in\{0,\dots,p\}$ and $k\in\{1,\dots,k_i\}$ there are morphisms from $X^\epsilon$ to $X^{\epsilon_{i,k}}$ with $\epsilon<\epsilon_{i,k}$ big enough and  $\epsilon>\epsilon_{i,k}$ small enough, defining flips. For any $i\in\{1,\dots,p\}$, there are morphisms from $X^\epsilon$ to $X^{\epsilon_{i,0}}$ with $\epsilon<\epsilon_{i,0}$ big enough, defining divisorial contractions. Actually, divisorial contractions appear exactly when an inequality corresponding to some $x_j$ with $j\in\{1,\dots,m\}$ becomes superfluous to define  $\tilde{Q}^\epsilon$.
	
\vskip\baselineskip
Also, there exists $P'$ and $M'$ such that $(P',M',Q^{\epsilon_{max}},\tilde{Q}^{\epsilon_{max}})$ is an admissible quadruple associated to a variety $X^{\epsilon_{max}}$ and such that there is a fibration from $X^\epsilon$ to $X^{\epsilon_{max}}$ with $\epsilon<\epsilon_{max}$ big enough. Moreover, the general fiber of this fibration is a horospherical variety and can be described. 
	
	In fact all fibers could be described with the following strategy: consider a $G$-orbit $G/H''$ of $X^{\epsilon_{max}}$ and list all $G$-orbits of $X^{\epsilon}$ with $\epsilon<\epsilon_{max}$ big enough that are sent to $G/H''$ by the fibration, then if there is a unique biggest such $G$-orbit $\Omega$, the fibers over $G/H''$ are isomorphic to the closure of $L''\cdot v$ where $L''$ is a Levi subgroup of $H''$ and $v$ is the projectivization of a sum of highest weight vectors in $\Omega$. Note that in this paper, there will always be such a biggest $G$-orbit.
	
\vskip\baselineskip
All  morphisms above are $G$-equivariant and the image of any $G$-orbit can be described as follows. To a face of $Q^\epsilon$ (or $\tilde{Q}^\epsilon$) we can associate the maximal set of rows for which equality holds for any element $x$ of the face (in the inequalities $Ax\geq B+\epsilon C$). And similarly to a set of rows we can also naturally associate a face of $Q^\epsilon$ (may be empty). For any $\epsilon$ and $\epsilon_{i,k}$ as above, for any face $F^\epsilon$ of $\tilde{Q}^\epsilon$, we construct a face of $\tilde{Q}^{\epsilon_{i,k}}$ by taking the maximal set of rows associated to $F^\epsilon$ and then the face $F^{\epsilon_{i,k}}$ associated to these rows. Then, since there is a morphism $\phi$ from $X^\epsilon$ to $X^{\epsilon_{i,k}}$, the non-empty face $F^{\epsilon_{i,k}}$ corresponds to the $G$-orbit image by $\phi$ of the $G$-orbit corresponding to $F^\epsilon$.

\vskip\baselineskip	
Several examples illustrating Theorem~\ref{th:recallMMP}, in rank~2, are given in Sections~\ref{LogMMP1} and~\ref{LogMMP2}.
	
\section{First combinatorial classification and first geometric description}
	
\subsection{Reduction to three cases}
	
	In this section, we only use Luna-Vust theory and Corollary~\ref{cor:locfactcrit} to reduce to the three main cases of Theorem~\ref{th:main}.
	
	\begin{lem}\label{lem:FirstReduction}
	Let $X$ be a smooth projective horospherical variety with Picard group $\Zbb^2$. Then one the three following cases occurs (with notation of Section~\ref{sec:notation}).
	\begin{enumerate}[label=Case (\arabic*): ,start=0]
	\item $n=0$, $|\mathcal{R}|=2$, $\mathcal{F}_X=\emptyset$,  and $X=G/P$.
	\item $n\geq 1$, $\mathcal{R}=\mathcal{F}_X\sqcup\{\beta\}$, there exist a basis $(e_1,\dots,e_n)$ of $N$ and $n$ integers $0\leq a_1\leq\cdots\leq a_n$ such that $\sigma$ induces an injective map $\tilde\sigma$ from $\mathcal{F}_X$ to $\{e_1,\dots,e_n,e_{0}:=-e_1-\cdots -e_n\}$, $\sigma(\beta)=a_1e_1+\cdots+a_ne_n$ and $$\mathbb{F}_X=\{(\mathcal{C}_I,\mathcal{F}_I)\,\mid\, I\subsetneq\{0,\dots,n\}\}$$ where $\mathcal{C}_I$ is the cone generated by the $e_i$'s with $i\in I$, and $\mathcal{F}_I=\tilde\sigma^{-1}(\{e_i\,\mid\,i\in I\})$.
	\item $n\geq 2$, $\mathcal{R}=\mathcal{F}_X$, there exist integers $r\geq 1$, $s\geq 1$, $0\leq a_1\leq\cdots\leq a_r$ and a basis $(u_1,\dots,u_r,v_1,\dots,v_s)$ of $N$ such that $\sigma$ induces an injective map  $\tilde\sigma$ from $\mathcal{F}_X=\mathcal{R}$ to $\{u_0,\dots,u_r, v_1, \dots, v_{s+1}\},$ with $u_0:=-u_1-\cdots -u_r$ and $v_{s+1}:=a_1u_1+\cdots +a_ru_r-v_1-\cdots -v_s$, and $$\mathbb{F}_X=\{(\mathcal{C}_{I,J},\mathcal{F}_{I,J})\,\mid\, I\subsetneq\{0,\dots,r\}\mbox{ and } J\subsetneq\{1,\dots,s+1\}\}$$ where $\mathcal{F}_{I,J}=\tilde\sigma^{-1}(\{u_i\,\mid\,i\in I\}\cup\{v_j\,\mid\,j\in J\} )$ and where $\mathcal{C}_{I,J}$ is the cone generated by the $u_i$'s with $i\in I$ and the $v_j$'s with $j\in J$.
\end{enumerate}	 
	\end{lem}
	
	\begin{rem}
	If $X$ is a toric variety, $\mathcal{R}=\emptyset$ then we are necessarily in Case~(2), and the lemma is already known \cite[Theorem~1]{Kleinschmidt}, and $X$ is the decomposable projective bundle $\Pbb(\mathcal{O}\oplus\mathcal{O}(a_1)\oplus\cdots\oplus\mathcal{O}(a_r)$ over $\Pbb^s$.
	\end{rem}
	
	\begin{proof}
	By Corollary~\ref{cor:locfactcrit}, the map $\sigma$ induces an injective map from $\mathcal{F}_X$ to $\mathbb{F}_X(1)$ and the Picard number of $X$ is $\rho_X=(|\mathbb{F}_X(1)|-n) + |\mathcal{R}\backslash\mathcal{F}_X|$. But, since $X$ and then $\mathbb{F}_X$ are complete, $|\mathbb{F}_X(1)|-n\geq 0$ with equality if and only if $n=0$. (And $|\mathcal{R}\backslash\mathcal{F}_X|\geq 0$.) Thus, since $\rho_X=2$ we distinguish three distinct cases:
	\begin{enumerate}[label=Case (\arabic*): ,start=0]
	\item $n=0$ and $|\mathcal{R}\backslash\mathcal{F}_X|=2$;
	\item $|\mathbb{F}_X(1)|=n+1$ and $|\mathcal{R}\backslash\mathcal{F}_X|=1$;
	\item $|\mathbb{F}_X(1)|=n+2$ and $|\mathcal{R}\backslash\mathcal{F}_X|=0$.
	\end{enumerate}	
	We now detail each case.
	\begin{enumerate}[label=Case (\arabic*): ,start=0]
	\item In the case where $n=0$, $X$ is the complete homogeneous variety $G/P$ (and $\mathcal{F}_X=\emptyset$). And then $|\mathcal{R}|=2$.
	\item 
	
	Consider the fan $\tilde{\mathbb{F}}:=\{\mathcal{C}\,\mid\,(\mathcal{C},\mathcal{F})\in\mathbb{F}_X\}$
	 associated to the colored fan $\mathbb{F}_X$ (in fact it is the fan of the toric fiber $Y$ of the toroidal variety $\tilde{X}:=G\times^P Y$ obtained from $X$ by erasing all colors of $X$). Since $X$ is locally factorial, the fan $\tilde{\mathbb{F}}$ is the fan of a smooth toric variety of Picard number~1 (because $|\tilde{\mathbb{F}}_X(1)|=n+1$). Then it is well-known that such a fan is the fan of the projective space $\Pbb^n$. In particular, there exists a basis $(e_1,\dots,e_n)$ of $N$ such that $\tilde{\mathbb{F}}=\{\mathcal{C}_I\,\mid\, I\subsetneq\{0,\dots,n\}\}$ where $e_{0}:=-e_1-\cdots -e_n$ and $\mathcal{C}_I$ is the cone generated by the $e_i$ with $i\in I$.
	
	Denote by $\beta$ the unique element of $\mathcal{R}\backslash\mathcal{F}_X$. Then, up to reordering the $e_i$'s (for $i\in\{0,\dots,n\}$), we can suppose that $\sigma(\beta)$ is in $\mathcal{C}_{\{1,\dots,n\}}$ and equals $a_1e_1+\cdots+a_ne_n$ with $0\leq a_1\leq\cdots\leq a_n$.
	
	\item As above, consider the fan $\tilde{\mathbb{F}}$. Since $X$ is locally factorial, it is the fan of a smooth toric variety of Picard number~2 (because $|\tilde{\mathbb{F}}_X(1)|=n+2$). Then, by \cite[Theorem~1]{Kleinschmidt}, there exist integers $r\geq 1$, $s\geq 1$, $0\leq a_1\leq\cdots\leq a_r$ and a basis $(u_1,\dots,u_r,v_1,\dots,v_s)$ of $N$ such that $\tilde{\mathbb{F}}=\{\mathcal{C}_{I,J}\,\mid\, I\subsetneq\{0,\dots,r\}\mbox{ and } J\subsetneq\{1,\dots,s+1\}\}$, where $u_0:=-u_1-\cdots -u_r$, $v_{s+1}:=a_1u_1+\cdots +a_ru_r-v_1-\cdots -v_s$ and $\mathcal{C}_{I,J}$ is the cone generated by the $u_i$'s with $i\in I$ and the $v_j$'s with $j\in J$.
	\end{enumerate}	
	
	We conclude by using the following facts: for any $\alpha\in\mathcal{F}_X$ and for any $(\mathcal{C},\mathcal{F})\in\mathbb{F}_X$, we have $\alpha\in\mathcal{F}$ if and only if $\sigma(\alpha)\in\mathcal{C}$;  and for any $\alpha\in\mathcal{F}_X$, $\sigma(\alpha)$ is the primitive element of an edge of $\mathbb{F}_X$ (using again Corollary~\ref{cor:locfactcrit}).
\end{proof}
	
	\begin{rem}\label{rem:debutMMP}
	In section~\ref{sec:MMP}, we will use the MMP or the Log MMP to study and compare geometrically all these varieties $X$. We can already describe some Mori fibrations from these varieties, by using the following description of $G$-equivariant morphisms between horospherical varieties in terms of colored fans (\cite{Knop89}). Let $G/H$ and $G/H'$ be two horospherical homogeneous spaces with $H\subset H'$, and denote by $\pi:\,G/H\longrightarrow G/H'$ the projection.  We keep the same notations as before for the data associated to $G/H$ and we use the same notations with prime for the data associated to $G/H'$. In particular, we have $M'\subset M$, $P\subset P'$ and $\mathcal{R}'\subset\mathcal{R}$. By duality, we also have a projection $\pi_*:\,N_\Qbb\longrightarrow N'_\Qbb$. Let $X$ be a $G/H$-embedding with colored fan $\mathbb{F}_X$ and let $X'$ be a $G/H'$-embedding with colored fan $\mathbb{F}_{X'}$. Then the morphism $\pi$ extends to a $G$-equivariant morphism from $X$ to $X'$ if and only if for any colored cone $(\mathcal{C},\mathcal{F})\in\mathbb{F}_X$, there exists a colored cone $(\mathcal{C}',\mathcal{F}')\in\mathbb{F}_{X'}$ such that $\pi_*(\mathcal{C})\subset \mathcal{C}'$ and $\mathcal{F}\cap\mathcal{R}'\subset\mathcal{F}'$. 
	
	\begin{enumerate}[label=Case (\arabic*): ,start=0]
	\item If $X$ is a complete homogeneous variety $G/P$ of Picard group $\Zbb^2$, then the MMP gives two Mori fibrations from $X$ to the complete homogeneous varieties $G/P_1$ and $G/P_2$ of Picard group $\Zbb$, where $P_1$ and $P_2$ are the maximal proper parabolic subgroups of $G$ containing $B$ such that $P=P_1\cap P_2$. Note moreover that $G/P$ is a product if and only if $\Aut^0(G/P)$ is not simple.
	\item Let $G/H'=X'$ be $G/P(\varpi_\beta)$, i.e. $M'=\{0\}$, $P'=P(\varpi_\beta)$, $\mathbb{F}_{X'}=\{(\{0\},\emptyset)\}$ (and $\mathcal{R}'=\{\beta\}$). Then we can easily check the condition above to prove that there exists a $G$-equivariant morphism from $X$ to $G/P(\varpi_\beta)$. Note that the general fiber of this fibration is smooth horospherical of Picard group $\Zbb$ (in particular, it is homogeneous or one of the two-orbit varieties described in \cite{2orbits}).
	\item Let $P'$ be the parabolic subgroup containing $B$ (and $P$) such that $\mathcal{R}':=\tilde\sigma^{-1}(\{v_j\,\mid\,j\in \{1,\dots,s+1\}\} )$. Let $M'$ be the sublattice of $M$ orthogonal to $\Zbb u_1\oplus\cdots\oplus\Zbb u_r\subset N$. The pair $(P',M')$ corresponds to a horospherical homogeneous space $G/H'$ with $H'$ containing $H$. Also the dual lattice $N'$ of $M'$ is the image of the projection from $N$ to $\Zbb u_1\oplus\cdots\oplus\Zbb u_r$. We denote by $v_{1}',\dots,v_{s+1}'$ the images of $v_1,\dots,v_{s+1}$ in $N'$, in particular $v_{s+1}'=-v_{1}'-\cdots -v_{s}'$. And finally we denote by $\mathbb{F}_{X'}$ the colored fan $\{\mathcal{C}_{J}',\mathcal{F}_{J}')\,\mid\,J\subsetneq\{1,\dots,s\}\}$ where $\mathcal{C}_{J}'$ is the cone generated by the $v_{j}'$ with $j\in J$, and $\mathcal{F}_{J}'=\tilde\sigma^{-1}(\{v_j\,\mid\,j\in J\})$. The colored fan $\mathbb{F}_{X'}$ corresponds to a $G/H'$-embedding $X'$. Then we can check the condition above to prove that there exists a $G$-equivariant morphism  from $X$ to $X'$, which is a Mori fibration. Note that $X'$ and the general fiber of this fibration are smooth horospherical varieties of Picard group $\Zbb$ (in particular, they are homogeneous or one of the two-orbit varieties described in \cite{2orbits}).
	\end{enumerate}
	In the rest of the paper, in cases (1) and (2), we will denote this fibration by $\psi:\,X\longrightarrow Z$.
	\end{rem}

\subsection{Description via polytopes}	
	
	We now describe $X$ embedded in the projectivization of a $G$-module, by choosing the smallest ample Cartier divisor of $X$ and by applying Corollary~\ref{cor:smoothembedd}. We first study the nef cone of $X$, which is 2-dimensional.
	
	Recall that any Cartier divisor of $X$ is linearly equivalent to a $B$-stable divisor, and any prime $G$-stable divisor corresponds to an edge of $\mathbb{F}_X$ that is not generated by  some $\sigma(\alpha)$ with $\alpha\in\mathcal{F}_X$, and any other  $B$-stable prime divisor is the closure of a color of $G/H$. Then in Cases (1) and (2), we have $n+2$ prime $B$-stable divisors that we can denote naturally as follows:
	\begin{enumerate}[label=Case (\arabic*): ,start=1]
	\item $D_{n+1}=D_\beta$; for any $i\in\{0,\dots,n\}$, $D_i$ is the $B$-stable divisor corresponding to the edge generated by $e_i$ (which equals $D_\alpha$ with $\alpha\in\mathcal{F}_X=\mathcal{R}\backslash\{\beta\}$ if and only if the edge is generated by $\sigma(\alpha)$, and which is $G$-stable otherwise).
	\item for any $i\in\{0,\dots,r\}$,  $D_i$ is the $B$-stable divisor corresponding to the edge generated by $u_i$; and for any $j\in\{1,\dots,s+1\}$,  $D_{j+r}$ is the $B$-stable divisor corresponding to the edge generated by $v_i$ (which equals $D_\alpha$ with $\alpha\in\mathcal{F}_X=\mathcal{R}$ if and only if the edge is generated by $\sigma(\alpha)$, and which is $G$-stable otherwise).
	\end{enumerate}

	\begin{prop}\label{prop:NefCone}
	In both cases (1) and (2), the nef cone of $X$ is generated by $D_0$ and $D_{n+1}$. In particular, $D_0+D_{n+1}$ is ample. Moreover $(D_0,D_{n+1})$ is a basis of $\operatorname{Pic}(X)$.
	\end{prop}
	
	\begin{proof} We begin by computing the piecewise linear functions $h_{D_0}$ and  $h_{D_{n+1}}$ associated to these two Cartier divisors.
	
	\begin{enumerate}[label=Case (\arabic*): ,start=1]
	\item Consider the basis $(e_1^*,\dots, e_n^*)$ of $M$ that is dual to the basis $(e_1,\dots, e_n)$ of $N$. Then $h_{D_0}$ is defined on $N_\Qbb$ by: $(h_{D_0})_{|\mathcal{C}_I}=0$ if $I=\{1,\dots,n\}$; and for any $i\in\{1,\dots,n\}$, $(h_{D_0})_{|\mathcal{C}_I}=-e_i^*$ where $I=\{0,\dots,i-1,i+1,\dots,n\}$. And $h_{D_{n+1}}=0$.
	
	\item Consider the basis $(u_1^*,\dots, u_r^*, v_1^*,\dots,v_s^*)$ of $M$ that is dual to the basis $(u_1,\dots, u_r, v_1, \dots, v_s)$ of $N$. Then $h_{D_0}$ is defined on $N_\Qbb$ by: for any $J\subsetneq\{1,\dots,s+1\}$,  $(h_{D_0})_{|\mathcal{C}_{I,J}}=0$ if $I=\{1,\dots,r\}$; for any $i\in\{1,\dots,r\}$, $(h_{D_0})_{|\mathcal{C}_{I,J}}=-u_i^*$ where $I=\{0,\dots,i-1,i+1,\dots,r\}$ and $J=\{1,\dots,s\}$;  and, for any $i\in\{1,\dots,r\}$, for any $j\in\{1,\dots,s\}$, $(h_{D_0})_{|\mathcal{C}_{I,J}}=-u_i^*-a_iv_j^*$ where $I=\{0,\dots,i-1,i+1,\dots,r\}$ and $J=\{1,\dots,j-1,j+1,\dots,s+1\}$. And $h_{D_{n+1}}$ is defined on $N_\Qbb$ by: for any $I\subsetneq\{0,\dots,r\}$,  for  any $j\in\{1,\dots,s\}$, $(h_{D_{n+1}})_{|\mathcal{C}_{I,J}}=-v_j^*$ where $J=\{1,\dots,j-1,j+1,\dots,s+1\}$; and for any $I\subsetneq\{0,\dots,r\}$, $(h_{D_{n+1}})_{|\mathcal{C}_{I,J}}=0$ if $J=\{1,\dots,s\}$.
	\end{enumerate}
	
	By Theorem~\ref{th:divcrit}, one checks that $D_0$ and $D_{n+1}$ are globally generated but not ample. We also check that for any $a$ and $b$ in $\Qbb$, $aD_0+bD_{n+1}$ is Cartier if and only if $a$ and $b$ are integers.
	\end{proof}
	
	Before applying Corollary~\ref{cor:smoothembedd}, we reduce to the case where $G$ is the product of simply connected simple groups and a torus, with the following lemma.

	\begin{lem}[cf. proof of Proposition 3.10 in \cite{these}]\label{lem:Gred}
	Let $G':=[G,G]$ and let $\mathbb{T}$ be the torus $P/H$. Then $X$ is also a horospherical $G'\times \mathbb{T}$-variety. Moreover, if $\hat{G'}$ is the universal cover of $\hat{G'}$,  $X$ is also a horospherical $\hat{G'}\times \mathbb{T}$-variety.
	\end{lem}
	
	Without loss of generality by the lemma, we now assume that $G$ is the product $G'$ of simply connected simple groups and a torus $\mathbb{T}$. In particular, $P$ is the product of a parabolic subgroup of $G'$ with $\mathbb{T}$, and the characters of $P$ are sums of weights of the maximal torus of $G'$ and characters of $\mathbb{T}$. Hence a basis of $M\simeq\mathfrak{X}(\mathbb{T})$ is of the form $(\chi_i+\theta_i)_{i\in\{1,\dots,n\}}$ such that $(\chi_i)_{i\in\{1,\dots,n\}}$ form a basis of $M=\mathfrak{X}(\mathbb{T})$, and the $\theta_i$'s are weights of the maximal torus of $G'$.

	With these assumptions, we get the following result.
	
	\begin{lem}\label{lem:Xembedded} 
	The embedding of $X$ given by the ample Cartier divisor $D_0+D_{n+1}$ is:
	\begin{enumerate}[label=Case (\arabic*): ,start=1]
	\item $$X\hookrightarrow\mathbb{P}(\bigoplus_{i=0}^nV(\chi_i+\varpi_i+(1+a_i)\varpi_\beta)),$$ where $\chi_0=0$, $\chi_1,\dots,\chi_n$ are characters of $\mathbb{T}$, and for any $i\in\{0,\dots,n\}$, $\varpi_i$ is either $\varpi_\alpha$ if $e_i=\sigma(\alpha)$ with $\alpha\in\mathcal{F}_X$ or 0 otherwise.
	\item 	$$X\hookrightarrow\mathbb{P}(\bigoplus_{i,b_1,\dots,b_{s+1}}V(\chi_i+\varpi_i+\sum_{j=1}^{s+1}b_j(\chi_{r+j}+\varpi_{r+j})),$$ where $\chi_0=\chi_{n+1}=0$, $\chi_1,\dots,\chi_n$ are characters of $\mathbb{T}$, and for any $i\in\{0,\dots,n+1\}$, $\varpi_i$ is either $\varpi_\alpha$ if $u_i$ or $v_{i-r}$ is $\sigma(\alpha)$ with $\alpha\in\mathcal{F}_X$ or 0 otherwise; and where the sum is taken over all $s+2$-tuples of non-negative integers $(i,b_1,\dots,b_{s+1})$ such that $0\leq i\leq r$ and $\sum_{j=1}^{s+1}b_j=1+a_i$ (with $a_0:=0$).
	\end{enumerate}
	\end{lem}
	
	\begin{proof}
	In each case, we describe the pseudo-moment polytope of $(X,D_0+D_{n+1})$ in a particular basis of $M$ and then the moment polytope of $(X,D_0+D_{n+1})$. Then we use Corollary~\ref{cor:smoothembedd} to conclude.
	\begin{enumerate}[label=Case (\arabic*): ,start=1]
	\item  By the previous lemma and the description of the images of colors, for any $i\in\{1,\dots,n\}$, the element $e_i^*$ is of the form $\chi_i+\varpi_i-\varpi_0+a_i\varpi_\beta$, where $\chi_1,\dots,\chi_n$ are characters of $\mathbb{T}$ and for any $i\in\{0,\dots,n\}$, $\varpi_i$ is either $\varpi_\alpha$ if $e_i=\sigma(\alpha)$ with $\alpha\in\mathcal{F}_X$ or 0 otherwise.
	
	 The pseudo-moment polytope of $(X,D_0+D_{n+1})$ is the simplex with vertices $0$, $e_1^*,\dots, e_n^*$. The weight of the canonical section of $ D_0+D_{n+1}$ is $\varpi_0+\varpi_\beta$, where $\varpi_0$ is either $\varpi_\alpha$ if $e_0=\sigma(\alpha)$ with $\alpha\in\mathcal{F}_X$ or 0 otherwise. Hence, the moment polytope of $(X,D_0+D_{n+1})$ is the simplex with vertices $0+\varpi_0+\varpi_\beta=\chi_0+\varpi_0+(1+a_0)\varpi_\beta$ and $(\chi_i+\varpi_i-\varpi_0+a_i\varpi_\beta)+(\varpi_0+\varpi_\beta)=\chi_i+\varpi_i+(1+a_i)\varpi_\beta$ for any $i\in\{1,\dots,n\}$.
	
	\item  By the previous lemma and the description of the images of colors, for any $i\in\{1,\dots,r\}$ the element $u_i^*$ is of the form $\chi_i+\varpi_i-\varpi_0+a_i\varpi_{n+1}$ and for any $j\in\{1,\dots,s\}$ the element $v_j^*$ is of the form $\chi_{r+j}+\varpi_{r+j}-\varpi_{n+1}$, where $\chi_1,\dots,\chi_n$ are characters of $\mathbb{T}$, and for any $i\in\{0,\dots,n+1\}$, $\varpi_i$ is either $\varpi_\alpha$ if $u_i$ (with $0\leq i\leq r$) or $v_{i-r}$ (with $r+1\leq i\leq n+1$) is $\sigma(\alpha)$ with $\alpha\in\mathcal{F}_X$ or 0 otherwise. 
	
	The pseudo-moment polytope of $(X,D_0+D_{n+1})$ is the polytope with the following vertices: $0$, $u_1^*,\dots, u_r^*$, $v_1^*,\dots,v_s^*$ and $u_i^*+(a_i+1)v_j^*$ for any $1\leq i\leq r$ and for any $1\leq j\leq s$. Note that the lattice points of this polytope are exactly $0$, $v_1^*,\dots,v_s^*$ and for any $1\leq i\leq r$ all the points of the form $u_i^*+\sum_{j=1}^s b_jv_j^*$ where the $b_j$'s are non-negative integers such that $\sum_{j=1}^sb_j\leq a_i+1$. Moreover, the weight of the canonical section of $ D_0+D_{n+1}$ is $\varpi_0+\varpi_{n+1}$, where $\varpi_0$ (respectively $\varpi_{n+1}$) is either $\varpi_\alpha$ if $u_0$ (respectively $v_{s+1}$) equals $\sigma(\alpha)$ with $\alpha\in\mathcal{F}_X$ or 0 otherwise. Hence, the moment polytope of $(X,D_0+D_{n+1})$ is the polytope with vertices $0+\varpi_0+\varpi_{n+1}=\chi_0+\varpi_0+(1+a_0)(\chi_{n+1}+\varpi_{n+1})$; for any $i\in\{1,\dots,r\}$, $\chi_i+\varpi_i+(a_i+1)(\chi_{n+1}+\varpi_{n+1})$; for any $j\in\{1,\dots,s\}$, $\chi_0+\varpi_0+\chi_{r+j}+\varpi_{r+j}$; and for any $1\leq i\leq r$, for any $1\leq j\leq s$, $\chi_i+\varpi_i-\varpi_0+a_i\varpi_{n+1}+(a_i+1)(\chi_{r+j}+\varpi_{r+j}-\varpi_{n+1})+\varpi_0+\varpi_{n+1}=\chi_i+\varpi_i+(a_i+1)(\chi_{r+j}+\varpi_{r+j})$.
	
	In particular, the lattice points of the pseudo-moment polytope translated by $\varpi_0+\varpi_{n+1}$ are exactly the $\chi_i+\varpi_i+\sum_{j=1}^{s+1}b_j(\chi_{r+j}+\varpi_{r+j})$ where the sum is taken over all $s+2$-tuples of non-negative integers $(i,b_1,\dots,b_{s+1})$ such that $0\leq i\leq r$ and $\sum_{j=1}^{s+1}b_j=1+a_i$.
\end{enumerate}
\end{proof}

	Recall that, by Lemma~\ref{lem:Gred}, $(\chi_1,\dots,\chi_n)$ is a basis of $\mathfrak{X}(\mathbb{T})$. Hence, there exists a subtorus $\mathbb{S}$ of $\mathbb{T}$ such that: $(\chi_{i|\mathbb{S}} )_{i\in\{1,\dots,n\},\,\varpi_i=0}$ is a basis of $\mathfrak{X}(\mathbb{S})$, and for any $i\in\{1,\dots,n\}$ such that $\varpi_i\neq 0$, we have $\chi_{i|\mathbb{S}}=0$.
	
\begin{lem}
In both cases (1) and (2), $X$ is also a horospherical $G'\times\mathbb{S}$-variety.
\end{lem}

\begin{proof}
Consider Case (1). For any $i\in\{1,\dots,n\}$ such that $\varpi_i\neq 0$, the $G$-orbit and the $G'\times\mathbb{S}$-orbit of the highest weight vector $v_{\chi_i+\varpi_i+(1+a_i)\varpi_\beta}$ in
$$V_G( \chi_i+\varpi_i+(1+a_i)\varpi_\beta)\simeq V_{G'\times\mathbb{S}}( \chi_i+\varpi_i+(1+a_i)\varpi_\beta)=V_{G'\times\mathbb{S}}(\varpi_i+(1+a_i)\varpi_\beta)$$
are equal. Case (2) is similar.
\end{proof}
	
	We can replace $\chi_i+\varpi_i$ with $\varpi_{\alpha_i}$ such that 
	\begin{itemize}[label=$\ast$]
	\item if $\chi_{i|\mathbb{S}}=0$ and $\varpi_i\neq 0$, $\alpha_i$ is a simple root of $G'$ (that is supposed to be a product of simply connected simple groups);
	\item $\mathbb{S}$ is a product of $\Cbb^*$'s whose trivial simple roots are the $\alpha_i$'s with $i$ such that $\chi_{i|\mathbb{S}}\neq 0$ and $\varpi_i=0$;
	\item if $i=0$ or $n+1$, $\chi_{i|\mathbb{S}}=0$, and $\varpi_i=0$, we have that $\alpha_i$ is the trivial root of $\{1\}$.
	\end{itemize} 
	This finally gives the following proposition.
	
	\begin{prop}\label{prop:embeded}
	
	Let $X$ be a smooth projective horospherical variety of Picard group $\Zbb^2$ as in Case (1) or (2). Then $X$ is isomorphic to a smooth closure of a $G$-orbit of a sum of highest weight vectors as follows where $G$ is the product
 $G_0\times\cdots\times G_t$ of simply connected simple groups, $\Cbb^*$ and $\{1\}$:
	\begin{enumerate}[label=Case (\arabic*): ,start=1]
	\item $$\mathbb{P}(\bigoplus_{i=0}^nV(\varpi_{\alpha_i}+(1+a_i)\varpi_\beta)),$$ where \begin{itemize}[label=$\ast$]
	
	\item $n\geq 1$ and $\beta$ is a (non-trivial) simple root of $G_0$;
	 \item $\alpha_0,\dots,\alpha_n$ are distinct simple roots (may be trivial) of $G$ distinct from $\beta$;
	
	\item for any $k\in\{1,\dots,t\}$, $G_k=\{1\}$ if and only if $k=1$ and $\alpha_0$ is the trivial root of $G_1$;
	
	\item and $0=a_0\leq a_1\leq\cdots\leq a_n$ are integers.
	\end{itemize}
	
	\item 	$$\mathbb{P}(\bigoplus_{i,b_1,\dots,b_{s+1}}V(\varpi_{\alpha_i}+\sum_{j=1}^{s+1}b_j(\varpi_{\alpha_{r+j}})),$$ where 
	
	\begin{itemize}[label=$\ast$]
	
	\item the sum is taken over all $s+2$-tuples of non-negative integers $(i,b_1,\dots,b_{s+1})$ such that $0\leq i\leq r$ and $\sum_{j=1}^{s+1}b_j=1+a_i$ (with $a_0:=0$);
	\item 	$r\geq 1$, $s\geq 1$ and $r+s=n$;
	
	\item 	$\alpha_0,\dots,\alpha_{n+1}$ are distinct simple roots (may be trivial) of $G$;
	\item 	for any $k\in\{0,\dots,t\}$, $G_k=\{1\}$ if and only if, $k=0$ and $\alpha_0$ is the trivial root of $G_0$, or $k=t$ and $\alpha_{n+1}$ is the trivial root of $G_t$;
	\item 	and $0=a_0\leq a_1\leq\cdots\leq a_r$ are integers.
	 \end{itemize}
	\end{enumerate}
	\end{prop}
	
	These two cases of Proposition~\ref{prop:embeded} justify {\it the definition of two types of varieties}. In Case (2), we already only consider the case where $s=1$ to simplify the definition ; we will prove in Section~\ref{sect43} that we can reduce to this case.

\begin{defi}\label{def:varieties}
Let $G=G_0\times\cdots\times G_t$ be a product of simply connected simple groups, $\Cbb^*$ and $\{1\}$ (with $t\geq 0$).
\begin{enumerate}[label= (\arabic*) ,start=1]
		\item Suppose $G_0$ to be a simple group. Let $\beta$ be a simple root of $G_0$ (non-trivial), let $n\geq \max\{1,t\}$,  let $\alpha_0,\dots,\alpha_n$ be distinct, possibly trivial, simple roots of $G$ different from $\beta$ and let $0=a_0\leq a_1\leq \cdots\leq a_n$ be integers. Suppose also that, for any $k\in\{1,\dots,t\}$, $G_k=\{1\}$ if and only if $k=1$ and $\alpha_0$ is the trivial root of $G_1$. Denote $\underline{\alpha}:=(\alpha_0,\dots,\alpha_n)$ and $\underline{a}:=(a_0,\dots,a_n)$. We define $\mathbb{X}^1(G,\beta,\underline{\alpha},\underline{a})$ to be the closure of the $G$-orbit of a sum of highest weight vectors in $$\Pbb\left(\bigoplus_{i=0}^nV(\varpi_{\alpha_i}+(1+a_i)\varpi_\beta) \right).$$
		\item Suppose $t\geq 1$. Let $n\geq 2$, let $0=a_0\leq a_1\leq \cdots\leq a_{n-1}$ be integers, and let $\alpha_0,\dots,\alpha_{n+1}$ be distinct, possibly trivial, simple roots of $G$. Suppose also that, for any $k\in\{0,\dots,t\}$, $G_k=\{1\}$ if and only if, $k=0$ and $\alpha_0$ is the trivial root of $G_0$, or $k=t$ and $\alpha_{n+1}$ is the trivial root of $G_t$. Denote $\underline{\alpha}:=(\alpha_0,\dots,\alpha_{n+1})$ and $\underline{a}:=(a_0,\dots,a_{n-1})$. We define $\mathbb{X}^2(G,\underline{\alpha},\underline{a})$ to be the closure of the $G$-orbit of a sum of highest weight vectors in $$\Pbb\left(\bigoplus_{i=0}^{n-1} \bigoplus_{b=0}^{1+a_i}V(\varpi_{\alpha_i}+b\varpi_{\alpha_{n}}+(1+a_i-b)\varpi_{\alpha_{n+1}}) \right).$$ 
\end{enumerate}
\end{defi}

	\begin{rem}
	Up to reordering the $G_k$'s and taking $t$ minimal, we can assume that:
	\begin{enumerate}[label=Case (\arabic*): ,start=1]
	\item the map $\{\alpha_0,\dots,\alpha_n\}\backslash R_0 \longrightarrow\{1,\dots,t\}$ is surjective and increasing, where $R_0$ denotes the set of simple roots of $G_0$;
	  \item the map $\{\alpha_0,\dots,\alpha_{n+1}\} \longrightarrow\{0,\dots,t\}$ is surjective and increasing.
	  \end{enumerate}
	  
	\end{rem}

	\begin{ex}\label{ex:X1X2}
	We give here some examples of smooth horospherical varieties of types $\mathbb{X}^1$ and $\mathbb{X}^2$:
\begin{align*}
&W_1:=\mathbb{X}^1(E_6\times\{0\}\times\Cbb^*\times\Cbb^*,\alpha_4(E_6),(\alpha(\{0\}),\alpha_5(E_6),\alpha(\Cbb^*),\alpha(\Cbb^*),\alpha_1(E_6),\alpha_2(E_6)),(0,0,1,1,2,2));\\
&W_2:=\mathbb{X}^2(\operatorname{SL}_2\times\operatorname{SL}_3\times\operatorname{Sp}_8\times \operatorname{Spin}_{7},(\alpha_1(\operatorname{SL}_2),\alpha_1(\operatorname{SL}_3),\alpha_1(\operatorname{Sp}_8),\alpha_1(\operatorname{Spin}_{7}),\alpha_3(\operatorname{Spin}_{7})),(0,0,1));\\
&\mathrm{and}\, W_3 :=\mathbb{X}^2(\{0\}\times\Cbb^*\times\Cbb^*\times\operatorname{SL}_2\times\{0\},(\alpha(\{0\},\alpha(\Cbb^*),\alpha(\Cbb^*),\alpha(\operatorname{SL}_2),\alpha(\{0\})),(0,2,3)).
\end{align*}
	 
 We will see that $W_3$ is the toric variety $\Pbb(\mathcal{O}_{\Pbb^2}\oplus\mathcal{O}_{\Pbb^2}(2)\oplus\mathcal{O}_{\Pbb^2}(3))$.
	\end{ex}
	\section{Reduction to the cases of Theorem~\ref{th:main}}
	
	This section we define the restricted conditions mentioned in Theorem~\ref{th:main}, and we prove that we can reduce the cases of Proposition~\ref{prop:embeded} to the varieties $\mathbb{X}^1(G,\beta,\underline{\alpha},\underline{a})$ and $\mathbb{X}^2(G,\underline{\alpha},\underline{a})$ with these restricted conditions.
	
	\subsection{Smooth horospherical varieties and $G$-modules}
	To prove Theorem~\ref{th:main} from Proposition~\ref{prop:embeded}, we replace sums of irreducible $G$-modules with irreducible $\mathbb{G}$-modules with $G\subset\mathbb{G}$ as soon as we can. Then we enlarge the group $G$ and we reduce to ``smaller'' cases (for example to horospherical varieties with smaller rank).  For this, we first need to apply the smoothness criterion to $X$ (Theorem~\ref{th:smooth}), which comes from the fact that  horospherical $G$-modules (i.e. $G$-modules that are horospherical as varieties) are the $\Cbb^*$-modules $\Cbb$, the $\operatorname{SL}_d$-modules $V(\varpi_1)=\Cbb^d$ and $\operatorname{Sp}_d$-modules (with $d$ even) $V(\varpi_1)=\Cbb^d$. And then we use easy facts as ``the $\SL_d\times\SL_e$-module $\Cbb^{d}\oplus\Cbb^e$ is isomorphic to the $\SL_{d+e}$-module $\Cbb^{d+e}$''.

	As in \cite[Theorem~1.7]{2orbits}, the smoothness criterion reveals 8 configurations including the 5 configurations that give the five families of horospherical two-orbit varieties corresponding to non-homogeneous smooth projective horospherical varieties of Picard group $\Zbb$. We recall these 8 configurations in the following definition.
	
	\begin{defi}\label{def:smoothtriple}
	Let $K$ be a simple algebraic group over $\Cbb$ and let $\gamma$, $\delta$ be two simple roots of $K$. The {\it triple} $(K,\gamma,\delta)$  is said to be {\it smooth} if $(\mbox{type of }K,\gamma,\delta)$ is one of the following 8 cases, up to exchanging $\gamma$ and $\delta$  (with the notation of Bourbaki \cite{Bourbaki8}).
	\begin{enumerate}[label=\arabic*.]
	\item $(A_m,\alpha_1,\alpha_m)$, with $m\geq 2$

\item $(A_m,\alpha_i,\alpha_{i+1})$, with $m\geq 3$ and $i\in\{1,\ldots,m-1\}$

\item $(B_m,\alpha_{m-1},\alpha_m)$, with $m\geq 3$

\item $(B_3,\alpha_1,\alpha_3)$

\item $(C_m,\alpha_i,\alpha_{i+1})$ with $m\geq 2$ and $i\in\{1,\ldots,m-1\}$

\item $(D_m,\alpha_{m-1},\alpha_m)$, with $m\geq 4$

\item $(F_4,\alpha_2,\alpha_3)$

\item $(G_2,\alpha_1,\alpha_2)$
	\end{enumerate}
	
	We say that the triple $(\mbox{type of }K,\gamma,\delta)$ is smooth of two-orbit type if it is one of the cases 3, 4, 5, 7 or 8 above.
	\end{defi}
	
	\begin{rem}\label{rem:2orbits}
	The smooth triples of two-orbit type correspond bijectively to the isomorphism classes of non-homogeneous projective smooth horospherical varieties with Picard group $\Zbb$. These varieties have two orbits under the action of their automorphism groups, which are given in \cite[Theorem~1.11]{2orbits} and justify that all these varieties are distinct. 
	\end{rem}
	
	Here we also need to introduce another ``smooth object'' (only used in Case (1)).
	
\begin{defi}\label{def:smoothtriplebis} 
	Let $K$ be a simple algebraic group over $\Cbb$ and let $\beta$ be a simple root of $K$ and let $R$ be a subset of simple roots of $K$, all distinct from $\beta$. Let $n$ be a non-negative integer.
	Denote by $L$ the Levi subgroup of the maximal parabolic subgroup $P(\varpi_\beta)$ of $K$, then the semi-simple part of $L$ is a quotient by a finite central group of a product of simple groups $L^1,\dots,L^{q}$ (with $q\geq 0$).

	The {\it quadruple} $(K,\beta,R,n)$  is said to be {\it smooth} if 
	
	\begin{enumerate}
	\item $n=1$, $R=\{\gamma,\delta\}$ such that  $\gamma$ and $\delta$ are simple roots of the same $L^k$ such that the triple $(L^k,\gamma,\delta)$ is smooth;
	\item or for any $k\in\{1,\dots,q\}$, at most one simple root of $L^k$ is in $R$, and if $\gamma\in R$ is a simple root of $L_k$, then $L_k$ is of type $A$ or $C$ and $\gamma$ is a short extremal simple root of $L_k$.
	
	\end{enumerate}
	\end{defi}	
	
	We can list all smooth quadruples $(K,\beta,R,n)$ (see the appendix).
We remark, in particular, that $R$ is at most of cardinality 3.

We can now define the restricted conditions that allow us to state Theorems~\ref{th:main} and~\ref{th:main2}.

\begin{defi}\label{def:RC1}
Let $X=	\mathbb{X}^1(G,\beta,\underline{\alpha},\underline{a})$ as in Definition~\ref{def:varieties}. Recall that $R_0$ is the maximal subset of $\{\alpha_0,\dots,\alpha_n\}$ consisting of simple roots of $G_0$. We say that $X$ satisfies the restricted condition (a), (b) or (c) respectively if it satisfies all the following properties including (a), (b) or (c) respectively.

\begin{enumerate}

\item  The quadruple $(G_0,\beta,R_0,n)$ is smooth. 
	
\item If $R_0$ is empty, then $G_0$ is the universal cover of the automorphism group of $G/P(\varpi_\beta)$.
	
\item If $i<j$ and $a_i=a_j$ then  $\alpha_j\in R_0$. Moreover, if  $\alpha_i$ and $\alpha_j$ are in $R_0$, we suppose them to be ordered with Bourbaki's notation as simple roots of $G_0$.
	
\item One of the three following cases occurs.
	\begin{enumerate}[label=(\alph*)]
	\item 
	
	We have $n=t=1$,  $\alpha_0$ and $\alpha_1$ are both simple roots of $G_1$ such that the triple $(G_1,\alpha_0,\alpha_1)$ is smooth; in particular, $R_0=\emptyset$ and $a_0<a_1$. 
	
\vskip\baselineskip
In the two next cases, the map $\{\alpha_0,\dots,\alpha_n\}\backslash R_0 \longrightarrow\{1,\dots,t\}$ is surjective and strictly increasing, and for any $k\in\{1,\dots,t\}$, either $G_k$ is isomorphic to some $\SL_{d_k}$ and $\alpha_{i_k}$ is the first simple root of $G_k$, or  $G_k$ is isomorphic to $\Cbb^*$ or $\{1\}$ and $\alpha_{i_k}$ is the trivial simple root of $G_k$.
	
\item The simple root $\alpha_n$ is not trivial (in particular if $a_{n-1}=a_n$). 
	\item The simple root $\alpha_n$ is trivial (and then $a_{n-1}<a_n$).
			\end{enumerate}
			
			\end{enumerate}
\end{defi}

\begin{defi}\label{def:RC2}
Let $X=	\mathbb{X}^2(G,\underline{\alpha},\underline{a})$ as in Definition~\ref{def:varieties}. We say that $X$ satisfies the restricted condition (a), (b) or (c) respectively if it satisfies all the following properties including (a), (b) or (c) respectively.
	\begin{enumerate}

\item We have $0=a_0< a_1<\cdots< a_n$.
\item The triple $(G_t,\alpha_{n},\alpha_{n+1})$ is smooth of two-orbit type; in particular, $\alpha_{n}$ and $\alpha_{n+1}$ are both simple roots of $G_t$ and $\alpha_0,\dots,\alpha_{n-1}$ are simple roots of $G_0\times G_1\times\cdots\times G_{t-1}$.

 \item  One of the three following cases occurs.
	 
	\begin{enumerate}[label=(\alph*)]
	\item  We have $n=2$, $t=1$ and the triple $(G_0,\alpha_{0},\alpha_{1})$ is smooth.
	
\vskip\baselineskip
In the two next cases:  $t=n$, the map $\{\alpha_0,\dots,\alpha_{n-1}\} \longrightarrow\{0,\dots,t-1\}$ is surjective and strictly increasing; and for any $i\in\{1,\dots,t\}$, either $G_i$ is isomorphic to some $\SL_{d_i}$ and $\alpha_i$ is the first simple root of $G_i$, or  $G_i$ is isomorphic to $\Cbb^*$ or $\{1\}$ and  $\alpha_i$ is  the trivial simple root of $G_i$.
	
	\item The simple root $\alpha_{n-1}$ is not  trivial.
	\item The simple root $\alpha_{n-1}$ is trivial.
			\end{enumerate}\end{enumerate}			
\end{defi}

\begin{rem}\label{rem:decprojbund}
	In Theorem~\ref{th:main}, the decomposable projective bundles over projective spaces are the horospherical varieties $X$ in Case (1) with restricted condition (b) or (c), and such that $R_0=\emptyset$ and   $\varpi_\beta$ is the first simple root of $G_0=\operatorname{SL}_{d_0}$ for some $d_0\geq 2$ (and $0<a_1<\cdots<a_n$).
	\end{rem}
	
	\begin{ex}
	The three varieties given in Examples~\ref{ex:X1X2} do not satisfy the restricted condition. Indeed, for $W_1$ we have $a_2=a_3$ but  $\alpha_3$ is not a simple root of $G_0$. For $W_2$, we have $a_0=a_1$ and  $G_2=\operatorname{Sp}_8$. And for $W_3$, $(G_t,\alpha_n,\alpha_{n+1})$ is not smooth of two-orbit type.
	
	But we will prove in the rest of the section that these three varieties are isomorphic to horospherical varieties of type $\mathbb{X}^1$ or $\mathbb{X}^2$ satisfying the restricted condition. More precisely, $$W_1\simeq\mathbb{X}^1(E_6\times\{0\}\times\operatorname{SL}_2,\alpha_4(E_6),(\alpha(\{0\}),\alpha_5(E_6),\alpha(\operatorname{SL}_2),\alpha_1(E_6),\alpha_2(E_6)),(0,0,1,2,2))$$ which satisfies the restricted condition (b);
		$$W_2\simeq\mathbb{X}^2(\operatorname{SL}_5\times\operatorname{SL}_8\times \operatorname{Spin}_{7},(\alpha_1(\operatorname{SL}_5),\alpha_1(\operatorname{SL}_8),\alpha_1(\operatorname{Spin}_{7}),\alpha_3(\operatorname{Spin}_{7})),(0,1))$$ and satisfies the restricted condition (b);
	 $$\mbox{and } W_3:=\mathbb{X}^1(\operatorname{SL}_3\times\{0\}\times\Cbb^*\times\Cbb^*,\alpha_1(\operatorname{SL}_3), (\alpha(\{0\}),\alpha(\Cbb^*),\alpha(\Cbb^*)),(0,2,3))$$ and satisfies the restricted condition (c).
	 
	 We can give other examples satisfying the restricted condition in Case (1) (a): for any $G_0$ and $\beta$, $$\mathbb{X}^1(G_0\times\operatorname{SL}_4,\beta,(\alpha_1(\operatorname{SL}_4),\alpha_3(\operatorname{SL}_4)),(0,1));$$
	 
	 in Case (2) (a): $$\mathbb{X}^2(\operatorname{SL}_4\times\operatorname{Sp_6},(\alpha_1(\operatorname{SL}_4),\alpha_2(\operatorname{SL}_4),\alpha_2(\operatorname{Sp_6}),\alpha_3(\operatorname{Sp_6})),(0,1));$$
	 
	 and in Case (2) (c): $$\mathbb{X}^2(\{0\}\times\Cbb^*\times F_4,(\alpha(\{0\}),\alpha(\Cbb^*), \alpha_2(F_4),\alpha_3(F_4)),(0,1)).$$
	\end{ex}

We begin by applying the smoothness criterion to get some part of the restricted condition. We suppose that $X$ is as in Proposition~\ref{prop:embeded}. Recall that the colored fans $\mathbb{F}^1$ and $\mathbb{F}^2$ of the horospherical varieties in Cases (1) and (2) respectively are as follows.

 The colored fan $\mathbb{F}^1$ is the complete colored fan whose maximal colored cones are generated by all $u_0,\dots,u_n$ except one and with all possible colors except $\beta$, where $(u_1,\dots,u_n)$ is a basis of $N$ and $u_0=-u_1-\cdots-u_n$. Recall also that the map $\sigma$ is injective from the set $\mathcal{R}\backslash\{\beta\}$ of colors of the horospherical variety to $\{u_0,\dots,u_n\}$ and $\sigma(\beta)=\beta_M^\vee$ is $a_1u_1+\cdots+a_nu_n$.
 
 The colored fan $\mathbb{F}^2$ is the complete colored fan whose maximal colored cones are generated by all $u_0,\dots,u_r,v_1,\dots,v_{s+1}$ except one $u_i$ and one $v_j$, and with all possible colors, where $(u_1,\dots,u_r,v_1,\dots,v_s)$ is a basis of $N$, $u_0=-u_1-\cdots-u_r$ and $v_s=a_1u_1+\cdots+a_ru_r-v_1-\cdots-v_s$. Recall also that the map $\sigma$ is injective from the set $\mathcal{R}$ of colors of the horospherical variety to $\{u_0,\dots,u_r,v_1,\dots,v_s\}$.
	
	\begin{lem}\label{lem:WithSmoothness} 
	~
	\begin{enumerate}[label=Case (\arabic*): ,start=1]
	\item The quadruple $(G_0,\beta,R_0,n)$ is smooth.  If there exist $0\leq i<j\leq n$ such that $\alpha_i$ and $\alpha_j$ are simple roots of the same simple group $G_{k}$ with $k\in\{1,\dots,t\}$ then $n=1$, $i=0$ and $j=1$ (also $t=k=1$). Moreover in that case, the triple $(G_{k},\alpha_i,\alpha_j)$ is smooth.
	
\noindent otherwise, for any $i\in\{0,\dots,n\}$, the simple root $\alpha_i$ is either trivial or in $G_0$ or the short extremal simple root of some simple group $G_k$ with $k\in\{1,\dots,t\}$ that is of type $A$ or $C$.

	\item 	If there exist $0\leq i<j\leq n+1$ such that $\alpha_i$ and $\alpha_j$ are simple roots of the same simple group $G_{k}$ with $k\in\{0,\dots,t\}$ then either $r=1$, $i=0$ and $j=1$, or $s=1$, $i=n$ and $j=n+1$. Moreover in that case, the triple $(G_{k},\alpha_i,\alpha_j)$ is smooth.
	
\noindent For any $i\in\{0,\dots,n\}$, such that the simple root $\alpha_i$ is the unique $\alpha_j$ of a simple group $G_{k}$ with $k\in\{0,\dots,t\}$, the root $\alpha_i$ is either trivial or the short extremal simple root of $G_{k}$ that is of type $A$ or $C$. 
	\end{enumerate}
	\end{lem}

	\begin{proof}
	~
	\begin{enumerate}[label=Case (\arabic*): ,start=1]
	\item With notation of Definition~\ref{def:smoothtriplebis} (with $K=G_0$), suppose $\gamma$ and $\delta$ are two simple roots of the same $L^j$. If $n>1$, then there exists a maximal colored cone of $\mathbb{F}_X$ that contains $\gamma^\vee_{M}$ and $\delta^\vee_{M}$. By applying Theorem~\ref{th:smooth}, we get a contradiction.  Then $n=1$ and applying Theorem~\ref{th:smooth} to the two one-dimensional colored cones of $\mathbb{F}_X$, we have that the pairs $(R_0\backslash\{\beta,\delta\},\gamma)$ and $(R_0\backslash\{\beta,\gamma\},\delta)$ are smooth, so that $(L^j,\gamma,\delta)$ is smooth (from a case by case study done in \cite[Proof of Theorem~1.7]{2orbits}).
	
	Suppose that $\alpha$ is the unique simple root of $L^j$ in $R_0$. By applying Theorem~\ref{th:smooth} to the colored cone $(\Qbb_{\geq 0}\alpha_M^\vee,\{\alpha\})$ we get that $L^j$ is of type $A$ or $C$ and $\alpha$ is a short extremal simple root of $L^j$.	
	This finishes the proof of the smoothness of $(G_0,\beta,R_0,n)$.
	
	If there exist $0\leq i<j\leq n$ such that $\alpha_i$ and $\alpha_j$ are simple roots of the same simple group $G_{k}$ with $k\in\{1,\dots,t\}$ then as above Theorem~\ref{th:smooth} implies that $n=1$ and $(G_{k},\alpha_i,\alpha_j)$ is smooth. The fact that $i=0$, $j=1$ and  $t=k=1$ is obvious.
	
	Now, let $i\in\{0,\dots,n\}$ such that the simple root $\alpha_i$ is the unique $\alpha_j$ of a simple group $G_{k}$ with $k\in\{1,\dots,t\}$ and suppose that $\alpha_i$ is not trivial. Apply again Theorem~\ref{th:smooth} to the colored cone $(\Qbb_{\geq 0}\alpha_M^\vee,\{\alpha\})$ to get that $\alpha_i$ is    the short extremal simple root $G_k$ with $k\in\{1,\dots,t\}$ that is of type $A$ or $C$. This finishes the proof of the lemma in Case (1).
	
	\item Suppose there exist $0\leq i<j\leq n+1$ such that $\alpha_i$ and $\alpha_j$ are simple roots of the same simple group $G_{k}$ with $k\in\{0,\dots,t\}$. Then Theorem~\ref{th:smooth} implies that $(G_{k},\alpha_i,\alpha_j)$ is smooth (still from the case by case study done in \cite[Proof of Theorem~1.7]{2orbits}). But this also gives a contradiction if there exists a maximal colored cone of $\mathbb{F}_X$ that contains $\alpha^\vee_{i,M}$ and $\alpha^\vee_{j,M}$. This contradiction occurs if and only if $0\leq i\leq r$ and $r+1\leq j\leq n+1$, or $0\leq i,j\leq r$ and $r\geq 2$, or $r+1\leq i,j\leq n+1$ and $s\geq 2$.
	
	We conclude the proof of the lemma in Case (2) as in Case (1).
	\end{enumerate}
	\end{proof}

	Now we list different ways to replace sums of irreducible $G$-modules with irreducible $\mathbb{G}$-modules with $G\subset\mathbb{G}$.
	
	\begin{lem}\label{lem:gather}
	
	Let $\tau\geq 1$. For $i\in\{1,\dots,\tau\}$, let $G_i$ be $\Cbb^*$, $\operatorname{SL}_{d_i}$ (with $d_i\geq 2$) or $\operatorname{Sp}_{d_i}$ (with $d_i\geq 2$ even). If $G_i=\Cbb^*$ set $d_i=1$ and $\varpi_1^i$ the identity character of $\Cbb^*$. otherwise, set $\varpi_1^i$ the first fundamental weight of $G_i$. Let $G=G_1\times\cdots\times G_\tau$.
	
	\begin{enumerate}[label=(\alph*)]
	\item   Let $\mathbb{G}=\operatorname{SL}_{d}$ where $d=d_1+\cdots+d_\tau$.
	
	\noindent Then $V_{\mathbb{G}}(\varpi_1)=\bigoplus_{i=1}^\tau V_G(\varpi^i_1)$ and  $G\cdot\left(\sum_{i=1}^\tau v_{\varpi_1}^i\right)\subset\mathbb{G}\cdot v_{\varpi_1}$.
	
		\item Let $\mathbb{G}=\operatorname{SL}_{d}$ where $d=d_1+\cdots+d_\tau+1$.
		
		\noindent Then $V_{\mathbb{G}}(\varpi_1)=V_G(0)\oplus\bigoplus_{i=1}^\tau V_G(\varpi_1^i) $  and  $G\cdot\left(1+\sum_{i=1}^\tau v_{\varpi_1}^i\right)\subset\mathbb{G}\cdot v_{\varpi_1}$, where $1$ denotes the unit in the trivial $G$-module $V_G(0)=\Cbb$.
		
\vskip\baselineskip
\noindent With notation of Bourbaki \cite{Bourbaki8} (we put primes to write differently fundamental weights  of $\mathbb{G}$ from those of $G$).

		\item Let $G=\operatorname{SL}_d$ (with $d\geq 3$) and $\mathbb{G}=\operatorname{SO}_{2d}$.
		
		\noindent Then $V_{\mathbb{G}}(\varpi_1')=V_G(\varpi_1)\oplus V_G(\varpi_{d-1})$ and $G\cdot\left(v_{\varpi_1}+v_{\varpi_{d-1}}\right)\subset\mathbb{G}\cdot v_{\varpi_1'}$. 
		
		\item
		Let $G=\operatorname{SL}_d$ (with $d\geq 4$), $\mathbb{G}=\operatorname{SL}_{d+1}$ and $1\leq i\leq d-2$. 
		
		\noindent Then $V_{\mathbb{G}}(\varpi_{i+1}')=V_G(\varpi_i)\oplus V_G(\varpi_{i+1})$ and $G\cdot\left(v_{\varpi_i}+v_{\varpi_{i+1}}\right)\subset\mathbb{G}\cdot v_{\varpi_{i+1}'}$. 
		
		\item Let $G=\operatorname{Spin}_{2d}$ (with $d\geq 4$) and $\mathbb{G}=\operatorname{Spin}_{2d+1}$.
		
		\noindent Then $V_{\mathbb{G}}(\varpi_{d}')=V_G(\varpi_{d-1})\oplus V_G(\varpi_{d})$ and $G\cdot\left(v_{\varpi_{d-1}}+v_{\varpi_{d}}\right)\subset\mathbb{G}\cdot v_{\varpi_{d}'}$.
	\end{enumerate}
	
	Moreover in each case, the projectivizations of the $G$-orbit and the $\mathbb{G}$-orbit have the same dimension, in particular the two  projective varieties defined as the closure of these two orbits in the corresponding projective spaces are the same.
	\end{lem}

\begin{rem}
\leavevmode
\begin{enumerate}
\item In the first case of Lemma~\ref{lem:gather}, with $\tau=1$ we have in particular that, for $d$ even, $V_{\operatorname{Sp}_{d}}(\varpi_1)=V_{\operatorname{SL}_{d}}(\varpi_1)$. Note also that $\operatorname{Sp}_{d}/P(\varpi_1)=\operatorname{SL}_{d}/P(\varpi_1)(=\Pbb^{d-1})$.
\item Cases (c), (d) and (e) correspond to the triples of Definition~\ref{def:smoothtriple} that are not of two-orbit type.
\end{enumerate}
\end{rem}

\begin{proof}
The first two items are easy and left to the reader. The last three items are given in \cite[Propositions~1.8,~1.9 and 1.10]{2orbits}.
\end{proof}

	In Case (2), we need the following generalization of Lemma~\ref{lem:gather}.

\begin{lem}\label{lem:gatherbis}
Let $a\in\Nbb^*$.

Let $\tau\geq 0$. For $i\in\{0,\dots,\tau\}$, let $G_i$ be $\Cbb^*$, $\operatorname{SL}_{d_i}$ (with $d_i\geq 2$) or $\operatorname{Sp}_{d_i}$ (with $d_i\geq 2$ even). If $G_i=\Cbb^*$ set $d_i=1$ and $\varpi_1^i$ the identity character of $\Cbb^*$. Else set $\varpi_1^i$ the first fundamental weight of $G_i$. Let $G=G_0\times\cdots\times G_\tau$. 
\begin{enumerate}[label=(\alph*)]
	\item  Let $\mathbb{G}=\operatorname{SL}_{d}$ where $d=d_0+\cdots+d_\tau$.
	Then $$V_{\mathbb{G}}(a\varpi_1)=\bigoplus_{b_0,\dots,b_\tau}V_G(\sum_{i=0}^{\tau}b_i\varpi_1^i),$$  where  the sum is taken over all $(\tau+1)$-tuples of non-negative integers $(b_0,\dots,b_\tau)$ such that $\sum_{i=0}^{\tau}b_i=a$.
	And $$G\cdot\left(\sum_{b_0,\dots,b_\tau}v_{\sum_{i=0}^\tau b_i\varpi_1^i}\right)\subset\mathbb{G}\cdot v_{a\varpi_1}.$$
	
		\item Let $\mathbb{G}=\operatorname{SL}_{d}$ where $d=d_0+\cdots+d_\tau+1$.
		Then $$V_{\mathbb{G}}(a\varpi_1)=\bigoplus_{b_0,\dots,b_\tau}V_G(\sum_{i=0}^{\tau}b_i\varpi_1^i),$$ where  the sum is taken over all $(\tau+1)$-tuples of non-negative integers $(b_0,\dots,b_\tau)$ such that $\sum_{i=0}^{\tau}b_i\leq a$. And $$G\cdot\left(\sum_{b_0,\dots,b_\tau}v_{\sum_{i=0}^\tau b_i\varpi_1^i}\right)\subset\mathbb{G}\cdot v_{a\varpi_1}.$$
		
\vskip\baselineskip
With notation of Bourbaki \cite{Bourbaki8} (we put primes to write differently fundamental weights of $\mathbb{G}$ from those of $G$).

		\item Let $G=\operatorname{SL}_d$ (with $d\geq 3$) and $\mathbb{G}=\operatorname{SO}_{2d}$. Then $$V_{\mathbb{G}}(a\varpi_1')=\bigoplus_{b=0}^aV_G(b\varpi_1+(a-b)\varpi_{d-1})\,\mbox{ and }\,G\cdot\left(\sum_{b=0}^a v_{b\varpi_1+(a-b)\varpi_{d-1}}\right)\subset\mathbb{G}\cdot v_{a\varpi_1'}.$$
		
		\item
		Let $G=\operatorname{SL}_d$ (with $d\geq 4$), $\mathbb{G}=\operatorname{SL}_{d+1}$ and $1\leq i\leq d-2$. Then $$V_{\mathbb{G}}(a\varpi_{i+1}')=\bigoplus_{b=0}^aV_G(b\varpi_i+(b-a)\varpi_{i+1})\,\mbox{ and }\,G\cdot\left(\sum_{b=0}^a v_{b\varpi_i+(a-b)\varpi_{i+1}}\right)\subset\mathbb{G}\cdot v_{a\varpi_{i+1}'}.$$
		
		\item Let $G=\operatorname{Spin}_{2d}$ (with $d\geq 4$) and $\mathbb{G}=\operatorname{Spin}_{2d+1}$. Then $$V_{\mathbb{G}}(a\varpi_{d}')=\bigoplus_{b=0}^aV_G(b\varpi_{d-1}+(b-a)\varpi_{d})\,\mbox{ and }\, G\cdot\left(\sum_{b=0}^a v_{b\varpi_{d-1}+(b-a)\varpi_{d}}\right)\subset\mathbb{G}\cdot v_{a\varpi_{d}'}.$$
	\end{enumerate}
	Moreover in each case, the projectivizations the $G$-orbit and the $\mathbb{G}$-orbit have the same dimension, in particular the two  projective varieties defined as the closure of these two orbits in the corresponding projective spaces are the same.
\end{lem}

\begin{proof}
Remark that for $a=1$ the lemma is Lemma~\ref{lem:gather}. 
For any $a\geq 1$, we denote by $V_a$ the $G$-module that we consider in each case.

Consider the horospherical $G$-variety $X$ defined as the closure of the $G$-orbit of a sum $x_1$ of highest weight vectors in $\Pbb(V_1)$: it is a smooth projective variety with Picard group $\Zbb$ (it is isomorphic to $\Pbb^{d-1}$, $\Pbb^{d-1}$, the quadric $Q^{2d-2}$, the Grassmannian $\operatorname{Gr}(i+1,d+1)$, $\operatorname{Spin}(2d+1)/P(\varpi_d)$ respectively). Moreover $V_1^*$ is the $G$-module of global sections of $\mathcal{O}_X(1)$. And, for any $a\geq 1$, the $G$-module $V_a^*$ is the set of global sections of $\mathcal{O}_X(a)$. But, in each case, $X$ is also a homogeneous projective $\mathbb{G}$-variety $\mathbb{G}/P(\varpi)$ (with $\varpi=\varpi_1$, $\varpi_1$, $\varpi_1'$, $\varpi_{i+1}'$ and $\varpi_d'$ respectively) by Lemma~\ref{lem:gather}, then $V_a$ is also the irreducible $\mathbb{G}$-module $V_{\mathbb{G}}(a\varpi)$. 

Also, the image of $x_1$ in $\Pbb(V_a)$ is the projectivization of a highest weight vector in $V_{\mathbb{G}}(a\varpi)$  for a good choice of a Borel subgroup of $\mathbb{G}$ (because $\mathbb{G}\cdot x_1$ is the homogeneous projective $\mathbb{G}$-variety $\mathbb{G}/P(\varpi)$).
\end{proof}
	
\subsection{Proof of Theorem~\ref{th:main} in Case (1)} 
The first part is already proved by Proposition~\ref{prop:embeded} and Lemma~\ref{lem:WithSmoothness}, in particular $X$ is embedded as the closure of the $G$-orbit of a sum of highest weight vectors in $$\mathbb{P}:=\mathbb{P}\left(\bigoplus_{i=0}^nV(\varpi_{\alpha_i}+(1+a_i)\varpi_\beta)\right).$$

It remains to prove that we can suppose that 
\begin{itemize}[label=$\ast$]
\item $G_0$ is the universal cover of the automorphism group of $G_0/P(\varpi_\beta)$ if $R_0$ is empty;
\item if $i<j$ and $a_i=a_j$ then  $\alpha_j\in R_0$;
\item and some groups $G_k$ of type $C$ can be replaced by groups of type $A$. 
\end{itemize}

$\bullet$ If $R_0$ is empty and $G_0$ is not the universal cover of the automorphism group of $G_0/P(\varpi_\beta)$, then $G_0/P(\varpi_\beta)$ is isomorphic to $G_0'/P(\varpi_{\beta'})$ where $G_0'$ is the universal cover of $\operatorname{Aut} (G_0/P(\varpi_\beta))$ and $(G_0,\beta, G_0',\beta')$ is one of the following: $(\operatorname{Sp}_{2m},\varpi_1,\operatorname{SL}_{2m},\varpi_1)$, $(G_2,\varpi_1,\operatorname{Spin}_7,\varpi_1)$, $(\operatorname{Spin}_{2m+1},\varpi_m,\operatorname{Spin}_{2m+2},\varpi_m)$ or $(\operatorname{Spin}_{2m+1},\varpi_m,\operatorname{Spin}_{2m+2},\varpi_{m+1})$. In any case, $V_{G_0}(\varpi_\beta)\simeq V_{G_0'}(\varpi_{\beta'})$ and $G_0\cdot v_{\varpi_\beta}\simeq G_0'\cdot v_{\varpi_{\beta'}}$. Hence, the fact that $R_0$ is empty implies that  $\bigoplus_{i=0}^nV_G(\varpi_{\alpha_i}+(1+a_i)\varpi_\beta)\simeq \bigoplus_{i=0}^nV_\mathbb{G}(\varpi_{\alpha_i}+(1+a_i)\varpi_{\beta'})$ where $\mathbb{G}=G_0'\times G_1\times\cdots\times G_t$, and $X$ is isomorphic to the closure of the $\mathbb{G}$-orbit of a sum of highest weight vectors in $$\mathbb{P}:=\mathbb{P}\left(\bigoplus_{i=0}^nV_\mathbb{G}(\varpi_{\alpha_i}+(1+a_i)\varpi_{\beta'})\right).$$

$\bullet$ Suppose that there is $0\leq i<j\leq n$ such that $\alpha_i$ and $\alpha_j$ are simple roots of the same simple group among $G_1,\dots,G_t$. Then by Lemma~\ref{lem:WithSmoothness}, we have
$n=1$, $i=0$, $j=1$ (also $t=1$) and the triple $(G_1,\alpha_i,\alpha_j)$ is smooth. In particular, $X$ is embedded as the closure of the $G$-orbit of a sum of  highest weight vectors in $$\mathbb{P}\left(V(\varpi_{\alpha_0}+\varpi_\beta)\oplus V(\varpi_{\alpha_1}+(1+a_1)\varpi_\beta)\right).$$ If $a_1=0$, the $G$-module $V(\varpi_{\alpha_0}+\varpi_\beta)\oplus V(\varpi_{\alpha_1}+(1+a_1)\varpi_\beta)$ is isomorphic to the tensor product of the $G_0$-module $V(\varpi_\beta)$ by the $G_1$-module $V(\varpi_{\alpha_0})\oplus V(\varpi_{\alpha_1})$, so that $X$ is the product of $G/P(\varpi_\beta)$ by the smooth projective horospherical variety of Picard group $\Zbb$ defined as the closure of the $G_1$-orbit of a sum of  highest weight vectors in $\mathbb{P}(V(\varpi_{\alpha_0})\oplus V(\varpi_{\alpha_1})).$

We conclude that if $X$ is not a product, $X$ is as in Case (1a) (with $a_1>0$).

\vskip\baselineskip
From now on, we suppose that there is no $0\leq i<j\leq n$ such that $\alpha_i$ and $\alpha_j$ are simple roots of the same simple group among  $G_1,\dots,G_t$.

\vskip\baselineskip
$\bullet$ Suppose that there exists $0\leq i<j\leq n$ such that $a_i=a_j$ and both $\alpha_i$ and $\alpha_j$ are not simple roots of $G_0$. 

Up to reordering, assume that $\alpha_i$ and $\alpha_j$ are simple roots of $G_{1}$ and $G_{2}$ ($t\geq 2$). Note that if $i=0$ and $\alpha_0$ is trivial, $G_1=\{1\}$.
By Lemma~\ref{lem:WithSmoothness}, $G_{1}$ and $G_2$ are $\{1\}$, $\Cbb^*$ ($d_k=1$ in these two cases), $\operatorname{SL}_{d_k}$ (with $d_k\geq 2$) or $\operatorname{Sp}_{d_k}$ (with $d_k\geq 2$ even) and $\alpha_i$, respectively $\alpha_j$, is either a trivial root or a short extremal root of $G_{1}$, respectively $G_2$. 

Let $\mathbb{G}=G_0\times G_3\times \cdots\times G_{t}\times \operatorname{SL}_{d_{1}+d_2}$. By Lemma~\ref{lem:gather} ((a) if $i>0$ or $\alpha_0$ is not trivial and (b) otherwise), the $G$-module $V(\varpi_{\alpha_i}+(1+a_i)\varpi_\beta)\oplus V(\varpi_{\alpha_j}+(1+a_j)\varpi_\beta)$ is isomorphic to the  $\mathbb{G}$-module $ V((1+a_i)\varpi_\beta)\otimes\Cbb^{d_{1}+d_2}$. And $X$ is a subvariety of the closure $\mathbb{X}$ of the $\mathbb{G}$-orbit of a sum of  highest weight vectors in $\mathbb{P}$ under the action of $\mathbb{G}$. 

We can now compare the dimension of the open $\mathbb{G}$-orbit $\Omega_\mathbb{X}$ of $\mathbb{X}$ with the dimension of the open $G$-orbit of $X$. Indeed $\Omega_\mathbb{X}$ is isomorphic to a horospherical homogeneous space of rank~$n-1$ over $((G_0\times G_3\times\cdots\times G_{t})/(P\cap G_0\times G_3\times\cdots\times G_{t}))\times (\operatorname{SL}_{d_{1}+d_2}/P(\varpi_1))$, while $G/H$ is of rank~$n$ over $((G_0\times G_3\times\cdots\times G_{t})/P\cap (G_0\times G_3\times\cdots\times G_{t}))\times ((G_{1}\times G_2)/P\cap (G_{1}\times G_2))$. But the dimension of $\operatorname{SL}_{d_{1}+d_2}/P(\varpi_1)$ is $d_{1}+d_2-1$ while the dimension of $(G_{1}\times G_2)/P\cap (G_{1}\times G_2)$ is $(d_{1}-1)+(d_2-1)$. Hence $\Omega_\mathbb{X}$ and $G/H$ have the same dimension, so that $X=\mathbb{X}$.

Then we can replace, without changing $X$, the product of the two simple groups corresponding to two simple roots $\alpha_i$ and $\alpha_j$ with $a_i=a_j$, with a unique simple group of type $A$. Note that $n$ decreases by this change. (Also note that, if $i=0$ and $\alpha_0$ is trivial then the new $\alpha_0$ is not trivial any more.)

With similar arguments, we can also replace any group $G_1,\dots,G_t$, of type $C$ and that contains a unique simple root $\alpha_i$, by a group of type $A$.

\vskip\baselineskip
$\bullet$ What we did just above also works in the cases where $n=1$, $a_1=0$, $\alpha_0$ and $\alpha_1$ are simple roots of $G_1$ and $G_2$ (and $t=2$). In that case, this proves that $X$ is the closure of the $\operatorname{SL}_{d}\times G_0$-orbit of a  highest weight vector in $\mathbb{P}\left(\Cbb^{d}\bigotimes V(\varpi_\beta)\right).$ Hence, in that case, $X$ is isomorphic to $\Pbb^{d-1}\times G_0/P(\varpi_\beta)$.

Hence, we conclude the proof by iteration.

\subsection{Proof of Theorem~\ref{th:main} in Case (2)}\label{sect43}

The first part is already proved by Proposition~\ref{prop:embeded} and Lemma~\ref{lem:WithSmoothness}, in particular $X$ is embedded as the closure of the $G$-orbit of a sum of  highest weight vectors in $$\mathbb{P}:=\mathbb{P}\left(\bigoplus_{i,b_1,\dots,b_{s+1}}V(\varpi_{\alpha_i}+\sum_{j=1}^{s+1}b_j\varpi_{\alpha_{r+j}})\right),$$ where the sum is taken over all $s+2$-tuples of non-negative integers $(i,b_1,\dots,b_{s+1})$ such that $0\leq i\leq r$ and $\sum_{j=1}^{s+1}b_j=1+a_i$.

It remains to prove that we can suppose that 
\begin{itemize}[label=$\ast$]
\item $s=1$, $\alpha_{n},\,\alpha_{n+1}$ are both simple roots of $G_t$ and $(G_t,\alpha_{n},\alpha_{n+1})$ is smooth of two-orbit type;
\item $0<a_1<\cdots<a_r$;
\item and some groups $G_k$ of type $C$ can be replaced by groups of type $A$. 
\end{itemize}

$\bullet$ Suppose first that $s>1$, or $s=1$ and $\alpha_{n}$, $\alpha_{n+1}$ are not simple roots of the same simple group $G_k$.  Up to reordering and applying Lemma~\ref{lem:WithSmoothness}, for any $j\in\{1,\dots,s\}$, $\alpha_{r+j}$ is either a trivial root of $G_{t-s+j}$ that is $\Cbb^*$ or $\{1\}$, or a short extremal simple root of $G_{t-s+j}$ that is of type $A$ or $C$. Moreover, the simple groups $G_{t-s+1},\dots,G_{t}$ contain no other $\alpha_i$ with $i\in\{0,\dots,r\}$. Also, $G_{t-s+j}=\{1\}$ if and only if $j=s$ and $\alpha_{r+s}$ is trivial.

We now apply Lemma~\ref{lem:gatherbis} ((a) if $\alpha_{r+s}$ is not trivial and (b) otherwise). Hence, there exists $d\leq 2$ such that, with  $G\subset\mathbb{G}:=G_{0}\times\cdots\times G_{t-s}\times \operatorname{SL}_d$, we have  $$\mathbb{P}=\mathbb{P}\left(\bigoplus_{i,b_1,\dots,b_{s+1}}V(\varpi_{\alpha_i})\otimes V(\sum_{j=1}^{s+1}b_j\varpi_{\alpha_{r+j}})\right)=\mathbb{P}\left(\bigoplus_{i=0}^rV_\mathbb{G}(\varpi_{\alpha_i})\otimes V_\mathbb{G}((1+a_i)\varpi_1)\right),$$ $X$ is a subvariety of the closure $\mathbb{X}$ of the $\mathbb{G}$-orbit $\Omega_\mathbb{X}$ of a sum of  highest weight vectors  in $\Pbb$, and $\dim((G_{t+1-s}\times\cdots\times G_{t})/P\cap(G_{t+1-s}\times\cdots\times G_{t})=d-s-1$. In particular the dimension of $\Omega_\mathbb{X}$ (which is horospherical of rank $r$) equals the dimension of $G/H$. Hence, $X=\mathbb{X}$. Now remark that $\mathbb{X}$ is a horospherical variety as in Case (1).

\vskip\baselineskip
$\bullet$ From now on, we suppose that $s=1$ (and $n=r+1$), and that $\alpha_{n}$, $\alpha_{n+1}$ are both simple roots of $G_t$ (up to reordering). In particular, $X$ is of type $\mathbb{X}^2(G,\underline{\alpha},\underline{a})$ and then is embedded as the closure of the $G$-orbit of a sum of  highest weight vectors in $$\mathbb{P}\left(\bigoplus_{i=0}^{n-1}\bigoplus_{b=0}^{1+a_i}V(\varpi_{\alpha_i}+b\varpi_{\alpha_{r+1}}+(1+a_i-b)\varpi_{\alpha_{r+2}})\right).$$  Note now that for any $k\in\{0,\dots,t\}$, $G_k=\{1\}$ if and only if $k=0$ and $\alpha_0$ is trivial.

\vskip\baselineskip
Recall that, by Lemma~\ref{lem:WithSmoothness}, $\alpha_0,\dots,\alpha_{r}$ are not simple roots of $G_t$ and the triple $(G_t,\alpha_{n},\alpha_{n+1})$ is smooth. Then $X$ is embedded as the closure of the $G$-orbit of a sum of  highest weight vectors in $$\mathbb{P}:=\mathbb{P}\left(\bigoplus_{i=0}^{n-1}\bigoplus_{b=0}^{1+a_i}V(\varpi_{\alpha_i})\otimes V(b\varpi_{\alpha_{r+1}}+(1+a_i-b)\varpi_{\alpha_{r+2}})\right).$$
 
If $(G_t,\alpha_{n},\alpha_{n+1})$ is not of two-orbit type, we can apply Lemma~\ref{lem:gatherbis} ((c), (d) or (e)): we get $G\subset \mathbb{G}$ with $\mathbb{G}:=G_0\times\cdots\times G_{t-1}\times \mathbb{G}_t$ such that $\mathbb{P}=\mathbb{P}\left(\bigoplus_{i=0}^rV_\mathbb{G}(\varpi_{\alpha_i})\otimes V_\mathbb{G}((1+a_i)\varpi)\right)$, $X$ is a subvariety of the closure  $\mathbb{X}$ of the $\mathbb{G}$-orbit $\Omega_\mathbb{X}$ of a sum of  highest weight vectors  in $\Pbb$, and $\dim(G_t/P\cap G_t)+1=\dim (\mathbb{G}_t/P(\varpi))$. In particular the dimension of  $\Omega_\mathbb{X}$ (which is horospherical of rank $r$) equals the dimension of $G/H$. Hence, $X=\mathbb{X}$. And remark that $\mathbb{X}$ is a horospherical variety as in Case (1).

\vskip\baselineskip
$\bullet$ Now suppose that $r>1$, or $r=1$ and $\alpha_0$, $\alpha_1$ are not simple roots of the same simple group.

Let $i\neq i'$ in $\{0,\dots,r\}$ such that $a_i=a_{i'}$.
Up to reordering and applying Lemma.~\ref{lem:WithSmoothness}, $\alpha_i$ and $\alpha_{i'}$ are, trivial or short extremal, simple roots respectively of $G_0$ and $G_1$ that are $\Cbb^*$, $\{1\}$ or simple groups of type $A$ or $C$. Moreover $G_0$ and $G_1$ contain no other $\alpha_k$'s.

We can apply Lemma~\ref{lem:gather} ((a) if $i>0$ or $\alpha_0$ is trivial and (b) otherwise) to get $G\subset \mathbb{G}:=\operatorname{SL}_d\times G_2\cdots\times G_{t}$ such that 
\begin{multline*}
\mathbb{P}=\mathbb{P}\left(\left(\bigoplus_{k\neq i,\,i'}\bigoplus_{b=0}^{1+a_k}V_\mathbb{G}(\varpi_{\alpha_k})\otimes V_\mathbb{G}(b\varpi_{\alpha_{r+1}}+(1+a_k-b)\varpi_{\alpha_{r+2}})\right)\right.\\
\oplus \left.\left(\bigoplus_{b=0}^{1+a_i}V_\mathbb{G}(\varpi)\otimes V_\mathbb{G}(b\varpi_{\alpha_{r+1}}+(1+a_i-b)\varpi_{\alpha_{r+2}})\right)\right),
\end{multline*} 

$X$ is a subvariety of the closure $\mathbb{X}$ of the $\mathbb{G}$-orbit $\Omega_\mathbb{X}$ of a sum of  highest weight vectors  in $\Pbb$, and $\dim((G_0\times G_1)/P\cap( G_0\times G_1))+1=d-1$. In particular the dimension of  $\Omega_\mathbb{X}$ (which is horospherical of rank $(r-1)+1$) equals the dimension of $G/H$. Hence, $X=\mathbb{X}$. Now remark that $\mathbb{X}$ is either a horospherical variety as in Case (2) of rank one less than $X$, or a horospherical variety as in Case (1) if $r=1$.

With similar arguments, we can also replace any group $G_0,\dots,G_{t-1}$, of type $C$ and that contains a unique simple root $\alpha_i$, by a group of type $A$.

\vskip\baselineskip
$\bullet$ By iteration of the above process, we can now assume that $0<a_1<\cdots<a_r$, or that $r=1$ (and $t=1$) and $\alpha_0$, $\alpha_1$ are two simple roots of $G_0$. In the second case, note that by Lemma.~\ref{lem:WithSmoothness}, the triple $(G_0,\alpha_0,\alpha_1)$ is smooth.  

Suppose $r=1$, $\alpha_0$, $\alpha_1$ are two simple roots of $G_0$ and that $a_1=a_0=0$. Then, $X$ is the closure of the $G_0\times G_1$-orbit of a sum of  highest weight vectors  in $$\mathbb{P}=\mathbb{P}\left((V_{G_0}(\varpi_{\alpha_{0}})\oplus V_{G_0}(\varpi_{\alpha_{1}}))\otimes (V_{G_1}(\varpi_{\alpha_{2}})\oplus V_{G_1}(\varpi_{\alpha_{3}}))\right).$$ 
Hence in that case, $X$ is the product of two varieties: the closure of the $G_0$-orbit of a sum of  highest weight vectors  in $\mathbb{P}\left((V_{G_0}(\varpi_{\alpha_{0}})\oplus V_{G_0}(\varpi_{\alpha_{1}}))\right)$
 and the closure of the $G_1$-orbit of a sum of  highest weight vectors  in $\mathbb{P}\left((V_{G_1}(\varpi_{\alpha_{2}})\oplus V_{G_1}(\varpi_{\alpha_{3}}))\right)$.
 
 Hence, in any case we can assume that $0<a_1<\cdots<a_r$. This finishes the proof of Theorem~\ref{th:main}.\hfill$\qed$

\section{The MMP and Log MMP for smooth projective horospherical varieties of Picard group $\Zbb^2$}\label{sec:MMP}
	
The main goal of this section is to prove Theorem~\ref{th:main2}. For this we apply the Log MMP from the horospherical varieties $\mathbb{X}^1$ and $\mathbb{X}^2$. 
	
The principle of the Log MMP is the following. We begin with a pair $(X,\Delta)$ where $X$ is a not too singular projective variety and $\Delta$ is a $\Qbb$-divisor such that $K_X+\Delta$ is $\Qbb$-Cartier. We want to contract curves having negative intersection with $K_X+\Delta$ in order to get a new variety with smaller Picard number. In general, we can do this by choosing an extremal ray (whose curves have negative intersection with $K_X+\Delta$) in the cone of effective curves up to numerical equivalence. 
	 
	 In our context, note that this cone is two dimensional and then has two extremal rays; this explains why we have two ways to do the Log MMP. 
	 
	 After contracting a curve it may happen that the new variety is too singular, so that we have to partially desingularize it in a natural and unique way; we call this a flip. 
	 
	 To continue the program, we have to choose again an extremal ray in the cone of effective curves of the new variety, until we finish with a minimal model (when there is no curve with negative intersection with $K_X+\Delta$) or a fibration (when the dimension decreases).
	 
\vskip\baselineskip
For horospherical varieties, we can compute a Log MMP to the end just by choosing an ample divisor at the beginning (and not an extremal ray at each step), and by considering a one-parameter family of polytopes (Theorem~\ref{th:recallMMP}).

\subsection{Generalities}	
	
Let $X$ be a smooth projective horospherical variety with Picard group $\Zbb^2$. Here, we suppose that $X$ is as in Case (1) or (2) of Lemma~\ref{lem:FirstReduction} (or Theorem~\ref{th:main}).

 By Proposition~\ref{prop:NefCone}, up to linear equivalence, the ample Cartier divisors of $X$ are of the form $D=d_0D_0+d_{n+1}D_{n+1}$ with positive integers $d_0$ and $d_{n+1}$. 

We can apply \cite{MMPhoro} to the polarized variety $(X,D)$ and obtain a description of the MMP from $X$, via moment polytopes (if $X$ is Fano, we obtain two different paths of the program depending on the choice of $d_0$ and $d_{n+1}$; if $X$ is not Fano, we obtain a unique path of the program). 

Moreover, we can also choose a $B$-stable $\Qbb$-divisor $\Delta$ of $X$ and apply  \cite{LogMMPhoro} to the polarized pair $((X,D),\Delta)$ and obtain a description of the Log MMP from $(X,\Delta)$, via moment polytopes as described in Section~\ref{sec:recallMMP}. To get a uniform Log MMP for any smooth projective horospherical variety with Picard group $\Zbb^2$, we choose $D=D_0+D_{n+1}$ and $\Delta=-D_i-K_X$ for $i\in\{0,n+1\}$.

\begin{rem}
	In Case (1), an anticanonical divisor of $X$ is (see for example \cite[Proposition 3.1]{Fanohoro}) $$-K_X=\sum_{i=0}^{n}b_iD_i+b_{\beta}D_{\beta}\sim (\sum_{i=0}^{n}b_i)D_0 +(b_{\beta}-\sum_{i=0}^{n}a_ib_i)D_{n+1},$$ where $b_i=1$ if $D_i$ is $G$-stable, $b_i=b_{\alpha_i}\geq 2$ if $D_i$ is the color $D_{\alpha_i}$ and $b_\beta\geq 2$ (recall that $D_\beta=D_{n+1}$).  In particular, $X$ is Fano (i.e., $-K_X$ ample) if and only if $b_{\beta}>\sum_{i=1}^{n}a_ib_i$.
	
	To describe the MMP from $X$ we could choose the ample divisor $D=(\sum_{i=0}^{n}b_i)D_1+(b_{\beta}+1)D_\beta$, so that $D+\epsilon K_X$ is ample for any $\epsilon\in[0,1[$ and $D+K_X\sim (\sum_{i=0}^{n}a_ib_i+1)D_\beta$ is not ample but globally generated. Then, for that choice of $D$, the MMP from $X$ consists of the Mori fibration to $G/P(\varpi_\beta)$ described in Remark~\ref{rem:debutMMP}.\\ 
	Moreover, this Mori fibration is also the unique contraction of the Log MMP obtained with the choices  $D=D_0+D_{n+1}$ and $\Delta=-D_0-K_X$ in Theorem~\ref{th:recallMMP} (in that case, $Q^1$ is a multiple of $\varpi_\beta$).
	
\vskip\baselineskip
In Case (2), an anticanonical divisor of $X$ is $$-K_X=\sum_{i=0}^{r}b_iD_i+\sum_{j=1}^{s+1}b_{r+j}D_{r+j}\sim (\sum_{i=0}^{r}b_i)D_0 +(\sum_{j=1}^{s+1}b_{r+j}-\sum_{i=0}^{r}a_ib_i)D_{n+1},$$ where $b_i=1$ (respectively $b_{r+j}$) if $D_i$ (resp. $D_{r+j}$) is $G$-stable and $b_i=b_{\alpha_i}\geq 2$ (resp. $b_{r+j}=b_{\alpha_{r+j}}\geq 2$) if $D_i$ is the color $D_{\alpha_i}$ (respectively $D_{r+j}$ is the color $D_{\alpha_{r+j}}$). In particular, $X$ is Fano if and only if the inequality $\sum_{j=1}^{s+1}b_{r+j}>\sum_{i=0}^{r}a_ib_i$ is s.
	
To describe the MMP from $X$ we could choose the ample divisor $D=(\sum_{i=0}^{r}b_i)D_0+(1+\sum_{j=1}^{s+1}b_{r+j})D_{n+1}$, so that $D+\epsilon K_X$ is ample for any $\epsilon\in[0,1[$ and $D+K_X\sim (1+\sum_{i=0}^{r}a_ib_i)D_{n+1}$ is not ample but globally generated. Then, for that choice of $D$, the MMP from $X$ consists of the Mori fibration $\psi$ from $X$ to $Z$ described in Remark~\ref{rem:debutMMP}. Moreover, this Mori fibration is also the unique contraction of the Log MMP obtained with the choices  $D=D_0+D_{n+1}$ and $\Delta=-D_0-K_X$ in Theorem~\ref{th:recallMMP} (in that case, $Q^1$ is a simplex of dimension $s$).
\end{rem}
	
	Hence, in both cases, we will describe the Log MMP obtained with the choices $D=D_0+D_{n+1}$ and $\Delta=-D_{n+1}-K_X$.
	
\vskip\baselineskip
In the next four subsections, $X$ is one the varieties of Theorem~\ref{th:main} in Case (1) or (2).  We begin by constructing the families of polytopes for the log pairs $(X,\Delta=-D_{n+1}-K_X)$ with the choice of ample divisor $D=D_0+D_{n+1}$, and then we describe in detail the Log MMP's obtained with these families.

	\subsection{Case (1): the "second" Log MMP via moment polytopes}\label{LogMMP1}

To describe the one-parameter family $(\tilde{Q}^\epsilon)_{\epsilon\in\Qbb_{\geq 0}}$ defined in Theorem~\ref{th:recallMMP}, we consider the basis $(e_i^*)_{i\in\{1,\dots,n\}}$ of $M$, where for any $i\in\{1,\dots,n\}$, $e_i^*=\varpi_{\alpha_i}-\varpi_{\alpha_0}+a_i\varpi_\beta$, and we define the matrices $\mathcal{A}$, $\mathcal{B}$ and $\mathcal{C}$ as follows 
$$\mathcal{A}= \left(
\begin{array}{cccc}
-1 & \cdots & \cdots & -1\\
1 & 0 &\cdots&0\\
0 & \ddots & \ddots & \vdots\\
\vdots & \ddots & \ddots & 0\\
0 & \cdots & 0 & 1\\
a_1 & \cdots & \cdots & a_n
  \end{array}
\right),\,\,\mathcal{B}= \left(
\begin{array}{c}
-1\\
0 \\
\vdots \\ \vdots \\
0\\

-1
  \end{array}
\right)\mbox{ and }\,\mathcal{C}= \left(
\begin{array}{c}
0 \\
\vdots \\ \vdots \\
0\\
0\\
1
  \end{array}
\right).$$

Then $\tilde{Q}^\epsilon=\{x\in M_\Qbb\,\mid\,\mathcal{A}x\geq \mathcal{B}+\epsilon \mathcal{C}\}$ is the set of $x=(x_1,\dots,x_n)$ such that $x_1,\dots,x_n$ are non-negative, $x_1+\cdots+x_n\leq 1$ and $a_1x_1+\cdots+a_nx_n\geq \epsilon-1$.

\begin{ex}
If $n=2$ we are in one of the following situations: 

\begin{enumerate}

\item $a_2>a_1>0$ and $\alpha_2$ is not trivial;

\item $a_2>a_1>0$ and $\alpha_2$ is  trivial;

\item $a_2>a_1=0$ and $\alpha_2$ is not trivial;

\item $a_2>a_1=0$ and $\alpha_2$ is  trivial;

\item $a_2=a_1>0$;

\item $a_2=a_1=0$.
\end{enumerate}

We draw, in Figure~\ref{figure1}, these polytopes for $\epsilon=0$ in different cases with the hyperplane $H^0:=\{x\in M_\Qbb\,\mid\, a_1x_1+a_2x_2=-1\}$. Note that there is no such hyperplane if $a_2=a_1=0$. 

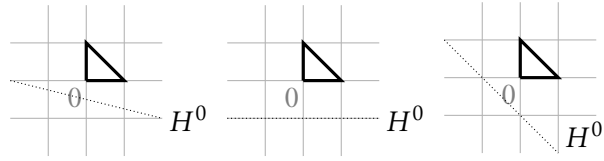
\begin{figure}

\begin{center}
\begin{tikzpicture}[scale=0.5]
\chadom
\draw[very thick ] (2,2) -- (2,3) -- (3,2) -- (2,2);
%\node at (3.5,3.5) {$\tilde{Q}^0$};
\draw[densely dotted] (0,2) -- (4,1) ;
\node at (4.7,1) {$H^0$};
\end{tikzpicture}
\begin{tikzpicture}[scale=0.5]
\chadom
\draw[very thick] (2,2) -- (2,3) -- (3,2) -- (2,2);
%\node at (3.5,3.5) {$\tilde{Q}^0$};
\draw[densely dotted] (0,1) -- (4,1) ;
\node at (4.7,1) {$H^0$};
\end{tikzpicture}
\begin{tikzpicture}[scale=0.5]
\chadom
\draw[very thick] (2,2) -- (2,3) -- (3,2) -- (2,2);
%\node at (3.5,3.5) {$\tilde{Q}^0$};
\draw[densely dotted] (0,3) -- (3,0) ;
\node at (3.7,0.5) {$H^0$};
\end{tikzpicture}
\caption{The polytopes $\tilde{Q}^0$ in the cases where $a_1=1$ and $a_2=2$, $a_1=0$ and $a_2=1$ and $a_1=a_2=1$ respectively}
\label{figure1}
\end{center}
\end{figure}

\end{ex}

$\bullet$ If $a_n=0$, $\tilde{Q}^\epsilon=\tilde{Q}^0$ for any $\epsilon\in [0,1]$ and it is empty if $\epsilon>1$. Moreover, for any $\epsilon\in [0,1]$, $Q^\epsilon$ intersects the interior of $\mathfrak{X}(P)^+_\Qbb$ if and only if $\epsilon<1$. 
In that case, the Log MMP described by the family $(Q^\epsilon)_{\epsilon\in\Qbb_{\geq 0}}$ consists of a fibration $\phi_0:\,X\longrightarrow Y^0$. 

The fibers of this fibration can be easily computed (by the strategy given in Section~\ref{sec:recallMMP}) because the faces of $Q^0$ are ``the same'' as the faces of $Q^1$ and then the fibration induces a bijection between the sets of $G$-orbits of $X$ and $Y^0$. Then the fibers of $\phi_0$  are isomorphic to the homogeneous projective spaces $(\bigcap_{i\in I}P(\varpi_{\alpha_i}))/(P(\varpi_\beta)\cap\bigcap_{i\in I}P(\varpi_{\alpha_i}))$ (of Picard group $\Zbb$), with $\emptyset\neq I\subset\{0,\dots,n\}$. Here, we use the following notation: if $\alpha_i$ is trivial, $P(\varpi_{\alpha_i})=G$ (and otherwise, it is the proper maximal parabolic subgroup of $G$ associated to $\alpha_i$).

 In particular, the general fiber of the fibration is $(\bigcap_{i=0}^nP(\varpi_{\alpha_i}))/(P(\varpi_\beta)\cap\bigcap_{i=0}^nP(\varpi_{\alpha_i}))$ and the smallest fibers are the $P(\varpi_{\alpha_i})/(P(\varpi_\beta)\cap P(\varpi_{\alpha_i}))$ with $i\in\{0,\dots,n\}$. Then we deduce that $\alpha_0\not\in R_0$ if and only if there exists a fiber isomorphic to $G/P(\varpi_\beta)$.

\vskip\baselineskip
$\bullet$ Suppose now that $a_n\neq 0$, then $\tilde{Q}^\epsilon$ is the intersection of the simplex $\tilde{Q}=\operatorname{Conv}(e_0^*,e_1^*,\dots,e_n^*)$ with the closed half-space $H_+^\epsilon:=\{x\in M_\Qbb\,\mid\, a_1x_1+\cdots+a_nx_n\geq \epsilon-1\}$, where $e_0^*:=0$. We denote by  $H_{++}^\epsilon$ the interior of $H_{+}^\epsilon$ and by $H^\epsilon$ the hyperplane $H_{+}^\epsilon\backslash H_{++}^\epsilon$.

In the next proposition, we give a description of the non-empty faces of $\tilde{Q}^\epsilon$ by distinguishing whether  a face is in the hyperplane $H^\epsilon$ or not.

Note first that the non-empty faces of the simplex $\tilde{Q}$ are the $F_I:=\operatorname{Conv}(e_i^*\,\mid\, i\in\{0,\dots,n\}\backslash I)$, with $I\subsetneq \{0,\dots,n\}$. In particular, the facets of $\tilde{Q}$ are the $F_i:=F_{\{i\}}$ and for any 
$I\subsetneq \{0,\dots,n\}$, $F_I=\bigcap_{i\in I}F_i$.

Then, for any $I\subsetneq \{0,\dots,n\}$, we define $F_I^\epsilon:=F_I\cap H_+^\epsilon$ and $F_{I,\beta}^\epsilon:=F_I\cap H^\epsilon$. They are faces (may be empty and not distinct) of $\tilde{Q}^\epsilon$.

\begin{prop}[recall that $a_0=0$ and that $a_n\neq 0$ here]\label{prop:polyfaces1}
\leavevmode

The polytope $\tilde{Q}^\epsilon$ is of dimension $n$ if and only if $\epsilon<\max_{i=0}^n(1+a_i)=1+a_n$.

\vspace{0.1cm}

Suppose now that $\epsilon<1+a_n$.
The non-empty faces of $\tilde{Q}^\epsilon$ are the distinct $F_I^\epsilon$ and $F_{I,\beta}^\epsilon$ $($with $I\subsetneq \{0,\dots,n\})$ defined as follows:
\begin{itemize}[label=$\ast$]
\item $F_I^\epsilon$ (of codimension $|I|$) if $\epsilon<\max_{i\not\in I}(1+a_i)$;
\item $F_{I,\beta}^\epsilon$ (of codimension $|I|+1$ or $|I|$ respectively) if $\min_{i\not\in I}(1+a_i)< \epsilon< \max_{i\not\in I}(1+a_i)$ or $\epsilon=\min_{i\not\in I}(1+a_i)=\max_{i\not\in I}(1+a_i)$. 
\end{itemize}
In particular, the facets of $\tilde{Q}^\epsilon$ are: $F_i^\epsilon$ with $i\in\{0,\dots,n-1\}$ (for any $\epsilon<1+a_n$), $F_n^\epsilon$ if $\epsilon<1+a_{n-1}$,
%if there exists $j\neq i$ such that $\epsilon<1+a_j$,  
$F_{\emptyset,\beta}^\epsilon$ if $\epsilon>1$, and $F_{n,\beta}^\epsilon$ if $\epsilon=1$ and $a_{n-1}=0$. 

Moreover, we can write any face of $\tilde{Q}^\epsilon$ as the intersection of all the facets that contain it, as follows. \\
For any $I\subsetneq \{0,\dots,n\}$ such that $\epsilon<\max_{i\not\in I}(1+a_i)$, $F_I^\epsilon=\bigcap_{i\in I}F_i^\epsilon$.\\  
For any $I\subsetneq \{0,\dots,n\}$ such that $\min_{i\not\in I}(1+a_i)< \epsilon< \max_{i\not\in I}(1+a_i)$, $F_{I,\beta}^\epsilon= F_{\emptyset,\beta}^\epsilon\cap\bigcap_{i\in I}F_i^\epsilon$.\\
For any $I\subsetneq \{0,\dots,n\}$ such that $\epsilon=\min_{i\not\in I}(1+a_i)=\max_{i\not\in I}(1+a_i)$, $F_{I,\beta}^\epsilon= F_{n,\beta}^\epsilon\cap\bigcap_{i\in I} F_i^\epsilon$ if $\epsilon=1$, $n\in I$ and $a_{n-1}=0$ or $F_{I,\beta}^\epsilon= \bigcap_{i\in I}F_i^\epsilon$ if $\epsilon\neq 1$, $n\not\in I$ or $a_{n-1}\neq 0$. 
\end{prop}

\begin{proof}
The polytope $\tilde{Q}^\epsilon$ is of dimension $n$ if and only if $\tilde{Q}$ intersects $H_{++}^\epsilon$ if and only if there exists $i\in\{0,\dots,n\}$ such that $e_i^*\in H_{++}^\epsilon$ if and only if there exists $i\in\{0,\dots,n\}$ such that $a_i>\epsilon-1$ if and only if $a_n>\epsilon-1$ (because $0=a_0\leq\cdots\leq a_n$). This proves the first statement of the proposition.

\vspace{0.1cm}

Suppose now that $\epsilon<1+a_n$.
For any non-empty face $F$  of $\tilde{Q}^\epsilon$, either $F\not\subset H^\epsilon$ and $F$ is the intersection of a non-empty face of $\tilde{Q}$ with $H_+^\epsilon$,  or $F\subset H^\epsilon$ and $F$ is the intersection of a non-empty face of $\tilde{Q}$ with $H^\epsilon$. 

Let $I\subsetneq \{0,\dots,n\}$. The set $F_I^\epsilon$ is not empty if and only if there exists $i\not\in I$ such that $e_i^*\in H_{+}^\epsilon$ if and only if there exists $i\not\in I$ such that $a_i\geq \epsilon-1$ if and only if $\epsilon\leq\max_{i\not\in I}(1+a_i)$. Moreover, $F_I^\epsilon$ is not empty and not included in $H^\epsilon$ if and only if it intersects $H_{++}^\epsilon$ if and only if there exists $i\not\in I$ such that $e_i^*\in H_{++}^\epsilon$ if and only if there exists $i\not\in I$ such that $a_i>\epsilon-1$ if and only if $\epsilon<\max_{i\not\in I}(1+a_i)$. Also, in that latter case, the dimension of $F_I^\epsilon$ is the same as the dimension of $F_I$; in particular the non-empty $F_I^\epsilon$ that are not included in $H^\epsilon$ are all distinct.

Similarly, $F_{I,\beta}^\epsilon$ is not empty if and only if there exist $i$ and $j$ not in $I$ (may be equal) such that $e_i^*\in H_{+}^\epsilon$ and $e_j^*\not\in H_{++}^\epsilon$ (i.e., $a_i\geq \epsilon-1$ and $a_j\leq \epsilon-1$). Then  $F_{I,\beta}^\epsilon$ is not empty if and only $\min_{i\not\in I}(1+a_i)\leq \epsilon\leq\max_{i\not\in I}(1+a_i)$. Moreover, $F_{I,\beta}^\epsilon$ is not empty and included in no proper face of $F_I$ (i.e., $H^\epsilon$ intersects the relative interior of $F_I$) if and only if there exist $i\neq j$ not in $I$ such that $e_i^*\in H_{++}^\epsilon$ and $e_j^*\not\in H_{+}^\epsilon$ (i.e., $a_i> \epsilon-1$ and $a_j< \epsilon-1$) or for any $i\not\in I$ we have $e_i^*\in H^\epsilon$ (i.e., $a_i=\epsilon-1$). Then  $F_{I,\beta}^\epsilon$ is not empty and included in no proper face of $F_I$ if and only $\min_{i\not\in I}(1+a_i)< \epsilon<\max_{i\not\in I}(1+a_i)$ or $\epsilon=\min_{i\not\in I}(1+a_i)=\max_{i\not\in I}(1+a_i)$. Note also that the non-empty $F_{I,\beta}^\epsilon$ that are not included in a proper face of $F_I$ are all distinct and yield all non-empty faces of $\tilde{Q}^\epsilon$ included in $H^\epsilon$. This finishes the proof of the second statement of the proposition. 

To describe the facets, it is sufficient to find the  $F_i^\epsilon$ with $\epsilon< \max_{j\neq i}(1+a_j)$, the $F_{i,\beta}^\epsilon$ with $\epsilon$ equal to both $\min_{j\neq i}(1+a_j)$ and $\max_{j\neq i}(1+a_j)$, and
 $F_{\emptyset,\beta}^\epsilon$ with $1=\min_{i=0}^n(1+a_i)< \epsilon<\max_{i=0}^n(1+a_i)=1+a_n$. We easily find the $F_i^\epsilon$ with $i\in\{0,\dots,n-1\}$ for any $\epsilon<1+a_n$, and $F_n ^\epsilon$ for any $\epsilon<1+a_{n-1}$. We conclude by noticing that, for any $i\in\{0,\dots,n\}$, we have $\epsilon=\min_{j\neq i}(1+a_j)=\max_{j\neq i}(1+a_j)<1+a_n$  if and only if $i=n$ and $0=a_0=\cdots=a_{n-1}$ (and in particular, $\epsilon=1$).
 
 To get the last statement, apply the fact that any face of a polytope is the intersection of the facets containing it.
\end{proof}

From Proposition~\ref{prop:polyfaces1}, we deduce the following result with the following notation.
Let us denote by $i_0:=0,i_1,\dots,i_k,i_{k+1}:=n+1$ some increasing positive integers so that
$$0=a_{i_0}=\cdots=a_{i_1-1}<a_{i_1}=\cdots=a_{i_1-1}<\cdots<a_{i_l}=\cdots=a_{i_l-1}<\cdots <a_{i_k}=\cdots=a_n.$$

\begin{cor}
The isomorphism classes of the horospherical varieties $X^\epsilon$ associated to the polytopes in the family $(Q^\epsilon)_{\epsilon\in\Qbb_{\geq 0}}$ are given by the following subsets of $\Qbb_{\geq 0}$:
\begin{itemize}[label=$\ast$]
\item $[0,1[$;
\item $]1+a_{i_l},1+a_{i_{l+1}}[$ for any $l\in\{0,\dots,k-2\}$;
\item $\{1+a_{i_l}\}$ for any $l\in\{0,\dots,k-2\}$;
\item $]1+a_{i_{k-1}},1+a_{i_k}[$ 
and $\{1+a_{i_{k-1}}\}$ if $i_k \neq n$ (i.e., if $a_{n-1}=a_n$) or the simple root $\alpha_n$ is not trivial (i.e., when $X$ is as in Case (1b) of Theorem~\ref{th:main});
\item $[1+a_{i_{k-1}},1+a_{i_k}[$ if $i_k=n$ (i.e., if $a_{n-1}<a_n$) and the simple root $\alpha_n$ is trivial (i.e., when $X$ is as in Case (1c) of Theorem~\ref{th:main}). 
\end{itemize}
\end{cor}

\begin{proof}
We apply the theory described in Section~\ref{sec:recallMMP}, in particular the fact that the isomorphism classes of the varieties $X^\epsilon$ are obtained by looking at the $\epsilon$'s for which ``the faces of $Q^\epsilon$ change''.

Note first that, by Proposition~\ref{prop:polyfaces1}, $(P,M,Q^\epsilon,\tilde{Q}^\epsilon)$ is an admissible quadruple if and only if $\epsilon<1+a_n$.

Also, the facets of $\tilde{Q}^\epsilon$ are: $F_i^\epsilon$ 
with $i\in\{0,\dots,n-1\}$, $F_n^\epsilon$ 
 if $\epsilon<1+a_{n-1}$,
$F_{\emptyset,\beta}^\epsilon$ 
 if $\epsilon>1$, and $F_{n,\beta}^\epsilon$ (orthogonal to $\alpha_{n,M}^\vee$) if $\epsilon=1$ and $a_{n-1}=0$. In particular,  for any $\epsilon,\,\eta\in  [0,1+a_n[$, if $a_{n-1}\neq 0$, the facets of $Q^\epsilon$ and $Q^\eta$ are ``the same'' % naturally in bijection (preserving directions)
  if and only if $\epsilon$ and $\eta$ are both in $[0,1]$ or $]1,1+a_{n-1}[$ or $[1+a_{n-1},1+a_n[$ (which may be empty). And if $a_{n-1}=0$, the facets of $Q^\epsilon$ and $Q^\eta$ are ``the same'' %naturally in bijection (preserving directions) 
  for any $\epsilon,\,\eta\in  [0,1+a_n[$ (indeed, in that case, the facets $F_n^\epsilon$ if $\epsilon<1$,
%if there exists $j\neq i$ such that $\epsilon<1+a_j$,  
$F_{\emptyset,\beta}^\epsilon$ if $\epsilon>1$, and $F_{n,\beta}^\epsilon$ if $\epsilon=1$ are ``the same'', in particular all orthogonal to $\beta_{M}^\vee=a_n\alpha_{n,M}^\vee$).

We now use a consequence of the proof of Proposition~\ref{prop:polyfaces1}: for any $I\subsetneq \{0,\dots,n\}$, $\bigcap_{i\in I}F_i^\epsilon$ is not empty if and only if $\epsilon\leq\max_{i\not\in I}(1+a_i)$, $F_{\emptyset,\beta}^\epsilon\cap\bigcap_{i\in I}F_i^\epsilon$ is not empty if and only if $\min_{i\not\in I}(1+a_i)\leq\epsilon\leq\max_{i\not\in I}(1+a_i)$ and $F_{n,\beta}^\epsilon\cap\bigcap_{i\in I}F_i^\epsilon$ is not empty if and only if $\min_{i\not\in I}(1+a_i)=\epsilon=\max_{i\not\in I}(1+a_i)$. In particular for any $l\in\{0,\dots,k-2\}$, suppose that for $I=\{i_{l+1},\dots,n\}$ and that $\bigcap_{i\in I}F_i^\epsilon$ is not empty; suppose also that for $I=\{0,\dots,i_l-1\}$ and that $F_{\emptyset,\beta}^\epsilon\cap\bigcap_{i\in I}F_i^\epsilon$ is not empty; then $\epsilon=1+a_{i_l}$. Similarly for any $l\in\{0,\dots,k-2\}$, suppose that for $I=\{i_{l+1}-1,\dots,n\}$ and $\bigcap_{i\in I}F_i^\epsilon$ is not empty; suppose also that for $I=\{0,\dots,i_l-1\}$ and that $F_{\emptyset,\beta}^\epsilon\cap\bigcap_{i\in I}F_i^\epsilon$ is not empty; then $\epsilon\in [1+a_{i_l},1+a_{i_{l+1}}]$. 
If $i_k\neq n$, $F_n^\epsilon$ is still a facet of $Q^\epsilon$ and what we did above with $l\in\{0,\dots,k-2\}$ can be done as well with $l=k-1$.

Hence, this proves that if the two varieties $X^\epsilon$ and $X^\eta$ are isomorphic then $\epsilon$ and $\eta$ are in one of the subsets described in the corollary.

To conclude, we have to prove that the two varieties $X^\epsilon$ and $X^\eta$ are isomorphic when $\epsilon$ and $\eta$ are in one of these subsets. It is obvious from Proposition~\ref{prop:polyfaces1} except in the case where $i_k=n$ and the simple root $\alpha_n$ is trivial. But in that case, all polytopes $Q^\epsilon$ with $\epsilon\in[1+a_{n-1},1+a_n[=[1+a_{i_{k-1}},1+a_{i_k}[$ %intersect the same walls (because $W_{\alpha_n}$ is not a wall)
are simplexes  with facets $F_i^\epsilon$ for $i\in\{0,\dots,n-1\}$ and $F_{\emptyset,\beta}^\epsilon$ or $F_{n,\beta}^\epsilon$ if $\epsilon= 1+a_{n-1}=1$, i.e., they could be defined even deleting the row corresponding to the simple root $\alpha_n$ that is trivial, so that their faces are ``the same''.
\end{proof}

We can reformulate this corollary as follows, and get the first statement of Theorem~\ref{th:main2} in Case (1).
We denote $X^0=X$ and for any $l\in\{1,\cdots, k\}$, $X^l:=X^\epsilon$ with $\epsilon\in ]1+a_{i_{l-1}},1+a_{i_{l}}[$, and for any $l\in\{0,\cdots, k\}$, $Y^l:=X^{1+a_{i_l}}$.

\begin{cor}\label{cor:LogMMP1}
The family $(Q^\epsilon)_{\epsilon\in\Qbb_{\geq 0}}$ describes a Log MMP from $X$ as follows: 
\begin{itemize}[label=$\ast$]
\item $k$ flips $\phi_l:\,X^l\longrightarrow Y^l\longleftarrow X^{l+1}\,:\phi_l^+$ for any $l\in\{0,\cdots, k-1\}$ and a fibration $\phi_k:\,X^k\longrightarrow Y^k$, if $i_k \neq n$ or the simple root $\alpha_n$ is not trivial;
\item $k-1$ flips $\phi_l:\,X^l\longrightarrow Y^l\longleftarrow X^{l+1}\,:\phi_l^+$ for any $l\in\{0,\cdots, k-2\}$, followed by a divisorial contraction $\phi_{k-1}:\,X^{k-1}\longrightarrow Y^{k-1}\simeq X^k$ and a fibration $X^k\longrightarrow Y^k\simeq\operatorname{pt}$, if $i_k = n$ and the simple root $\alpha_n$ is trivial.
\end{itemize}
\end{cor}

\begin{ex}
In the five different cases with $n=2$ and $a_2\neq 0$, we illustrate this corollary in terms of polytopes in Figures~\ref{figure2},~\ref{figure3},~\ref{figure4},~\ref{figure5} and \ref{figure6}.

\begin{figure}
\begin{center}

\begin{tikzpicture}\node [rectangle] (a) at (0,0) {
    \begin{tikzpicture}[scale=0.5]
     \chadom
\draw[very thick] (2,2) -- (2,3) -- (3,2) -- (2,2);
%\node at (3.5,3.5) {$\tilde{Q}^0$};
\draw[densely dotted] (0,2.5) -- (4,0.5) ;
\node at (4.7,1) {$H^0$};
    \end{tikzpicture}
};

\node [rectangle] (b) at (2,-3.25) {
    \begin{tikzpicture}[scale=0.5]
     \chadom
\draw[very thick] (2,2) -- (2,3) -- (3,2) -- (2,2);
%\node at (3.5,3.5) {$\tilde{Q}^0$};
\draw[densely dotted] (0,3) -- (4,1) ;
\node at (4.7,1.5) {$H^1$};
    \end{tikzpicture}
};

\node [rectangle] (c) at (4,0) {
    \begin{tikzpicture}[scale=0.5]
     \chadom
\draw[very thick] (2,2.25) -- (2,3) -- (3,2) -- (2.5,2) -- (2,2.25);
%\node at (3.5,3.5) {$\tilde{Q}^0$};
\draw[densely dotted] (0,3.25) -- (4,1.25) ;
\node at (4.7,1.5) {$H^\frac{3}{2}$};
    \end{tikzpicture}
};
\node [rectangle] (d) at (6,-3.25) {
    \begin{tikzpicture}[scale=0.5]
     \chadom
\draw[very thick] (2,2.5) -- (2,3) -- (3,2) -- (2,2.5);
%\node at (3.5,3.5) {$\tilde{Q}^0$};
\draw[densely dotted] (0,3.5) -- (4,1.5) ;
\node at (4.7,1.5) {$H^2$};
    \end{tikzpicture}
};
\node [rectangle] (e) at (8,0) {
    \begin{tikzpicture}[scale=0.5]
     \chadom
\draw[very thick] (2,2.75) -- (2.5,2.5) -- (2,3) -- (2,2.75);
%\node at (3.5,3.5) {$\tilde{Q}^0$};
\draw[densely dotted] (0,3.75) -- (4,1.75) ;
\node at (4.7,2) {$H^\frac{5}{2}$};
    \end{tikzpicture}
};
\node [rectangle] (f) at (10,-3.25) {
    \begin{tikzpicture}[scale=0.5]
     \chadom
\node at (2,3) {$\bullet$};
%\node at (3.5,3.5) {$\tilde{Q}^0$};
\draw[densely dotted] (0,4) -- (4,2) ;
\node at (4.7,2) {$H^3$};
    \end{tikzpicture}
};
\draw (a) [->]  to node[above right] {$\phi_0$} (b);
\draw (c) [->]  to node[above left] {$\phi_0^+$} (b);
\draw (c) [->]  to node[above right] {$\phi_1$} (d);
\draw (e) [->]  to node[above left] {$\phi_1^+$} (d);
\draw (e) [->]  to node[above right] {$\phi_2$} (f);
\end{tikzpicture}

\caption{The Log MMP described by the polytopes $\tilde{Q}^\epsilon$ in the case where $n=2$, $a_1=1$, $a_2=2$ and $\alpha_2$ is not trivial.}
\label{figure2}
\end{center}
\end{figure}
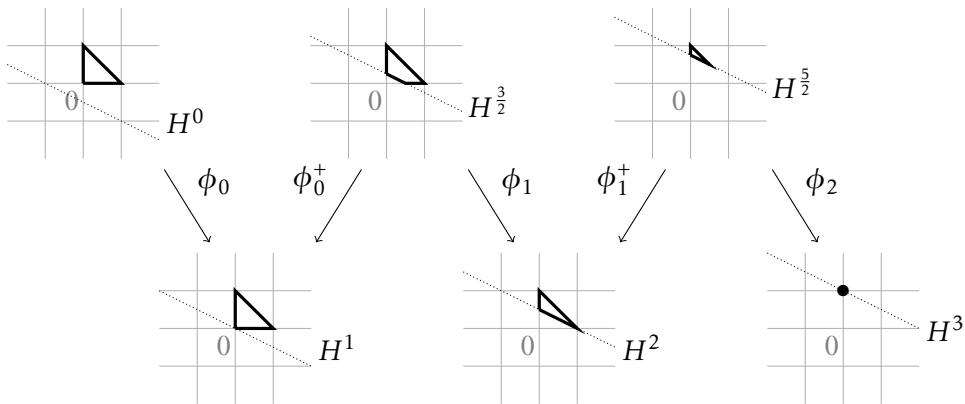
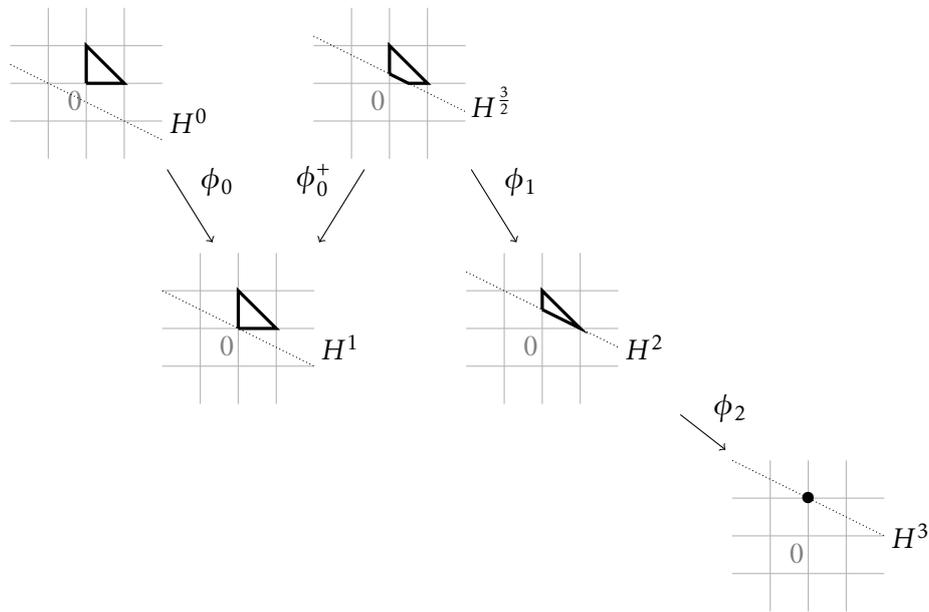
\begin{figure}
\begin{center}
\begin{tikzpicture}
\node [rectangle] (a) at (0,0) {
    \begin{tikzpicture}[scale=0.5]
     \chadom
\draw[very thick] (2,2) -- (2,3) -- (3,2) -- (2,2);
%\node at (3.5,3.5) {$\tilde{Q}^0$};
\draw[densely dotted] (0,2.5) -- (4,0.5) ;
\node at (4.7,1) {$H^0$};
    \end{tikzpicture}
};

\node [rectangle] (b) at (2,-3.25) {
    \begin{tikzpicture}[scale=0.5]
     \chadom
\draw[very thick] (2,2) -- (2,3) -- (3,2) -- (2,2);
%\node at (3.5,3.5) {$\tilde{Q}^0$};
\draw[densely dotted] (0,3) -- (4,1) ;
\node at (4.7,1.5) {$H^1$};
    \end{tikzpicture}
};

\node [rectangle] (c) at (4,0) {
    \begin{tikzpicture}[scale=0.5]
     \chadom
\draw[very thick] (2,2.25) -- (2,3) -- (3,2) -- (2.5,2) -- (2,2.25);
%\node at (3.5,3.5) {$\tilde{Q}^0$};
\draw[densely dotted] (0,3.25) -- (4,1.25) ;
\node at (4.7,1.5) {$H^\frac{3}{2}$};
    \end{tikzpicture}
};
\node [rectangle] (d) at (6,-3.25) {
    \begin{tikzpicture}[scale=0.5]
     \chadom
\draw[very thick] (2,2.5) -- (2,3) -- (3,2) -- (2,2.5);
%\node at (3.5,3.5) {$\tilde{Q}^0$};
\draw[densely dotted] (0,3.5) -- (4,1.5) ;
\node at (4.7,1.5) {$H^2$};
    \end{tikzpicture}
};

6

\node [rectangle] (f) at (9.5,-6) {
    \begin{tikzpicture}[scale=0.5]
     \chadom
\node at (2,3) {$\bullet$};
%\node at (3.5,3.5) {$\tilde{Q}^0$};
\draw[densely dotted] (0,4) -- (4,2) ;
\node at (4.7,2) {$H^3$};
    \end{tikzpicture}
};
\draw (a) [->]  to node[above right] {$\phi_0$} (b);
\draw (c) [->]  to node[above left] {$\phi_0^+$} (b);
\draw (c) [->]  to node[above right] {$\phi_1$} (d);

\draw (d) [->]  to node[above right] {$\phi_2$} (f);
\end{tikzpicture}

\caption{The Log MMP described by the polytopes $\tilde{Q}^\epsilon$ in the case where $n=2$, $a_1=1$, $a_2=2$ and $\alpha_2$ is trivial.}
\label{figure3}
\end{center}
\end{figure}

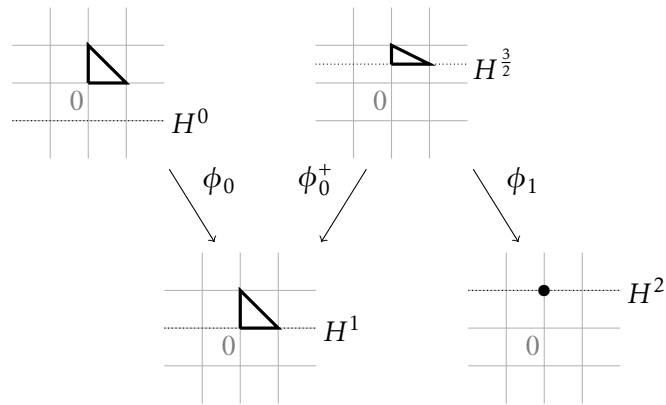
\begin{figure}
\begin{center}

\begin{tikzpicture}\node [rectangle] (a) at (0,0) {
    \begin{tikzpicture}[scale=0.5]
     \chadom
\draw[very thick] (2,2) -- (2,3) -- (3,2) -- (2,2);
%\node at (3.5,3.5) {$\tilde{Q}^0$};
\draw[densely dotted] (0,1) -- (4,1) ;
\node at (4.7,1) {$H^0$};
    \end{tikzpicture}
};

\node [rectangle] (b) at (2,-3.25) {
    \begin{tikzpicture}[scale=0.5]
     \chadom
\draw[very thick] (2,2) -- (2,3) -- (3,2) -- (2,2);
%\node at (3.5,3.5) {$\tilde{Q}^0$};
\draw[densely dotted] (0,2) -- (4,2) ;
\node at (4.7,2) {$H^1$};
    \end{tikzpicture}
};

\node [rectangle] (c) at (4,0) {
    \begin{tikzpicture}[scale=0.5]
     \chadom
\draw[very thick] (2,2.5) -- (2,3) -- (3,2.5) -- (2,2.5);
%\node at (3.5,3.5) {$\tilde{Q}^0$};
\draw[densely dotted] (0,2.5) -- (4,2.5) ;
\node at (4.7,2.5) {$H^\frac{3}{2}$};
    \end{tikzpicture}
};
\node [rectangle] (d) at (6,-3.25) {
    \begin{tikzpicture}[scale=0.5]
     \chadom
\node at (2,3) {$\bullet$};
\draw[densely dotted] (0,3) -- (4,3) ;
\node at (4.7,3) {$H^2$};
    \end{tikzpicture}
};

\draw (a) [->]  to node[above right] {$\phi_0$} (b);
\draw (c) [->]  to node[above left] {$\phi_0^+$} (b);
\draw (c) [->]  to node[above right] {$\phi_1$} (d);
\end{tikzpicture}

\caption{The Log MMP described by the polytopes $\tilde{Q}^\epsilon$ in the case where $n=2$, $a_1=0$, $a_2=1$ and $\alpha_2$ is not trivial.}
\label{figure4}
\end{center}
\end{figure}

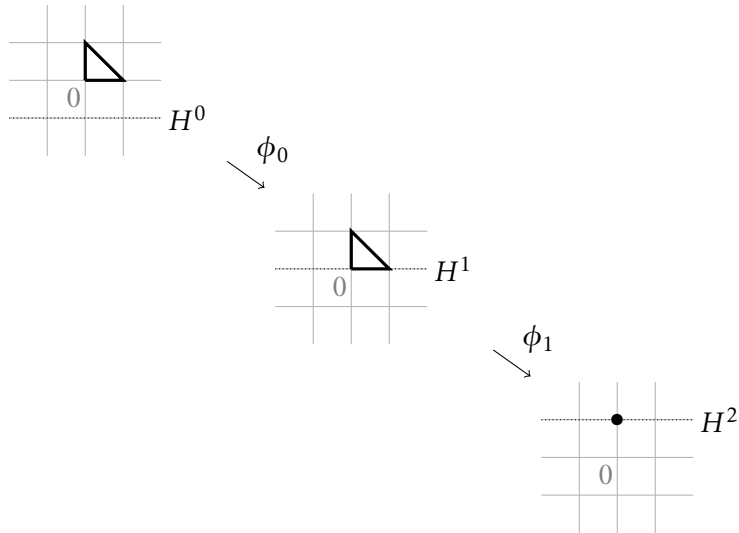
\begin{figure}
\begin{center}

\begin{tikzpicture}\node [rectangle] (a) at (0,0) {
    \begin{tikzpicture}[scale=0.5]
     \chadom
\draw[very thick] (2,2) -- (2,3) -- (3,2) -- (2,2);
%\node at (3.5,3.5) {$\tilde{Q}^0$};
\draw[densely dotted] (0,1) -- (4,1) ;
\node at (4.7,1) {$H^0$};
    \end{tikzpicture}
};

\node [rectangle] (b) at (3.5,-2.5) {
    \begin{tikzpicture}[scale=0.5]
     \chadom
\draw[very thick] (2,2) -- (2,3) -- (3,2) -- (2,2);
%\node at (3.5,3.5) {$\tilde{Q}^0$};
\draw[densely dotted] (0,2) -- (4,2) ;
\node at (4.7,2) {$H^1$};
    \end{tikzpicture}
};

\node [rectangle] (d) at (7,-5) {
    \begin{tikzpicture}[scale=0.5]
     \chadom
\node at (2,3) {$\bullet$};
\draw[densely dotted] (0,3) -- (4,3) ;
\node at (4.7,3) {$H^2$};
    \end{tikzpicture}
};

\draw (a) [->]  to node[above right] {$\phi_0$} (b);
\draw (b) [->]  to node[above right] {$\phi_1$} (d);
\end{tikzpicture}

\caption{The Log MMP described by the polytopes $\tilde{Q}^\epsilon$ in the case where $n=2$, $a_1=0$, $a_2=1$ and $\alpha_2$ is trivial.}
\label{figure5}
\end{center}
\end{figure}

\begin{figure}
\begin{center}

\begin{tikzpicture}

\node [rectangle] (a) at (0,0) {
    \begin{tikzpicture}[scale=0.5]
     \chadom
\draw[very thick] (2,2) -- (2,3) -- (3,2) -- (2,2);
%\node at (3.5,3.5) {$\tilde{Q}^0$};
\draw[densely dotted] (0,3) -- (3,0) ;
\node at (3.5,0.5) {$H^0$};
    \end{tikzpicture}
};

\node [rectangle] (b) at (2,-3.25) {
    \begin{tikzpicture}[scale=0.5]
     \chadom
\draw[very thick] (2,2) -- (2,3) -- (3,2) -- (2,2);
%\node at (3.5,3.5) {$\tilde{Q}^0$};
\draw[densely dotted] (0,4) -- (4,0) ;
\node at (4.7,0.5) {$H^1$};
    \end{tikzpicture}
};

\node [rectangle] (c) at (4,0) {
    \begin{tikzpicture}[scale=0.5]
     \chadom
\draw[very thick] (2,2.5) -- (2,3) -- (3,2) -- (2.5,2) -- (2,2.5);
%\node at (3.5,3.5) {$\tilde{Q}^0$};
\draw[densely dotted] (0.5,4) -- (4,0.5) ;
\node at (4.7,1) {$H^\frac{3}{2}$};
    \end{tikzpicture}
};
\node [rectangle] (d) at (6,-3.25) {
    \begin{tikzpicture}[scale=0.5]
     \chadom
\draw[very thick]  (2,3) -- (3,2) ;
%\node at (3.5,3.5) {$\tilde{Q}^0$};
\draw[densely dotted] (1,4) -- (4,1) ;
\node at (4.7,1) {$H^2$};
    \end{tikzpicture}
};

\draw (a) [->]  to node[above right] {$\phi_0$} (b);
\draw (c) [->]  to node[above left] {$\phi_0^+$} (b);
\draw (c) [->]  to node[above right] {$\phi_1$} (d);
\end{tikzpicture}

\caption{The Log MMP described by the polytopes $\tilde{Q}^\epsilon$ in the case where $n=2$ and $a_1=a_2=1$.}

%\caption{The Log MMP in the fives cases with $n=2$ and $a_2\neq 0$ described by the polytopes $\tilde{Q}^\epsilon$}
\label{figure6}
\end{center}
\end{figure}
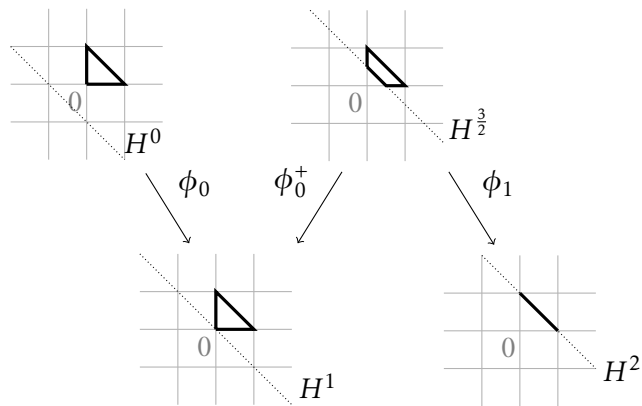

\end{ex}

\subsection{Proof of the last statement of Theorem~\ref{th:main2} in Case (1)}

The previous section proves that $a_{i_1},\dots,a_{i_k}$ are invariants of $X$.
To finish the proof of Theorem~\ref{th:main2} in Case (1), we have to prove that $G_0,\dots,G_t$, $\alpha_0,\dots,\alpha_n,\beta$ and $i_1,\dots,i_k$ are also invariants of $X$. For this, we have to describe some exceptional loci and some fibers of the different morphisms of the Log MMP.

\vskip\baselineskip
We first distinguish two cases by the following result. 

\begin{prop} 
Define the simple subgroups of $P(\varpi_\beta)$ as in Definition~\ref{def:smoothtriplebis}.
\begin{itemize}[label=$\ast$]
\item Suppose that $n=1$ and that $\alpha_0$ and $\alpha_1$ are two simple roots of the same simple subgroup of $P(\varpi_\beta)$. Then, the fiber of $\psi:\,X\longrightarrow G/P(\varpi_\beta)$ is either a homogeneous variety different from a projective space (a quadric $Q^{2m}$ with $m\geq 2$, a Grassmannian $\operatorname{Gr}(i,m)$ with $p\geq 5$ and $2\leq i\leq m-2$, or a spinor variety $\operatorname{Spin}(2m+1)/P(\varpi_m)$ with $m\geq 4$), or a two-orbit variety as in \cite{2orbits}.

\item Suppose that $n>1$ or that $\alpha_0$ and $\alpha_1$ are not two simple roots of the same simple subgroup of $P(\varpi_\beta)$. Then, the fiber of $\psi:\,X\longrightarrow G/P(\varpi_\beta)$ is a projective space.
\end{itemize}
\end{prop}

\begin{proof}
The fiber of $\psi:\,X\longrightarrow G/P(\varpi_\beta)$ is the smooth projective $P(\varpi_\beta)$-variety of Picard group $\Zbb$ isomorphic to the closure of the $P(\varpi_\beta)$-orbit of a sum of highest weight vectors in $\mathbb{P}:=\mathbb{P}(V(\varpi_{\alpha_0})\oplus\cdots\oplus V(\varpi_{\alpha_n}))$. Hence, the proposition is a consequence of \cite[Section~1]{2orbits}.
\end{proof}

$\bullet$ In the case where $n=1$ and that $\alpha_0$ and $\alpha_1$ are two simple roots of the same simple subgroup of $P(\varpi_\beta)$, $G=G_0$, the Log MMP described by Corollary~\ref{cor:LogMMP1} consists of a fibration if $a_1=0$, or a flip and a fibration if $a_1>0$.

\vskip\baselineskip
-- Suppose first that $a_1=0$. There are two cases to deal with.

\vskip\baselineskip
If $\alpha_1$ is between $\alpha_0$ and $\beta$ in the Dynkin diagram of $G_0$ (and similarly, up to exchanging $\alpha_0$ and $\alpha_1$, $\alpha_0$ is between $\alpha_1$ and $\beta$), since $X\subset\Pbb(V(\varpi_{\alpha_0}+\varpi_\beta)\oplus V(\varpi_{\alpha_1}+\varpi_\beta))$ and $Y^0\subset\Pbb(V(\varpi_{\alpha_0})\oplus V(\varpi_{\alpha_1}))$, we easily compute that 
the fibration $\phi_0:\,X\longrightarrow Y^0$ has two different types of fibers:
one isomorphic to $P(\varpi_{\alpha_0})/(P(\varpi_{\alpha_0})\cap P(\varpi_\beta))$ over a $G$-orbit isomorphic to $G/P(\varpi_{\alpha_0})$ and another one of smaller dimension isomorphic to $P(\varpi_{\alpha_1})/(P(\varpi_{\alpha_1})\cap P(\varpi_\beta))$.

In particular, the pair $(G/P(\varpi_{\alpha_0}),G/P(\varpi_{\beta}))$ is an invariant of $X$. Then if $G_0$ is not the universal cover of the automorphism group of $G/P(\varpi_{\beta})$  it must be the universal cover of the automorphism group of $G/P(\varpi_{\alpha_0})$, so that $G_0$ is an invariant of $X$. Also, $\phi_0^{-1}(G/P(\varpi_{\alpha_0}))=G/(P(\varpi_{\alpha_0})\cap P(\varpi_\beta))$, then the pair $(\alpha_0,\beta)$ is an invariant of $X$ up to symmetries of the Dynkin diagram of $G_0$. 

Moreover, if $\beta$ is fixed, the possible symmetries are the ones (which fixed $\beta$) in type $A_m$ with $m\geq 5$ odd, $\varpi_\beta=\varpi_{\frac{m+1}{2}}$ and any $\alpha_0$, type $E_6$ with $\varpi_\beta=\varpi_4$ and $\varpi_{\alpha_0}=\varpi_1,\,\varpi_3,\,\varpi_5$ or $\varpi_6$ , and type $D_m$ with $m\geq 4$, $\varpi_\beta=\varpi_i$ for any $i\in\{1,\dots,m-2\}$ and $\varpi_{\alpha_0}=\varpi_{m-1}$ or $\varpi_{m}$.

The description of the fiber of $\psi:\,X\longrightarrow G/P(\varpi_\beta)$, together with Remark~\ref{rem:2orbits}, implies that $\alpha_0$ and $\alpha_1$ are also invariants of $X$ up to symmetries of the Dynkin diagram of $G_0$. 

\vskip\baselineskip
Otherwise (this occurs only in types $D$ and $E$), $G_0$ is the universal cover of the automorphism group of $G/P(\varpi_{\beta})$, and then $G_0$ and $\beta$ are invariants of $X$ up to symmetries of the Dynkin diagram of $G_0$. 

We also easily compute that the fibration $\phi_0:\,X\longrightarrow Y^0$ has at least two different types of fibers:
one smaller isomorphic to $(P(\varpi_{\alpha_0})\cap P(\varpi_{\alpha_1}))/(P(\varpi_{\alpha_0})\cap P(\varpi_{\alpha_1})\cap P(\varpi_\beta))$ over the open $G$-orbit of $Y^0$, and two  others fibers isomorphic to $P(\varpi_{\alpha_0})/(P(\varpi_{\alpha_0})\cap P(\varpi_\beta))$ over $G/P(\varpi_{\alpha_0})$, respectively isomorphic to $P(\varpi_{\alpha_1})/(P(\varpi_{\alpha_1})\cap P(\varpi_\beta))$ over $G/P(\varpi_{\alpha_1})$ (note that the latter are possibly isomorphic).

In particular, the pair $(G/P(\varpi_{\alpha_0}),G/P(\varpi_{\alpha_1}))$ is an invariant of $X$ and then the pair $(\alpha_0,\alpha_1)$ is also an invariant of $X$ up to symmetries of the Dynkin diagram of $G_0$.

\vskip\baselineskip
-- Suppose now that $a_1>0$. We have then the following inclusions
\begin{align*}
&X\subset\Pbb(V(\varpi_{\alpha_0}+\varpi_\beta)\oplus V(\varpi_{\alpha_1}+(1+a_1)\varpi_\beta)),\, Y^0\subset\Pbb(V(\varpi_{\alpha_0})\oplus V(\varpi_{\alpha_1}+a_1\varpi_\beta)),\\
&X^1\subset\Pbb(V(\varpi_{\alpha_0}+\varpi_{\alpha_1})\oplus V(2\varpi_{\alpha_1}+a_1\varpi_\beta))\,\,\mathrm{and}\,\, Y^1\simeq G/P(\varpi_{\alpha_1})\subset\Pbb( V(\varpi_{\alpha_1})).
\end{align*}
In particular $X$, $Y^0$ and $X^1$ have two closed $G$-orbits and one open $G$-orbit so that we easily compute exceptional locus and fibers as follows. For example, the exceptional locus of $\phi_0:\,X\longrightarrow Y^0$ is the $G$-orbit of $X$ isomorphic to $G/(P(\varpi_{\alpha_0})\cap P(\varpi_\beta))$. Then the universal cover of its automorphism group $G_0$ is an invariant of $X$. And then $\beta$ is also an invariant of $X$ up to symmetries of the Dynkin diagram of $G_0$. 

Note now that the exceptional locus of $\phi_0$ is sent to the $G$-orbit of $Y^0$ isomorphic to $G/P(\varpi_{\alpha_0})$ so that the triple $(G/P(\varpi_{\alpha_0}),G/P(\varpi_{\alpha_1}),G/P(\varpi_\beta))$ is an invariant of $G$. Also the (same) description of the fiber of $\psi:\,X\longrightarrow G/P(\varpi_\beta)$ implies that the subgroup or $P(\varpi_\beta)$ and the pair $(\alpha_0,\alpha_1)$ are invariants of $X$ (up to symmetries in type $A$, $D$ and $E$ as in the case where $a_1=0$). Hence, the triple $(\beta,\alpha_0,\alpha_1)$ is an invariant of $X$ up to symmetries of the Dynkin diagram of $G_0$. 

\vskip\baselineskip
$\bullet$ Now we suppose that $n>1$ or that $\alpha_0$ and $\alpha_1$ are not two simple roots of the same simple subgroup of $P(\varpi_\beta)$.

We define various exceptional loci in $X$ as follows. Let $l\in\{0,\dots,k-1\}$, define $E_l$ to be the closure in $X$ of the set of points $x\in X$ such that $x$ is in the open isomorphism set of the first $l$ contractions and $x$ is in the exceptional locus of $\phi_l$.

\begin{prop}\label{prop:exceptloci}
For any $l\in\{0,\dots,k\}$ the exceptional locus $E_l$ is the closure in $X$ of the $G$-orbit associated to the non-empty face $F_{I_l}$ of $Q$ with $I_l:=\{i_{l+1},\dots,n\}$. In particular $E_l$ is isomorphic to the closure of the $G$-orbit of a sum of  highest weight vectors in $$\mathbb{P}:=\mathbb{P}(\bigoplus_{i=0}^{i_{l+1}-1}V(\varpi_{\alpha_i}+(1+a_i)\varpi_\beta)),$$ and $E_l$ is a smooth projective horospherical of Picard group $\Zbb^2$ as in Case (1), unless $l=0$, $i_1=1$ so that  $E_l$ is homogeneous (projective of Picard group $\Zbb$ or $\Zbb^2$).
\end{prop}

Note that for $l=k$, $I_k=\emptyset$ and $E_k=X$.

\begin{proof}
Let $l\in\{0,\dots,k\}$ and $\epsilon_l\in\Qbb_{\geq 0}$ such that $X^l=X^{\epsilon_l}$.

We denote by $\Omega_I^l$ and $\Omega_{I,\beta}^l$ the $G$-orbits of $X^l$ associated to the non-empty faces $F_I^{\epsilon_l}$ and $F_{I,\beta}^{\epsilon_l}$ of the polytope $\tilde{Q}^{\epsilon_l}$. We denote by $\omega_I^l$ and $\omega_{I,\beta}^l$ the $G$-orbits of $Y^l=X^{1+a_{i_l}}$ associated to the non-empty faces $F_I^{1+a_{i_l}}$ and $F_{I,\beta}^{1+a_{i_l}}$ of the polytope $\tilde{Q}^{1+a_{i_l}}$.  Recall that, for any $\epsilon\in\Qbb_{\geq 0}$,  we have an order on the $G$-orbits of $X^\epsilon$ compatible with the order on the non-empty faces of $\tilde{Q}^\epsilon$: in particular $\Omega_I^l\subset \overline{\Omega_{I'}^l}$ and $\Omega_{I,\beta}^l\subset \overline{\Omega_{I',\beta}^l}$ respectively if and only if $I'\subset I$, and $\Omega_{I,\beta}^l\subset \overline{\Omega_I^l}$ (as soon as these orbits are defined, i.e., as soon as the corresponding faces are non-empty).

 For any $I\subsetneq\{0,\dots,n\}$ such that there exists $i\geq i_l$ not in $I$ (i.e., such that $\Omega_I^l$ is defined), $\phi_l(\Omega_I^l)=\omega_I^l$ if there exists $i\geq i_{l+1}$ not in $I$, and $\phi_l(\Omega_I^l)=\omega_{I\cup\{0,\dots i_l-1\},\beta}^l$ if for any $i\geq i_{l+1}$, $i\in I$. Indeed $I\cup\{0,\dots i_l-1\}$ is the minimal subset of $\{0,\dots,n\}$ containing $I$ such that $\omega_{I\cup\{0,\dots i_l-1\},\beta}^l$ is defined and there is no $I'$ containing $I$ such that $\omega_{I'}^l$ is defined. And for any $I\subsetneq\{0,\dots,n\}$ such that there exist $i\geq i_l$ and $i'<i_l$ not in $I$ (i.e., such that $\Omega_{I,\beta}^l$ is defined), $\phi_l(\Omega_{I,\beta}^l)=\omega_{I,\beta}^l$ if there exists $i\geq i_{l+1}$ not in $I$, and $\phi_l(\Omega_{I,\beta}^l)=\omega_{I\cup\{0,\dots i_l-1\},\beta}^l$ if for any $i\geq i_{l+1}$, $i\in I$. Indeed $I\cup\{0,\dots i_l-1\}$ is the minimal subset of $\{0,\dots,n\}$ containing $I$ such that $\omega_{I\cup\{0,\dots i_l-1\},\beta}^l$ is defined.

In particular, we have $\phi_l(\Omega_{I_l}^l)=\omega_{I_l\cup\{0,\dots i_l-1\},\beta}^l$ (which is also $\phi_l(\Omega_{I_l,\beta}^l)$ if $l\geq 1$). But $\Omega_{I_l}^l$ and $\omega_{I_l\cup\{0,\dots i_l-1\},\beta}^l$ are non isomorphic horospherical homogeneous spaces by Proposition~\ref{prop:classifpoly}, so that $\Omega_{I_l}^l$ is in the exceptional locus of $\phi_l$. Moreover, if $\Omega$ is a $G$-orbit of $X^{\epsilon_l}$ not contained in $\overline{\Omega_{I_l}^l}$, it is of the form $\Omega_I^l$ or $\Omega_{I,\beta}^l$ where $I_l\not\subset I$. Hence, in that case $\phi_l(\Omega)=\Omega$. And then the exceptional locus of $\phi_l$ is $\overline{\Omega_{I_l}^l}$. Note that $\Omega_{I_l}^0,\dots,\Omega_{I_l}^{l-1}$ are not in the exceptional locus of $\phi_0,\dots,\phi_{l-1}$ respectively, to conclude that $E_l=\overline{\Omega_{I_l}^0}$.

We use again Proposition~\ref{prop:classifpoly} to see that $E_l=\overline{\Omega_{I_l}^0}$ corresponds to the admissible quadruple $(P_F,M_F,F,\tilde{F})$ with $F=F_{I_l}^0$ (and with some ample divisor of $E_l$). Then we conclude by Corollaries~\ref{cor:smoothembedd}  and~\ref{cor:smooth}.
\end{proof}

The Log MMP now defines, by restriction, fibrations $\tilde{\phi_l}:E_l\backslash E_{l-1}\longrightarrow E'_l:=\overline{\omega_{I_l\cup\{0,\dots i_l-1\},\beta}^l}$, for any $l\in\{0,\dots,k\}$. 

\begin{defi}We say that the fibers of $\tilde{\phi_l}$ are locally maximal over $\omega\subset E'_l$ if the dimensions of the fibers of $\tilde{\phi_l}$ over any point of 
$\omega$ are the same and bigger than the dimension of the fibers of $\tilde{\phi_l}$ over any point of a neighborhood of $\omega$ that is not in $\omega$.

We say that the fibers of $\tilde{\phi_l}$ are locally almost maximal over $\omega\subset E'_l$ if there exists $\omega'\subsetneq\omega$ such that the fibers of $\tilde{\phi_l}$ are locally maximal over $\omega'$ and the fibers of $\tilde{\phi_l}_{|\tilde{\phi_l}^{-1}(E'_l\backslash\omega')}$ are locally maximal over $\omega\backslash\omega'\subset E'_l\backslash\omega'$.
\end{defi}

We now prove the following result, which implies in particular that $i_1,\dots,i_k$ are invariant of $X$.

\begin{prop}\label{prop:dimfibres1}
 Suppose that $n> 1$ or that $\alpha_0$ and $\alpha_1$ are not two simple roots of the same simple subgroup of $P(\varpi_\beta)$. Let $l\in\{0,\dots,k\}$.

The map $\tilde{\phi_l}$ is surjective and we distinguish four distinct cases.
\begin{enumerate}

\item we have $i_{l+1}-i_l=1$ and $\alpha_{i_l}$ is not a simple root of $G_0$. The fibers of $\tilde{\phi_l}$ are locally maximal over $E'_l$ and $\dim E_l -\dim E_{l-1}=1+\dim E'_l$ (here we set $\dim E_{-1}:=\dim G/P(\varpi_\beta)-1$ so that it still holds for $l=0$).  Moreover, $E'_l$ is homogeneous and isomorphic to $G/P(\varpi_{\alpha_{i_l}})$ (which is a point if $\alpha_{i_l}$ is trivial).

\item we have $i_{l+1}-i_l=1$ and $\alpha_{i_l}$ is a simple root of $G_0$. The fibers of $\tilde{\phi_l}$ are locally maximal over $E'_l$ and $\dim E_l -\dim E_{l-1}\neq 1+\dim E'_l$ (also here $\dim E_{-1}:=\dim G/P(\varpi_\beta)-1$ so that it still holds for $l=0$). Moreover, $E'_l$ is homogeneous and isomorphic to $G/P(\varpi_{\alpha_{i_l}})$.

\item we have $i_{l+1}-i_l>1$ and $\alpha_{i_l}$ is not a simple root of $G_0$. The fibers of $\tilde{\phi_l}$ are locally maximal over a unique proper subset of $E'_l$, which is a closed $G$-orbit $\omega'$ of $E'_l$ isomorphic to $G/P(\varpi_{\alpha_{i_l}})$. Also the fibers of $\tilde{\phi_l}$ are locally almost maximal over exactly $i_{l+1}-i_l-1(>0)$ subvarieties of $E'_l$ containing $\omega'$, respectively of dimensions $\dim G/P(\varpi_{\alpha_{i_l}})+\dim G/P(\varpi_{\alpha_j})+1$ with $j\in\{i_l+1,\dots,i_{l+1}-1\}$.

\item we have $i_{l+1}-i_l>1$ and $\alpha_{i_l}$ is a simple root of $G_0$. The fibers of $\tilde{\phi_l}$ are locally maximal over $i_{l+1}-i_l$ closed $G$-orbits, which are respectively isomorphic to $G/P(\varpi_{\alpha_{j}})$ with $j\in\{i_l,\dots,i_{l+1}-1\}$.
\end{enumerate}

Moreover, in the four cases, the dimension of the fibers over all pointed subsets of $E'_l$ are as follows.

\begin{enumerate}
\item The dimension of the fibers of $\tilde{\phi_l}$ is $1+\dim E_{l-1}$ (in particular $\dim G/P(\varpi_\beta)$ if $l=0$).

\item The dimension of the fibers of $\tilde{\phi_l}$ is $$d_{i_l}:=i_l+\dim \left(P(\varpi_{\alpha_{i_l}})/(P(\varpi_\beta)\cap\bigcap_{i=0}^{i_l}P(\varpi_{\alpha_i}))\right).$$

\item The dimension of the locally maximal fibers of $\tilde{\phi_l}$ is $1+\dim E_{l-1}$ (in particular $\dim G/P(\varpi_\beta)$ if $l=0$). And for any $j\in\{i_l+1,\dots,i_{l+1}-1\}$, the dimension of locally almost maximal fibers of $\tilde{\phi_l}$ over the subset of $E'_l$ of dimension $\dim G/P(\varpi_{\alpha_{i_l}})+\dim G/P(\varpi_{\alpha_j})+1$ is $$d_j:=i_l+\dim \left(P(\varpi_{\alpha_j})/(P(\varpi_\beta)\cap\bigcap_{i=0}^{i_l-1}P(\varpi_{\alpha_i})\cap P(\varpi_{\alpha_j}))\right).$$

\item For any $j\in\{i_l,\dots,i_{l+1}-1\}$, the dimension of locally maximal fibers of $\tilde{\phi_l}$ over the closed $G$-orbit isomorphic to $G/P(\varpi_{\alpha_{j}})$  is $$d_j:=i_l+\dim \left(P(\varpi_{\alpha_j})/(P(\varpi_\beta)\cap\bigcap_{i=0}^{i_l-1}P(\varpi_{\alpha_i})\cap P(\varpi_{\alpha_j}))\right).$$
\end{enumerate}

\end{prop}

\begin{proof}

We keep the notation of the proof of Proposition~\ref{prop:exceptloci}. And we use Corollary~\ref{cor:classifpoly} to compute the dimension of the fibers. Let $\omega$ be a $G$-orbit of $Y^l$ in $\overline{\omega_{I_l\cup\{0,\dots i_l-1\},\beta}^l}$. Then there exists $I\subsetneq\{0,\dots,n\}$ containing $I_l\cup\{0,\dots i_l-1\}$ such that $\omega=\omega_{I,\beta}^l$. Then $\tilde{\phi_l}^{-1}(\omega)=\bigsqcup_J\Omega_J^l$ where the union is taken over all $J$ such that $J\cap I_{l-1} = I\cap I_{l-1}$. In particular, $\tilde{\phi_l}$ is surjective and $\overline{\tilde{\phi_l}^{-1}(\omega)}=\overline{\Omega_{I\cap I_{l-1}}^l}$.  We then compute
\begin{align*}
\dim(\omega)&=\dim(F_{I,\beta}^l)+\dim (G/\bigcap_{i\not\in I}P(\varpi_{\alpha_i})),\\
\mathrm{and}\,\, \dim(\Omega_{I\cap I_{l-1}}^l)&=\dim(F_{I\cap I_{l-1}})+\dim (G/P(\varpi_\beta)\cap\bigcap_{i\not\in I\cap I_{l-1}}P(\varpi_{\alpha_i})),
\end{align*}
so that the dimension $\delta_{l,\omega}$ of a fiber of $\tilde{\phi_l}$ over $\omega$ is 
\begin{align*}
\delta_{l,\omega}&=\dim(F_{I\cap I_{l-1}})-\dim(F_{I,\beta}^l))+(\dim (G/(P(\varpi_\beta)\cap\bigcap_{i\not\in I\cap I_{l-1}}P(\varpi_{\alpha_i}))-\dim (G/\bigcap_{i\not\in I}P(\varpi_{\alpha_i}))) \\
&= i_l+\dim(\bigcap_{i\not\in I}P(\varpi_{\alpha_i})/(P(\varpi_\beta)\cap\bigcap_{i\not\in I\cap I_{l-1}}P(\varpi_{\alpha_i}))\\ 
&=  i_l+\dim(\bigcap_{i\not\in I}P(\varpi_{\alpha_i})/(\bigcap_{i\not\in I}P(\varpi_{\alpha_i})\cap\bigcap_{i=0}^{i_l-1}P(\varpi_{\alpha_i})\cap P(\varpi_\beta)).
\end{align*}

The dimension $\delta_{l,\omega}$ is the biggest when $I$ is as big as possible (it would be $I=\{0,\dots,n\}$ if it was allowed to define $\omega$). Moreover, if we remove from $I$ some $i$, the dimension changes if and only if $j$ is  such that $\alpha_i$ is in $G_0$ (i.e., $\alpha_i$ is not trivial and not the only simple root $\alpha_j$ in a simple group of $G$ different from $G_0$, by hypothesis). From this, we will deduce the different following cases.

 If $\alpha_{i_l}$ is not a simple root of $G_0$, then the locus in $\overline{\omega_{I_l\cup\{0,\dots i_l-1\},\beta}^l}$ where the fibers of $\tilde{\phi_l}$ are maximal is the unique closed $G$-orbit $ \omega':=\omega_{\{0,\dots n\}\backslash\{i_l\},\beta}^l$ isomorphic to $G/P(\varpi_{\alpha_{i_l}})$. This gives the first case of the proposition if $i_{l+1}-i_l=1$. 
 And if $i_{l+1}-i_l>1$ the locus in $\overline{\omega_{I_l\cup\{0,\dots i_l-1\},\beta}^l}$ where the fiber of $\tilde{\phi_l}$ is almost maximal is the 
 union of the subsets $ \omega_{\{0,\dots n\}\backslash\{i_l,j\},\beta}^l\cup\omega'$ with $j\in\{i_l+1,\dots,i_{l+1}-1\}$, which are affine cones over $G/P(\varpi_{\alpha_i})$. This gives the third case of the proposition.

Now, if $\alpha_{i_l}$ is a simple root of $G_0$ (i.e., for any $j\in\{i_l,\dots,i_{l+1}-1\}$, 
$\alpha_j$ is a simple root of $G_0$), then the locus in $\overline{\omega_{I_l\cup\{0,\dots i_l-1\},\beta}^l}$ where the fiber of $\tilde{\phi_l}$ is maximal
is the (disjoint) union of the $i_{l+1}-i_l$ closed $G$-orbits $\omega_{\{0,\dots n\}\backslash\{j\},\beta}^l$ of $\overline{\omega_{I_l\cup\{0,\dots i_l-1\},\beta}^l}$, which are respectively isomorphic to $G/P(\varpi_{\alpha_j})$ for any $j\in\{i_l,\dots,i_{l+1}-1\}$. This gives the second case of the proposition if $i_{l+1}-i_l=1$ and the fourth case if $i_{l+1}-i_l>1$.

\end{proof}
We easily deduce the following.
\begin{cor}
With the notation of Proposition~\ref{prop:dimfibres1}: for any $j\in\{0,\dots,n\}$, we have
\begin{align*}
\dim G/P(\varpi_\beta)+d_j-\dim E_{l-1}-1&=\dim P(\varpi_{\alpha_j})/(P(\varpi_\beta)\cap P(\varpi_{\alpha_j}))\\
\textrm{and}\quad\dim G/P(\varpi_{\alpha_j})+d_j-\dim E_{l-1}-1&=\dim P(\varpi_\beta)/(P(\varpi_\beta)\cap P(\varpi_{\alpha_j})).
\end{align*}

In particular, for any $l\in\{0,\dots,k\}$, the sets $$\{(\dim P(\varpi_{\alpha_j})/(P(\varpi_\beta)\cap P(\varpi_{\alpha_j})),\dim P(\varpi_\beta)/(P(\varpi_\beta)\cap P(\varpi_{\alpha_j})))\,\mid\,j\in\{i_l,\dots,i_{l+1}-1\}\}$$ are invariants of $X$.
\end{cor}

And then we conclude the proof of Case (1) of Theorem~\ref{th:main2} (i.e., that $G_0$, $\beta$, $\alpha_0,\dots,\alpha_n$ are invariants of $X$) by the following lemma (still in the case where $n> 1$ or that $\alpha_0$ and $\alpha_1$ are not two simple roots of the same simple subgroup of $P(\varpi_\beta)$).

\begin{lem}
Let $G$, $G'$ be two products of simply connected simple groups and $\Cbb^*$'s. Let $\beta$, $\beta'$ be two simple roots of two of the simple factors $G_0$ and $G'_0$ of $G$ and $G'$ respectively. And let $\alpha_0,\dots,\alpha_n$, respectively $\alpha'_0,\dots,\alpha'_n$ be simple roots of $G$, $G'$ both as in Case (1) of Theroem~\ref{th:main} (with the same integers $k$ and $i_1,\dots,i_k$).

Suppose that $G/P(\varpi_\beta)$ is isomorphic to $G'/P(\varpi_{\beta'})$ and that for any $l\in\{0,\dots,k\}$,  $$\begin{array}{l}
 \{(\dim P(\varpi_{\alpha_j})/(P(\varpi_\beta)\cap P(\varpi_{\alpha_j})),\dim P(\varpi_\beta)/(P(\varpi_\beta)\cap P(\varpi_{\alpha_j})))\,\mid\,j\in\{i_l,\dots,i_{l+1}-1\}\}= \\ 
 \{(\dim P(\varpi_{\alpha'_j})/(P(\varpi_{\beta'})\cap P(\varpi_{\alpha'_j})),\dim P(\varpi_{\beta'})/(P(\varpi_{\beta'})\cap P(\varpi_{\alpha'_j})))\,\mid\,j\in\{i_l,\dots,i_{l+1}-1\}\}.\end{array}$$

Then $G=G'$, $\beta=\beta'$ and for any $i\in\{0,\cdots,n\}$, $\alpha_i=\alpha'_i$ up to reordering the $\alpha_i$'s and $\alpha'_i$'s inside the sets $\{i_l,\dots,i_{l+1}-1\}$.
\end{lem}

\begin{proof} We proceed in several steps.
\vskip\baselineskip
\noindent\textbf{Step 1.} For any $l\in\{0,\dots,k\}$, $\alpha_{i_l}\not\in R_0$ if and only if $\alpha_{i_l'}\not\in R_0'$, and in that case, $\alpha_{i_l}$ and $\alpha_{i_l'}$ are both extremal simple roots of $\operatorname{SL}_{m+1}$ with $m=\dim P(\varpi_\beta)/(P(\varpi_\beta)\cap P(\varpi_{\alpha_j}))=\dim P(\varpi_{\beta'})/(P(\varpi_{\beta'})\cap P(\varpi_{\alpha_j'}))$. Indeed, we have that $\alpha_{i_l}\not\in R_0$ if and only if
$$\dim P(\varpi_{\alpha_{i_l}})/(P(\varpi_\beta)\cap P(\varpi_{\alpha_{i_l}}))=\dim G/P(\varpi_\beta)=\dim G/P(\varpi_{\beta'})=\dim P(\varpi_{\alpha_{i_l}'})/(P(\varpi_{\beta'})\cap P(\varpi_{\alpha_{i_l}'}))$$
and this is equivalent to saying that $\alpha_{i_l'}\not\in R_0'$. The second statement is obvious from the hypothesis on the $\alpha_i$'s and $\alpha'_i$'s. Note that $\alpha_{i_l+1},\dots,\alpha_{i_{l+1}-1}$ are in $R_0$ by hypothesis.

\vskip\baselineskip
\noindent\textbf{Step 2.} $G_0=G'_0$ and $\beta=\beta'$ up to symmetries of the Dynkin diagram. otherwise, $R_0$ and $R_0'$ are not empty and $\{(G_0,\varpi_\beta),(G_0',\varpi_{\beta'})\}$ is one of the three following sets up to symmetries of the Dynkin diagram (by \cite[Section~3.3]{Akhiezer}): $\{(\operatorname{Sp}_{2m},\varpi_1),(\operatorname{SL}_{2m},\varpi_1)\}$, $\{(\operatorname{Spin}_{2m+1},\varpi_m),(\operatorname{Spin}_{2m+2},\varpi_{m+1})\}$ or $\{(G_2,\varpi_1),(\operatorname{Spin}_{7},\varpi_{1})\}$. Let $\alpha_j\in R_0$, there exists $l\in\{0,\dots,k\}$ such that $j\in \{i_l,\dots,i_{l+1}-1\}$. By Step~1, $\alpha_j'\in R_0'$ and up to reordering $\alpha_i$'s and $\alpha'_i$'s in $\{i_l,\dots,i_{l+1}-1\}$ we can suppose that $\dim P(\varpi_{\alpha_{j}})/(P(\varpi_\beta)\cap P(\varpi_{\alpha_{j}}))=\dim P(\varpi_{\alpha_{j}'})/(P(\varpi_{\beta}')\cap P(\varpi_{\alpha_{j}'}))$ and $\dim P(\varpi_\beta)/(P(\varpi_\beta)\cap P(\varpi_{\alpha_j}))=\dim P(\varpi_{\beta'})/(P(\varpi_{\beta'})\cap P(\varpi_{\alpha_j'}))$. 
We have to check that this is not possible in the three cases. 

If $((G_0,\varpi_\beta),(G_0',\varpi_{\beta'}))$ is $((\operatorname{Sp}_{2m},\varpi_1),(\operatorname{SL}_{2m},\varpi_1))$ then $\varpi_{\alpha_j}$ is the fundamental weight $\varpi_2$ of $\operatorname{Sp}_{2m}$ (by the smooth condition) so that $\dim P(\varpi_{\alpha_{j}})/(P(\varpi_\beta)\cap P(\varpi_{\alpha_{j}}))=1$ and $\varpi_{\alpha_j'}$ has to be the fundamental weight $\varpi_2$ (by the smooth condition and because $\dim P(\varpi_{\alpha_{j}'})/(P(\varpi_{\beta'})\cap P(\varpi_{\alpha_{j}'}))=1$).
But then 
$$\dim P(\varpi_\beta)/(P(\varpi_\beta)\cap P(\varpi_{\alpha_j}))=2m-3<2m-2=\dim P(\varpi_{\beta'})/(P(\varpi_{\beta'})\cap P(\varpi_{\alpha_j'})).$$

If $((G_0,\varpi_\beta),(G_0',\varpi_{\beta'}))$ is $((\operatorname{Spin}_{2m+1},\varpi_m),(\operatorname{Spin}_{2m+2},\varpi_{m+1}))$ then  $\varpi_{\alpha_j}$ is the fundamental weight $\varpi_1$ or $\varpi_{m-1}$ of $\operatorname{Spin}_{2m+1}$. In both cases,  $\dim P(\varpi_\beta)/(P(\varpi_\beta)\cap P(\varpi_{\alpha_j}))=m-1$. But $\varpi_{\alpha_j'}$ is the fundamental weight $\varpi_1$  or $\varpi_{m}$ of $\operatorname{Spin}_{2m+2}$ so that $\dim P(\varpi_{\beta'})/(P(\varpi_{\beta'})\cap P(\varpi_{\alpha_j'}))=m$.

If $((G_0,\varpi_\beta),(G_0',\varpi_{\beta'}))$ is $((G_2,\varpi_1),(\operatorname{Spin}_{7},\varpi_{1}))$, then $\varpi_{\alpha_j}$ is the fundamental weight $\varpi_2$ of $G_2$ and $\varpi_{\alpha_j'}$ is the fundamental weight $\varpi_3$ of $\operatorname{Spin}_{7}$. But then
$$\dim P(\varpi_\beta)/(P(\varpi_\beta)\cap P(\varpi_{\alpha_j}))=1<3=\dim P(\varpi_{\beta'})/(P(\varpi_{\beta'})\cap P(\varpi_{\alpha_j'})).$$

We can now assume that $G_0=G_0'$ and $\beta=\beta'$. There are at most three simple subgroups of $P(\varpi_\beta)$ (their Dynkin diagram can be obtained from the Dynkin diagram of $G_0$ by removing $\beta$).

\vskip\baselineskip
\noindent\textbf{Step 3.} Let $\alpha_j\in R_0$ and $\alpha_j'\in R_0'$ such that $\dim P(\varpi_\beta)/(P(\varpi_\beta)\cap P(\varpi_{\alpha_j}))=\dim P(\varpi_{\beta})/(P(\varpi_{\beta})\cap P(\varpi_{\alpha_j'})).$ By the smooth condition, $\alpha_j$ and $\alpha_j'$ are extremal short simple roots of a simple subgroup of $P(\varpi_\beta)$ of type $A$ or $C$. We have then $\dim P(\varpi_\beta)/(P(\varpi_\beta)\cap P(\varpi_{\alpha_j}))=p$ (resp. $2p-1$) if the type is $A_p$ (resp. $C_{p}$). Hence, we have two cases: they are extremal short simple roots of simple subgroups of $P(\varpi_\beta)$ both of type $A_p$, or they are extremal short simple roots of simple subgroups  of $P(\varpi_\beta)$ of types $A_{2p-1}$ and $C_p$.

\vskip\baselineskip
\noindent\textbf{Step 4.} Suppose moreover that $\dim P(\varpi_{\alpha_j})/(P(\varpi_\beta)\cap P(\varpi_{\alpha_j}))=\dim P(\varpi_{\alpha_j'})/(P(\varpi_\beta)\cap P(\varpi_{\alpha_j'}))$, then one checks that $\alpha_j=\alpha_j'$ up to symmetries, by studying all cases up to symmetries, where $P(\varpi_\beta)$ has at least two simple subgroups of types $A_p$ and $A_p$ with $p\geq 1$, or $A_{2p-1}$ and $C_p$ with $p\geq 2$.
\vskip\baselineskip
\begin{center}
\begin{tabular}{|c|c|c|c|}
\hline
 Type of $G_0$ & $\varpi_\beta$ & $\varpi_{\alpha_j}$  &  $\dim P(\varpi_{\alpha_j})/(P(\varpi_\beta)\cap P(\varpi_{\alpha_j}))$ \\
\hline
 $A_m,\,m\geq 5$ & $\varpi_{\frac{m+1}{2}}$ & $\varpi_1$ or $\varpi_\frac{m+3}{2}$  & $\frac{(m+1)(m-1)}{4}$ or $\frac{m+1}{2}$, and\\
   $m$ odd &  &   &  $\frac{(m+1)(m-1)}{4}=\frac{m+1}{2}\frac{m-1}{2}\geq 2\frac{m+1}{2}$\\
\hline

 $B_3$ & $\varpi_{2}$ & $\varpi_{1}$ or $\varpi_{3}$  & $2$ or $3$
\\
\hline
 $B_6$ & $\varpi_{4}$ & $\varpi_{1}$ or $\varpi_{6}$  & $18$ or $8$
\\
\hline
$B_6$ & $\varpi_{4}$ & $\varpi_{3}$ or $\varpi_{6}$  & $5$ or $8$
\\
\hline
 $C_m,\,m\geq 3$ & $\varpi_{i}$ & $\varpi_{1}$ or $\varpi_{i+1}$  & $\frac{(4m-3i)(i-1)}{2}=\frac{3i(i-1)}{2}$ or $i$,
\\
  $m$ multiple of 3 & $i=\frac{2}{3}m$ & & and $\frac{3i(i-1)}{2}>i$ because $i\geq 2$\\
 \hline

 $C_m,\,m\geq 3$ & $\varpi_{i}$ & $\varpi_{i-1}$ or $\varpi_{i+1}$  & $2m-2i-1=i-1$ or $i$,
\\
  $m$ multiple of 3 & $i=\frac{2}{3}m$ & & \\
\hline
 $D_7$ & $\varpi_{4}$ & $\varpi_{1}$ or $\varpi_{7}$  & 21 or 12
\\
\hline
 $D_7$ & $\varpi_{4}$ & $\varpi_{3}$ or $\varpi_{7}$  & 6 or 12
\\
\hline
 $E_6$ & $\varpi_{4}$ & $\varpi_{1}$ or $\varpi_{5}$  & 15 or 6
\\
\hline
\end{tabular}
\end{center}
\end{proof}

\subsection{Case (2): the "second" Log MMP via moment polytopes}
\label{LogMMP2}
		
To describe the one-parameter family $(\tilde{Q}^\epsilon)_{\epsilon\in\Qbb_{\geq 0}}$ defined in Theorem~\ref{th:recallMMP}, we consider the basis $(u_i^*)_{i\in\{1,\dots,r\}}\cup(v_1^*)$ of $M$, where for any $i\in\{1,\dots,r\}$, $u_i^*=\varpi_{\alpha_i}-\varpi_{\alpha_0}+a_i\varpi_{\alpha_{r+2}}$ and $v_1^*=\varpi_{\alpha_{r+1}}-\varpi_{\alpha_{r+2}}$ and we define the matrices $\mathcal{A}$, $\mathcal{B}$ and $\mathcal{C}$ as follows 
$$\mathcal{A}= \left(
\begin{array}{cccccc}
-1 & \cdots & -1 & 0 \\
1 & 0 &\cdots&  0\\
0 & \ddots & \ddots &  \vdots\\

\vdots &\ddots & \ddots  & 0\\
0 & \cdots & 0 & 1\\
a_1 & \cdots & a_r & -1
  \end{array}
\right),\,\,\mathcal{B}= \left(
\begin{array}{c}
-1\\
0 \\
\vdots \\  \vdots \\
0\\

-1
  \end{array}
\right)\mbox{ and }\,\mathcal{C}= \left(
\begin{array}{c}
0 \\
\vdots \\  \vdots \\
\vdots\\
0\\
1
  \end{array}
\right).$$

Then $\tilde{Q}^\epsilon=\{x\in M_\Qbb\,\mid\,\mathcal{A}x\geq \mathcal{B}+\epsilon \mathcal{C}\}$ is the set of $x=(x_1,\dots,x_n)$ such that $x_1,\dots,x_n$ are non-negative, $x_1+\cdots+x_r\leq 1$ and $a_1x_1+\cdots+a_rx_r-x_{r+1}-\cdots -x_n\geq \epsilon-1$.

In particular, $\tilde{Q}^\epsilon$ is the intersection of the closed half-space $H_+^\epsilon:=\{x\in M_\Qbb\,\mid\, a_1x_1+\cdots+a_rx_r-x_{r+1}\geq \epsilon-1\}$ with $\tilde{Q}^0$. We denote by  $H_{++}^\epsilon$ the interior of $H_{+}^\epsilon$ and by $H^\epsilon$ the hyperplane $H_{+}^\epsilon\backslash H_{++}^\epsilon$.

 \begin{ex}
If $n=2$ (i.e., $r=s=1$) we have $a_1>0$, and either $\alpha_1$ is trivial or not. We draw, in Figure~\ref{figure7}, such a polytope for $\epsilon=0$ with the hyperplane $H^0:=\{x\in M_\Qbb\,\mid\, a_1x_1-x_2=-1\}$. 

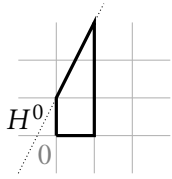
\begin{figure}[!ht]

\begin{center}
\begin{tikzpicture}[scale=0.5]
\chadombis
\draw[very thick ] (1,1) -- (2,1) -- (2,4) -- (1,2) -- (1,1);
%\node at (3.5,3.5) {$\tilde{Q}^0$};
\draw[densely dotted] (0,0) -- (2.25,4.5) ;
\node at (0.2,1.5) {$H^0$};
\end{tikzpicture}
\caption{The polytope $\tilde{Q}^0$ in the case where $a_1=2$}
\label{figure7}
\end{center}
\end{figure}

\end{ex}
 
 Note that $\tilde{Q}^0$ is a polytope with vertices $u_0^*:=0$, $u_1^*,\dots,u_r^*$, $u_0^*+(1+a_0)v_0^*,\dots,u_r^*+(1+a_r)v_1^*$ (recall that $a_0=0$) and facets $F_I:=\operatorname{Conv}((u_i^*\,\mid\,i\not\in I)\cup(u_i^*+(1+a_i)v_1^*\,\mid\,i\not\in I))$, $F_{I,1}:=\operatorname{Conv}(u_i^*\,\mid\,i\not\in I)$ and $F_{I,2}:=\operatorname{Conv}(u_i^*+(1+a_i)v_1^*\,\mid\,i\not\in I)$ with $I\subsetneq\{0,\dots,r\}$. 
 In particular, the facets of $\tilde{Q}^0$ are the $F_i:=F_{\{i\}}$ with $i\in\{0,\dots,r\}$, $F_{\emptyset,1}$  and $F_{\emptyset,2}$. Moreover for any 
$I\subsetneq \{0,\dots,n\}$, $F_I=\bigcap_{i\in I}F_i$, $F_{I,1}=\bigcap_{i\in I}F_i\cap F_{\emptyset,1}$ and $F_{I,2}=\bigcap_{i\in I}F_i\cap F_{\emptyset,2}$.

Then, for any $I\subsetneq \{0,\dots,r\}$, we define $F_I^\epsilon:=F_I\cap H_+^\epsilon$, $F_{I,1}^\epsilon:=F_{I,1}\cap H_+^\epsilon$, $F_{I,2}^\epsilon:=F_I\cap H^\epsilon$ and finally $F_{I,1,2}^\epsilon:=F_{I,1}\cap H^\epsilon$. They are faces (possibly empty and not distinct) of $\tilde{Q}^\epsilon$.
 (Recall $0=a_0<a_1<\cdots<a_r$ and $n=r+1$.)

\begin{prop}\label{prop:polyfaces2}
The polytope $\tilde{Q}^\epsilon$ is of dimension $n$ if and only if $\epsilon<1+a_r$.

\vspace{0.1cm}

Suppose now that $\epsilon<1+a_r$.
The non-empty faces of $\tilde{Q}^\epsilon$ are the distinct following $F_I^\epsilon$, $F_{I,1}^\epsilon$, $F_{I,2}^\epsilon$ and  $F_{I,1,2}^\epsilon$ with $I\subsetneq \{0,\dots,r\}$:
\begin{itemize}[label=$\ast$]
\item $F_I^\epsilon$ (of codimension $|I|$) if $\epsilon<\max_{i\not\in I}(1+a_i)$;
\item $F_{I,1}^\epsilon$ (of codimension $|I|+1$) if $\epsilon<\max_{i\not\in I}(1+a_i)$;
 \item $F_{I,2}^\epsilon$ (of codimension $|I|+1$) if $ \epsilon< \max_{i\not\in I}(1+a_i)$;
 \item $F_{I,1,2}^\epsilon$ (of codimension $|I|+2$, respectively $|I|+1$) if $\min_{i\not\in I}(1+a_i)< \epsilon< \max_{i\not\in I}(1+a_i)$, respectively if $\epsilon=\min_{i\not\in I}(1+a_i)=\max_{i\not\in I}(1+a_i)$.
\end{itemize}
In particular, the facets of $\tilde{Q}^\epsilon$ are: $F_i^\epsilon$ with $i\in\{0,\dots,r-1\}$, $F_r^\epsilon$ if $\epsilon<1+a_{r-1}$,
%if there exists $j\neq i$ such that $\epsilon<1+a_j$,  
$F_{\emptyset,1}^\epsilon$ and $F_{\emptyset,2}^\epsilon$. Moreover, we can write any face of $\tilde{Q}^\epsilon$ as the intersection of all the facets that contain it, as follows. \\
For any $I\subsetneq \{0,\dots,r\}$ such that $\epsilon<\max_{i\not\in I}(1+a_i)$, $F_I^\epsilon=\bigcap_{i\in I}F_i^\epsilon$.\\  
For any $I\subsetneq \{0,\dots,r\}$ such that $\epsilon<\max_{i\not\in I}(1+a_i)$, $F_{I,1}^\epsilon=\bigcap_{i\in I}F_i^\epsilon\cap F_{\emptyset,1}^\epsilon$.\\  
For any $I\subsetneq \{0,\dots,r\}$ such that $\epsilon<\max_{i\not\in I}(1+a_i)$, $F_{I,2}^\epsilon=\bigcap_{i\in I}F_i^\epsilon\cap F_{\emptyset,2}^\epsilon$.\\  
For any $I\subsetneq \{0,\dots,r\}$ such that $\min_{i\not\in I}(1+a_i)\leq \epsilon\leq \max_{i\not\in I}(1+a_i)$, $F_{I,1,2}^\epsilon= \bigcap_{i\in I}F_i^\epsilon\cap F_{\emptyset,1}^\epsilon\cap F_{\emptyset,2}^\epsilon$.
\end{prop}

 Remark that, if $\epsilon=\min_{i\not\in I}(1+a_i)=\max_{i\not\in I}(1+a_i)$, then $I=\{0,\dots,r\}\backslash \{i\}$ where $i$ is such that $\epsilon=1+a_i$.
 
 Note also that $\tilde{Q}^{1+a_r}$ is the point $u_r^*$ so that $Q^{1+a_r}$ is the point $\varpi_{\alpha_r}$.
 
 \begin{proof}
 For any $\epsilon\geq 0$, the polytope $\tilde{Q}^\epsilon$ is of dimension~$n$  if and only if $\tilde{Q}^0$ intersects $H_{++}^\epsilon$
  if and only if there exists $i\in\{0,\dots,r\}$ such that $u_i^*$ (or $u_i^*+(1+a_i)v_1^*$) is in $H_{++}^\epsilon$ 
  if and only if there exists $i\in\{0,\dots,r\}$ such that $a_i>\epsilon-1$ (or $-1>\epsilon-1$) 
  if and only if $a_r>\epsilon-1$. This proves the first statement of the proposition.
  
  Suppose now that $\epsilon<1+a_r$. A non-empty face of $\tilde{Q}^\epsilon$ is either the intersection with $H_{+}^\epsilon$ of a non-empty face of $\tilde{Q}^0$ that intersects $H_{++}^\epsilon$, or the intersection of a non-empty face of $\tilde{Q}^0$ with $H^\epsilon$.
  
 Let $I\subsetneq \{0,\dots,r\}$. The set $F_I^\epsilon$ is not empty if and only if there exists $i\not\in I$ such that $u_i^*$ (or $u_i^*+(1+a_i)v_1^*$) is in $H_{+}^\epsilon$ if and only if there exists $i\not\in I$ such that $a_i\geq \epsilon-1$ (or $-1\geq\epsilon-1$) if and only if $\epsilon\leq\max_{i\not\in I}(1+a_i)$. Moreover with the same argument, $F_I^\epsilon$ is not empty and intersects $H_{++}^\epsilon$ if and only if $\epsilon<\max_{i\not\in I}(1+a_i)$. Also, in that case, the dimension of $F_I^\epsilon$ is the same as the dimension of $F_I$; in particular the non-empty $F_I^\epsilon$ that intersect $H_{++}^\epsilon$ are all distinct.

Similarly, $F_{I,1}^\epsilon$ is not empty if and only if there exists $i\not\in I$ such that $u_i^*\in H_{+}^\epsilon$ if and only if there exists $i\not\in I$ such that $a_i\geq \epsilon-1$ if and only if $\epsilon\leq\max_{i\not\in I}(1+a_i)$. Also, $F_{I,1}^\epsilon$ is not empty and intersects $H_{++}^\epsilon$ if and only if $\epsilon<\max_{i\not\in I}(1+a_i)$. In that case, the dimension of $F_{I,1}^\epsilon$ is the same as the dimension of $F_{I,1}$; in particular the non-empty $F_{I,1}^\epsilon$ that intersect $H_{++}^\epsilon$ are all distinct and also distinct from the non-empty $F_{I}^\epsilon$.

Let $I\subsetneq \{0,\dots,r\}$. Note that for any $\epsilon\geq 0$ (respectively $\epsilon>0$) and for any $i\in\{0,\dots,r\}$, $u_i^*+(1+a_i)v_1^*\not\in H_{++}^\epsilon$ (respectively $u_i^*+(1+a_i)v_1^*\not\in H_{+}^\epsilon$). Then the set $F_{I,2}^\epsilon$ is not empty
if and only if there exists $i\not\in I$ such that $u_i^*\in H_{+}^\epsilon$ if and only if there exists $i\not\in I$ such that $a_i\geq \epsilon-1$ if and only if $\epsilon\leq\max_{i\not\in I}(1+a_i)$. Moreover, $F_{I,2}^\epsilon$ is not empty and $H^\epsilon$ intersects  $F_I$ in its relative interior if and only if there exists $i\not\in I$ such that $a_i> \epsilon-1$ if and only if $\epsilon<\max_{i\not\in I}(1+a_i)$. Hence, the dimension of $F_{I,2}^\epsilon$ is the  dimension of $F_{I}$ minus 1 if $\epsilon<\max_{i\not\in I}(1+a_i)$ and it equals the dimension of $F_{I}$ if $\epsilon=\max_{i\not\in I}(1+a_i)$. In the first case, the $F_{I,2}^\epsilon$ are all distinct and yield all non-empty faces of $\tilde{Q}^\epsilon$ included in $H^\epsilon$ but not in $F_{\emptyset,1}$. In the second case, $F_{I,2}^\epsilon=F_{I,1,2}^\epsilon$.

Now, the set $F_{I,1,2}^\epsilon$ is not empty
if and only if there exist $i$ and $j$ not in $I$ (may be equal) such that $u_i^*\in H_{+}^\epsilon$ and $u_j^*\not\in H_{++}^\epsilon$ if and only if there exist $i$ and $j$ not in $I$ such that $a_i\geq \epsilon-1$ and $a_j\leq \epsilon-1$ if and only if $\min_{i\not\in I}(1+a_i)\leq\epsilon\leq\max_{i\not\in I}(1+a_i)$. Moreover, $F_{I,1,2}^\epsilon$ is not empty and included in no proper face of $F_{I,1}$ if and only if there exist $i$ and $j$ not in $I$ such that $u_i^*\in H_{++}^\epsilon$ and $u_j^*\not\in H_{+}^\epsilon$ if and only if there exist $i$ and $j$ not in $I$ such that $a_i> \epsilon-1$ and $a_j< \epsilon-1$ (i.e., $a_i< \epsilon-1$ and $a_j> \epsilon-1$) or for any $i\not\in I$ we have $u_i^*\in H^\epsilon$ (i.e., $a_i=\epsilon-1$). Then  $F_{I,1,2}^\epsilon$ is not empty and included in no proper face of $F_{I,1}$ if and only if $\min_{i\not\in I}(1+a_i)< \epsilon<\max_{i\not\in I}(1+a_i)$ or $\epsilon=\min_{i\not\in I}(1+a_i)=\max_{i\not\in I}(1+a_i)$.  In particular, the dimension of $F_{I,1,2}^\epsilon$ is the dimension of $F_{I,1}$ minus 1 if $\min_{i\not\in I}(1+a_i)<\epsilon<\max_{i\not\in I}(1+a_i)$ and it equals the dimension of $F_{I,1}$ if $\epsilon=\min_{i\not\in I}(1+a_i)=\max_{i\not\in I}(1+a_i)$. Note also that the non-empty $F_{I,1,2}^\epsilon$ that are not included in a proper face of $F_{I,1}$ are all distinct and yield all non-empty faces of $\tilde{Q}^\epsilon$ included in $H^\epsilon\cap F_{\emptyset,1}$. This finishes the proof of the second statement of the proposition.

 To get the last statements, use again the fact that a facet is a face of codimension~1 and that any face of a polytope is the intersection of the facets containing it.
\end{proof}

From Proposition~\ref{prop:polyfaces2}, we deduce the following result.

\begin{cor}
The isomorphism classes of the horospherical varieties $X^\epsilon$ associated to the polytopes in the family $(Q^\epsilon)_{\epsilon\in\Qbb_{\geq 0}}$ are given by the following subsets of $\Qbb_{\geq 0}$:
\begin{itemize}[label=$\ast$]
\item $[0,1[$;
\item $]1+a_{i},1+a_{i+1}[$ for any $i\in\{0,\dots,r-2\}$;
\item $\{1+a_i\}$ for any $i\in\{0,\dots,r-2\}$;
\item $]1+a_{r-1},1+a_r[$ 
and $\{1+a_{r-1}\}$ if the simple root $\alpha_r$ is not trivial (i.e., when $X$ is as in Case (2b) of Theorem~\ref{th:main});
\item $[1+a_{r-1},1+a_r[$ if the simple root $\alpha_r$ is trivial (i.e., when $X$ is as in Case (2c) of Theorem~\ref{th:main}). 
\end{itemize}
\end{cor}

  \begin{proof}
  
We apply the theory described in Section~\ref{sec:recallMMP}, in particular the fact that the isomorphism classes of the varieties $X^\epsilon$ are obtained by looking at the $\epsilon$'s for which ``the faces of $Q^\epsilon$ change''. Note first that, by Proposition~\ref{prop:polyfaces2}, $(P,M,Q^\epsilon,\tilde{Q}^\epsilon)$ is an admissible quadruple if and only if $\epsilon<1+a_r$. Also, the facets of $\tilde{Q}^\epsilon$ are: $F_i^\epsilon$ (orthogonal to $\alpha_{i,M}^\vee$) with $i\in\{0,\dots,r-1\}$, $F_r^\epsilon$ (orthogonal to $\alpha_{r,M}^\vee$) if $\epsilon<1+a_{r-1}$,
%if there exists $j\neq i$ such that $\epsilon<1+a_j$,  
$F_{\emptyset,1}^\epsilon$ (orthogonal to $\alpha_{r+1,M}^\vee$) and $F_{\emptyset,2}^\epsilon$ (orthogonal to $\alpha_{r+2,M}^\vee$). In particular,  for any $\epsilon,\,\eta\in  [0,1+a_r[$, the facets of $Q^\epsilon$ and $Q^\eta$ are ``the same'' if and only if $\epsilon$ and $\eta$ are both in $[0,1+a_{r-1}[$ or $[1+a_{r-1},1+a_r[$. 

We now use a consequence of the proof of Proposition~\ref{prop:polyfaces2}: for any $I\subsetneq \{0,\dots,r\}$, $\bigcap_{i\in I}F_i^\epsilon$ is not empty if and only if $\epsilon\leq\max_{i\not\in I}(1+a_i)$, $F_{\emptyset,1}^\epsilon\cap\bigcap_{i\in I}F_i^\epsilon$ is not empty if and only if $\epsilon\leq\max_{i\not\in I}(1+a_i)$, $F_{\emptyset,2}^\epsilon\cap\bigcap_{i\in I}F_i^\epsilon$ is not empty if and only if $\epsilon\leq\max_{i\not\in I}(1+a_i)$,
and $F_{\emptyset,1,2}^\epsilon\cap\bigcap_{i\in I}F_i^\epsilon$ is not empty if and only if we have
$\min_{i\not\in I}(1+a_i)\leq\epsilon\leq\max_{i\not\in I}(1+a_i)$.  In particular, for any $i\in\{0,\dots,r-2\}$, suppose that for $I=\{i+1,\dots,r\}$ and that $\bigcap_{i\in I}F_i^\epsilon$ is not empty; suppose also that for $I=\{0,\dots,i-1\}$ and that $F_{\emptyset,1,2}^\epsilon\cap\bigcap_{i\in I}F_i^\epsilon$ is not empty; then $\epsilon=1+a_i$. Similarly for any $i\in\{0,\dots,r-2\}$, suppose that for $I=\{i+2,\dots,n\}$ and that $\bigcap_{i\in I}F_i^\epsilon$ is not empty; suppose also that for $I=\{0,\dots,i-1\}$ and that $F_{\emptyset,1,2}^\epsilon\cap\bigcap_{i\in I}F_i^\epsilon$ is not empty; then $\epsilon\in [1+a_i,1+a_{i+1}]$. 

Hence, this proves that if two varieties $X^\epsilon$ and $X^\eta$ are isomorphic then $\epsilon$ and $\eta$ are a one of the subsets described in the corollary.

To conclude, we have to prove that the two varieties $X^\epsilon$ and $X^\eta$ are isomorphic when $\epsilon$ and $\eta$ are in one of these subsets. It is obvious from Proposition~\ref{prop:polyfaces2} except in the case where the simple root $\alpha_n$ is trivial. But in that case, all polytopes $Q^\epsilon$ with $\epsilon\in[1+a_{r-1},1+a_r[$ could be defined even deleting the row corresponding to the simple root $\alpha_r$ that is trivial, so that their faces are ``the same'' (they are simplexes with facets $F_i^\epsilon$ for $i\in\{0,\dots,r-1\}$, $F_{\emptyset,1}^\epsilon$ and $F_{\emptyset,2}^\epsilon$).
\end{proof}

 We can reformulate this corollary as follows, and get the first statement of Theorem~\ref{th:main2} in Case (2). We denote $X_0=X$ and for any $i\in\{1,\cdots, r\}$, $X^i:=X^\epsilon$ with $\epsilon\in ]1+a_{i-1},1+a_i[$ and for any $i\in\{0,\cdots, r\}$, $Y^i:=X^{1+a_i}$.

\begin{cor}
The family $(Q^\epsilon)_{\epsilon\in\Qbb_{\geq 0}}$ describes a Log MMP from $X$ as follows: 
\begin{itemize}[label=$\ast$]
\item $r$ flips $\phi_i:\,X^i\longrightarrow Y^i\longleftarrow X^{i+1}\,:\phi_i^+$ for any $i\in\{0,\cdots, r-1\}$ and a fibration $\phi_r:\,X^r\longrightarrow Y^r$, if the simple root $\alpha_r$ is not trivial;
\item $r-1$ flips $\phi_i:\,X^i\longrightarrow Y^i\longleftarrow X^{i+1}\,:\phi_i^+$ for any $i\in\{0,\cdots, r-2\}$, followed by a divisorial contraction $\phi_{r-1}:\,X^{r-1}\longrightarrow Y^{r-1}\simeq X^r$ and a fibration $X^r\longrightarrow Y^r\simeq\operatorname{pt}$, if the simple root $\alpha_r$ is trivial.
\end{itemize}
\end{cor}

\begin{ex}
In the two different cases with $n=2$ and $a_1=2$, we illustrate this corollary in terms of polytopes in Figures~\ref{figure8} and \ref{figure9}.
\end{ex}
\begin{figure}[!ht]
\begin{center}
\begin{tikzpicture}
\node [rectangle] (a) at (0,0) {\begin{tikzpicture}[scale=0.5]
\chadombis
\draw[very thick ] (1,1) -- (2,1) -- (2,4) -- (1,2) -- (1,1);
%\node at (3.5,3.5) {$\tilde{Q}^0$};
\draw[densely dotted] (0,0) -- (2.25,4.5) ;
\node at (0.2,1.5) {$H^0$};
\end{tikzpicture}
};

\node [rectangle] (b) at (2,-3.25) {
    \begin{tikzpicture}[scale=0.5]
     \chadombis
\draw[very thick] (1,1) -- (2,1) -- (2,3) -- (1,1);
%\node at (3.5,3.5) {$\tilde{Q}^0$};
\draw[densely dotted] (0.5,0) -- (2.5,4) ;
\node at (2.9,3.5) {$H^1$};
    \end{tikzpicture}
};

\node [rectangle] (c) at (4,0) {
    \begin{tikzpicture}[scale=0.5]
     \chadombis
     \draw[very thick] (1.5,1) -- (2,1) -- (2,2) -- (1.5,1);
%\node at (3.5,3.5) {$\tilde{Q}^0$};
\draw[densely dotted] (1,0) -- (3,4) ;
\node at (3,2.7) {$H^2$};
    \end{tikzpicture}
};

\node [rectangle] (d) at (6,-3.25) {
    \begin{tikzpicture}[scale=0.5]
     \chadombis
\node at (2,1) {$\bullet$};
%\node at (3.5,3.5) {$\tilde{Q}^0$};
\draw[densely dotted] (1.5,0) -- (3.5,4) ;
\node at (3.5,2.5) {$H^3$};
    \end{tikzpicture}
};
\draw (a) [->]  to node[above right] {$\phi_0$} (b);
\draw (c) [->]  to node[above left] {$\phi_0^+$} (b);
\draw (c) [->]  to node[above right] {$\phi_1$} (d);
\end{tikzpicture}

\caption{The Log MMP described by the polytopes $\tilde{Q}^\epsilon$ in the case where $n=2$, $a_1=2$ and $\alpha_1$ is not trivial.}
\label{figure8}
\end{center}
\end{figure}
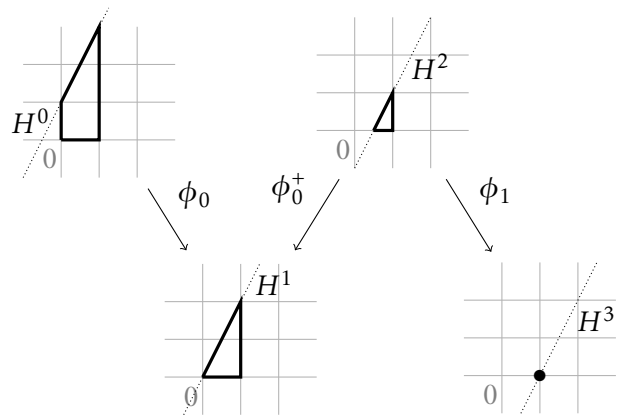

\begin{figure}[ht]
\begin{center}
\begin{tikzpicture}
\node [rectangle] (a) at (0,0) {\begin{tikzpicture}[scale=0.5]
\chadombis
\draw[very thick ] (1,1) -- (2,1) -- (2,4) -- (1,2) -- (1,1);
%\node at (3.5,3.5) {$\tilde{Q}^0$};
\draw[densely dotted] (0,0) -- (2.25,4.5) ;
\node at (0.2,1.5) {$H^0$};
\end{tikzpicture}
};

\node [rectangle] (b) at (3.5,-2.5) {
    \begin{tikzpicture}[scale=0.5]
     \chadombis
\draw[very thick] (1,1) -- (2,1) -- (2,3) -- (1,1);
%\node at (3.5,3.5) {$\tilde{Q}^0$};
\draw[densely dotted] (0.5,0) -- (2.5,4) ;
\node at (2.9,3.5) {$H^1$};
    \end{tikzpicture}
};

\node [rectangle] (c) at (7,-5) {
\begin{tikzpicture}[scale=0.5]
     \chadombis
\node at (2,1) {$\bullet$};
%\node at (3.5,3.5) {$\tilde{Q}^0$};
\draw[densely dotted] (1.5,0) -- (3.5,4) ;
\node at (3.5,2.5) {$H^3$};
    \end{tikzpicture}
};

\draw (a) [->]  to node[above right] {$\phi_0$} (b);
\draw (b) [->]  to node[above right] {$\phi_1$} (c);
\end{tikzpicture}

\caption{The Log MMP described by the polytopes $\tilde{Q}^\epsilon$ in the case where $n=2$, $a_1=2$ and $\alpha_1$ is trivial.}
\label{figure9}
\end{center}
\end{figure}
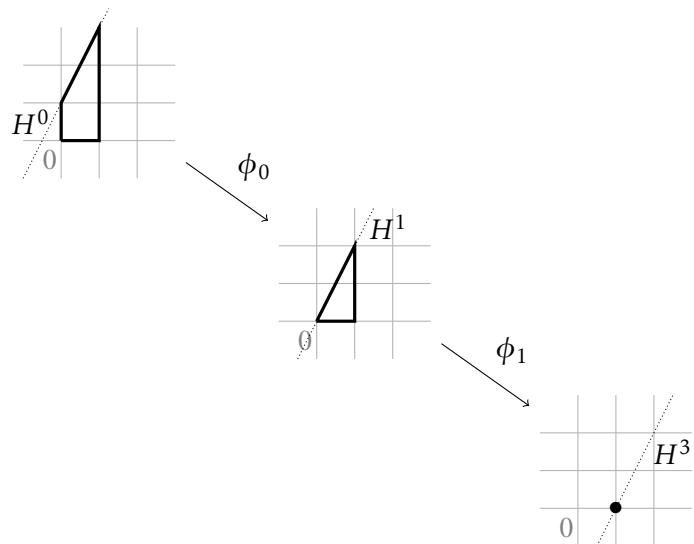

\subsection{Proof of the last statement of Theorem~\ref{th:main2} in Case (2)}

The previous section proves that $a_1,\dots,a_r$ are invariants of $X$. To finish the proof of Theorem~\ref{th:main2} in Case (2), we have to prove that $G_0,\dots,G_t$ and $\alpha_0,\dots,\alpha_{r+2}$ are also invariants. Since the "first" Log MMP consists of a fibration $\psi:\,X\longrightarrow Z$ where $Z$ is a  two-orbit variety embedded in $\Pbb(V(\varpi_{\alpha_{r+1}})\oplus V(\varpi_{\alpha_{r+2}}))$ as in \cite{2orbits}, $G_t$, $\alpha_{r+1}$ and $\alpha_{r+2}$ are invariants of $X$. As in Case (1), we will describe some exceptional loci and some fibers of different morphisms of the Log MMP, but we first distinguish two cases by the following result.
\vskip 1cm

\begin{prop}
Suppose that $r=1$ and that $\alpha_0$ and $\alpha_1$ are two simple roots of $G_0$ (and then $t=1$). Then, the general fiber of $\psi:\,X\longrightarrow Z$ is either a homogeneous variety different from a projective space (a quadric $Q^{2m}$ with $m\geq 2$, a Grassmannian $\operatorname{Gr}(i,m)$ with $m\geq 5$ and $2\leq i\leq m-2$, or a spinor variety $\operatorname{Spin}(2m+1)/P(\varpi_m)$ with $m\geq 4$), or a two-orbit variety as in \cite{2orbits}.

Suppose that $r>1$ or that $\alpha_0$ and $\alpha_1$ are simple roots of $G_0$ and $G_1$ respectively. Then, the general fiber of $\psi:\,X\longrightarrow Z$ is a projective space.

\end{prop}

\begin{proof}
The general fiber of $\psi:\,X\longrightarrow Z$ is the smooth projective horospherical $G_0\times\cdots\times G_{t-1}$-variety of Picard group $\Zbb$ isomorphic to the closure of the $G_0\times\cdots\times G_{t-1}$-orbit of a sum of  highest weight vectors in $\mathbb{P}:=\mathbb{P}(V(\varpi_{\alpha_0})\oplus\cdots\oplus V(\varpi_{\alpha_r}))$. Hence, the proposition is a consequence of \cite[Section~1]{2orbits}.
\end{proof}

$\bullet$ In the case where $r=1$ and that $\alpha_0$ and $\alpha_1$ are two simple roots of $G_0$, $G=G_0\times G_1$ and the description of the general fiber of $\psi:\,X\longrightarrow G/P(\varpi_\beta)$, with Remark~\ref{rem:2orbits}, implies that $G_0$, $\alpha_0$ and $\alpha_1$ are invariants of $X$.

\vskip\baselineskip
$\bullet$ Now we suppose that $r>1$ or that $\alpha_0$ and $\alpha_1$ are not two simple roots of the same simple subgroup of $P(\varpi_\beta)$.

\vskip\baselineskip
We define various exceptional loci in $X$ as follows. Let $i\in\{0,\dots,r\}$, define $E_i$ to be the closure in $X$ of the set of points $x\in X$ such that $x$ is in the open isomorphism set of the first $i$ contractions and $x$ is in the exceptional locus of $\phi_i$.

\begin{prop}
For any $i\in\{0,\dots,r\}$ the exceptional locus $E_i$ is the closure in $X$ of the $G$-orbit associated to the non-empty face $F_{I_i}$ with $I_i:=\{i+1,\dots,r\}$. In particular $E_i$ is isomorphic to the closure of the $G$-orbit of a sum of  highest weight vectors in $$\mathbb{P}:=\mathbb{P}\left(\bigoplus_{j=0}^{i}\bigoplus_{b=0}^{1+a_j}V(\varpi_{\alpha_j}+b\varpi_{\alpha_{r+1}}+(1+a_j-b)\varpi_{\alpha_{r+2}})\right),$$ hence for $i\in\{1,\dots,r\}$, $E_i$ is a smooth projective horospherical variety of Picard group $\Zbb^2$ as in Case (2), and $E_0$ is the product a two-orbit variety with a homogeneous variety (projective of Picard group $\Zbb$).
\end{prop}

Note that $E_r=X$ and that in any case, the rank of the horospherical $G$-variety $E_i$ is $i+1$.

 \begin{proof} Let $i\in\{0,\dots,r\}$ and $\epsilon_i\in\Qbb_{\geq 0}$ such that $X^i=X^{\epsilon_i}$.
 
We denote by $\Omega_I^i$, $\Omega_{I,1}^i$, $\Omega_{I,2}^i$ and $\Omega_{I,1,2}^i$ the $G$-orbits of $X^i$ associated to the non empty faces $F_I^{\epsilon_i}$,  $F_{I,1}^{\epsilon_i}$, $F_{I,2}^{\epsilon_i}$ and $F_{I,1,2}^{\epsilon_i}$ of the polytope $\tilde{Q}^{\epsilon_i}$. We denote by $\omega_I^i$, $\omega_{I,1}^i$, $\omega_{I,2}^i$ and $\omega_{I,1,2}^i$ the $G$-orbits of $Y^i=X^{1+a_i}$ associated to the non-empty faces $F_I^{1+a_i}$, $F_{I,1}^{1+a_i}$, $F_{I,2}^{1+a_i}$ and $F_{I,1,2}^{1+a_i}$ of the polytope $\tilde{Q}^{1+a_i}$.  Recall that, for any $\epsilon\in\Qbb_{\geq 0}$, we have an order on the $G$-orbits of $X^\epsilon$ compatible with the order on the non-empty faces of $\tilde{Q}^\epsilon$: in particular $\Omega_I^i\subset \overline{\Omega_{I'}^i}$, $\Omega_{I,1}^i\subset \overline{\Omega_{I',1}^i}$, $\Omega_{I,2}^i\subset \overline{\Omega_{I',2}^i}$ and $\Omega_{I,1,2}^i\subset \overline{\Omega_{I',1,2}^i}$ respectively if and only if $I'\subset I$, and $\Omega_{I,1}^i\subset \overline{\Omega_I^i}$, $\Omega_{I,2}^i\subset \overline{\Omega_I^i}$, $\Omega_{I,1,2}^i\subset \overline{\Omega_{I,1}^i}$ and $\Omega_{I,1,2}^i\subset \overline{\Omega_{I,2}^i}$ (as soon as these orbits are defined, i.e., as soon as the corresponding faces are non-empty).

 For any $I\subsetneq\{0,\dots,r\}$ such that there exists $j\geq i$ not in $I$ (i.e., such that $\Omega_I^i$ is defined), $\phi_i(\Omega_I^i)=\omega_I^i$ if there exists $j\geq i+1$ not in $I$, and $\phi_i(\Omega_I^i)=\omega_{I\backslash\{i\},1,2}^i$ if for any $j\geq i+1$, $j\in I$. Indeed $I\cup\{0,\dots i-1\}=I\backslash\{i\}$ is the minimal subset of $\{0,\dots,r\}$ containing $I$ such that $\omega_{I\backslash\{i\},1,2}^i$ is defined and there is no $I'$ containing $I$ such that $\omega_{I'}^i$, $\omega_{I',1}^i$ or $\omega_{I',2}^i$ is defined.
Similarly, with $k=1$ or $2$, for any $I\subsetneq\{0,\dots,r\}$ such that there exists $j\geq i$ not in $I$ (i.e., such that $\Omega_{I,k}^i$ is defined), $\phi_i(\Omega_{I,k}^i)=\omega_{I,k}^i$ if there exists $j\geq i+1$ not in $I$, and $\phi_i(\Omega_{I,k}^i)=\omega_{\{0,\dots,r\}\backslash\{i\},1,2}^i$ if for any $j\geq i+1$, $j\in I$. Indeed $I\cup\{0,\dots i-1\}=\{0,\dots,r\}\backslash\{i\}$ is the minimal subset of $\{0,\dots,r\}$ containing $I$ such that $\omega_{\{0,\dots,r\}\backslash\{i\},1,2}^i$ is defined and there is no $I'$ containing $I$ such that $\omega_{I',k}^i$ is defined.

And for any $I\subsetneq\{0,\dots,r\}$ such that there exist $j\geq i$ and $j'<i$ not in $I$ (i.e., such that $\Omega_{I,1,2}^i$ is defined), $\phi_i(\Omega_{I,1,2}^l)=\omega_{I,1,2}^i$ if there exists $i\geq i+1$ not in $I$, and $\phi_i(\Omega_{I,1,2}^i)=\omega_{\{0,\dots,r\}\backslash\{i\},\beta}^i$ if for any $j\geq i+1$, $j\in I$. Indeed $\{0,\dots,r\}\backslash\{i\}=I\cup\{0,\dots i_l-1\}$ is the minimal subset of $\{0,\dots,n\}$ containing $I$ such that $\omega_{\{0,\dots,r\}\backslash\{i\},1,2}^i$ is defined.

In particular, we have $\phi_i(\Omega_{I_i}^i)=\omega_{I\backslash\{i\},1,2}^i$. But $\Omega_{I_i}^i$ and $\omega_{\{0,\dots,r\}\backslash\{i\},1,2}^i$ are non-isomorphic horospherical homogeneous spaces by Proposition~\ref{prop:classifpoly}, so that $\Omega_{I_i}^i$ is in the exceptional locus of $\phi_l$. Moreover, if $\Omega$ is a $G$-orbit of $X^i$ not contained in $\overline{\Omega_{I_i}^i}$, it is of the form $\Omega_I^i$, $\Omega_{I,1}^i$,  $\Omega_{I,2}^i$ or $\Omega_{I,1,2}^i$ where $I_i\not\subset I$. Hence, in that case $\phi_i(\Omega)=\Omega$. And then the exceptional locus of $\phi_i$ is $\overline{\Omega_{I_i}^i}$. Note that $\Omega_{I_i}^0,\dots,\Omega_{I_i}^{l-1}$ are not in the exceptional locus of $\phi_0,\dots,\phi_{i-1}$ respectively, to conclude that $E_i=\overline{\Omega_{I_i}^0}$.

We use again Proposition~\ref{prop:classifpoly} to see that $E_i=\overline{\Omega_{I_i}^0}$ corresponds to the admissible quadruple $(P_F,M_F,F,\tilde{F})$ with $F=F_{I_i}^0$ (and with some ample divisor of $E_i$). Then we conclude by Corollaries~\ref{cor:smoothembedd}  and~\ref{cor:smooth}.
\end{proof}

 The Log MMP now defines, by restriction, fibrations $\tilde{\phi_i}:E_i\backslash E_{i-1}\longrightarrow E'_i:=\overline{\omega_{\{0,\dots,r\}\backslash\{i\},1,2}^i}$, for any $i\in\{0,\dots,i\}$.

\begin{prop} 
For any $i\in\{0,\dots,r\}$, $E'_i$ is a closed $G$-orbit of $Y^i$ isomorphic to $G/P(\varpi_{\alpha_i})$ (which is a point if $\alpha_i$ is trivial). In particular, the map $\tilde{\phi_i}$ is surjective. Moreover, the dimension of fibers of $\tilde{\phi_i}$ is $$i+1+\dim P(\varpi_{\alpha_i})/(P(\varpi_{\alpha_{r+1}})\cap P(\varpi_{\alpha_{r+2}})\cap\bigcap_{j=0}^i P(\varpi_{\alpha_j})).$$
\end{prop}

\begin{proof}

Let $i\in\{0,\dots,r\}$. The face $F_{\{0,\dots,r\}\backslash\{i\},1,2}^{1+a_i}$ of $\tilde{Q}^{1+a_i}$ is the vertex $u_i^*$ and then the corresponding face of $Q^{1+a_i}$ is the vertex $\varpi_{\alpha_i}$. In particular, the $G$-orbit $\omega_{\{0,\dots,r\}\backslash\{i\},1,2}^i$ is closed and isomorphic to $G/P(\varpi_{\alpha_i})$.

Now, since $\tilde{\phi_i}$ is $G$-equivariant, it must be surjective. Moreover, the dimension of the fibers of $\tilde{\phi_i}$ is $$\dim E_i-\dim E'_i=(i+1+\dim G/(P(\varpi_{\alpha_{r+1}})\cap P(\varpi_{\alpha_{r+2}})\cap\bigcap_{j=0}^i P(\varpi_{\alpha_j})))-\dim G/P(\varpi_{\alpha_i})$$ that is $i+1+\dim P(\varpi_{\alpha_i})/(P(\varpi_{\alpha_{r+1}})\cap P(\varpi_{\alpha_{r+2}})\cap\bigcap_{j=0}^i P(\varpi_{\alpha_j}))$.
\end{proof}

\begin{cor}
The dimension of the fibers of $\tilde{\phi_i}$ is  $$i+1+\dim G/(P(\varpi_{\alpha_{r+1}})\cap P(\varpi_{\alpha_{r+2}}))+\sum_{j=0}^{i-1}\dim G/P(\varpi_{\alpha_j}).$$
In particular the dimensions $d_j$ of the $G/P(\varpi_{\alpha_j})$'s, which are projective space under $G_i=\operatorname{SL}_{d_j+1}$, are invariants of $X$.
\end{cor}
\begin{proof}
Since $r>1$, or $r=1$ and $\alpha_0,\,\alpha_1$ are not two simple roots of the same simple subgroup of $G$, the simple roots $\alpha_0,\dots,\alpha_r$ are respectively the first simple roots of $G_0,\dots,G_r$ that are of type $A$. (And $\alpha_{r+1},\,\alpha_{r+2}$ are simple roots of $G_{r+1}$.) Then the corollary can be easily deduced from the proposition. 
\end{proof}
 
\section{Appendix}
\begin{prop}\label{prop:smoothtriplebis}
	Let $(K,\beta,R,n)$ be a smooth quadruple. Then we are in one of the following cases, up to symmetries.
	\begin{enumerate}
	\item $n=1$ and one of the following cases occurs.
	\begin{enumerate}[label=(\alph*)]
	\item $K$ is of type $A_m$ ($m\geq 3$). Then, $\beta=\alpha_k$ with $3\leq k\leq m$ and $R=\{\alpha_1,\alpha_{k-1}\}$; or $\beta=\alpha_k$ with $4\leq k\leq m$ and $R=\{\alpha_i,\alpha_{i+1}\}$ with $1\leq i\leq k-2$.
	\item $K$ is of type $B_m$ ($m\geq 3$). Then, $\beta=\alpha_k$ with $3\leq k\leq m$ and $R=\{\alpha_1,\alpha_{k-1}\}$ or $R=\{\alpha_i,\alpha_{i+1}\}$ with $1\leq i\leq k-2$; or $\beta=\alpha_k$ with $1\leq k\leq m-2$ and $R=\{\alpha_{m-1},\alpha_{m}\}$; or $\beta=\alpha_{m-3}$ and $R=\{\alpha_{m-2},\alpha_{m}\}$.
	\item $K$ is of type $C_m$ ($m\geq 3$). Then, $\beta=\alpha_k$ with $3\leq k\leq m$ and $R=\{\alpha_1,\alpha_{k-1}\}$; or $\beta=\alpha_k$ with $4\leq k\leq m$ and $R=\{\alpha_i,\alpha_{i+1}\}$ with $1\leq i\leq k-2$; $\beta=\alpha_k$ with $1\leq k\leq m-2$ and $R=\{\alpha_i,\alpha_{i+1}\}$ with $1\leq i\leq k-2$.
	\item $K$ is of type $D_m$ ($m\geq 4$). Then, $\beta=\alpha_k$ with $3\leq k\leq m-2$ or $k=m$ and $R=\{\alpha_1,\alpha_{k-1}\}$; or $\beta=\alpha_k$ with $4\leq k\leq m-2$ or $k=m$ and $R=\{\alpha_i,\alpha_{i+1}\}$ with $1\leq i\leq k-2$; $\beta=\alpha_k$ with $1\leq k\leq m-4$ and $R=\{\alpha_{m-1},\alpha_{m}\}$; or $m\geq 5$, $\beta=\alpha_{m-3}$ and $R$ is any subset of cardinality~2 of $\{\alpha_{m-2},\alpha_{m-1},\alpha_{m}\}$; or $m\geq 5$, $\beta=\alpha_{m-2}$ and $R=\{\alpha_{m-1},\alpha_{m}\}$; all modulo symmetries.
	\item $K$ is of type $E_6$. Then $\beta=\alpha_1$ and $R=\{\alpha_{2},\alpha_{3}\}$; or $\beta=\alpha_2$ and $R=\{\alpha_{1},\alpha_{6}\}$, $\{\alpha_{1},\alpha_{3}\}$ or $\{\alpha_{3},\alpha_{4}\}$; or $\beta=\alpha_3$ and $R=\{\alpha_{2},\alpha_{6}\}$, $\{\alpha_{2},\alpha_{4}\}$, $\{\alpha_{4},\alpha_{5}\}$ or $\{\alpha_{5},\alpha_{6}\}$; or $\beta=\alpha_4$ and $R=\{\alpha_{1},\alpha_{3}\}$.
	\item  $K$ is of type $E_7$. Then $\beta=\alpha_1$ and $R=\{\alpha_{2},\alpha_{3}\}$; or $\beta=\alpha_2$ and $R=\{\alpha_{1},\alpha_{7}\}$, $\{\alpha_{1},\alpha_{3}\}$, $R=\{\alpha_{3},\alpha_{4}\}$, $\{\alpha_{4},\alpha_{5}\}$, $\{\alpha_{5},\alpha_{6}\}$ or $\{\alpha_{6},\alpha_{7}\}$; or $\beta=\alpha_3$ and $R=\{\alpha_{2},\alpha_{7}\}$, $\{\alpha_{2},\alpha_{4}\}$, $\{\alpha_{4},\alpha_{5}\}$, $\{\alpha_{5},\alpha_{6}\}$ or $\{\alpha_{6},\alpha_{7}\}$; or $\beta=\alpha_4$ and $R=\{\alpha_{1},\alpha_{3}\}$,  $\{\alpha_{5},\alpha_{7}\}$, $\{\alpha_{5},\alpha_{6}\}$ or $\{\alpha_{6},\alpha_{7}\}$; or $\beta=\alpha_5$ and $R=\{\alpha_{1},\alpha_{2}\}$,  $\{\alpha_{1},\alpha_{3}\}$, $\{\alpha_{3},\alpha_{4}\}$, $\{\alpha_{2},\alpha_{4}\}$ or $\{\alpha_{6},\alpha_{7}\}$; or $\beta=\alpha_6$ and $R=\{\alpha_{2},\alpha_{5}\}$.
	\item  $K$ is of type $E_8$. Then $\beta=\alpha_1$ and $R=\{\alpha_{2},\alpha_{3}\}$; or $\beta=\alpha_2$ and $R=\{\alpha_{1},\alpha_{8}\}$, $\{\alpha_{1},\alpha_{3}\}$, $R=\{\alpha_{3},\alpha_{4}\}$, $\{\alpha_{4},\alpha_{5}\}$, $\{\alpha_{5},\alpha_{6}\}$, $\{\alpha_{6},\alpha_{7}\}$ or $\{\alpha_{7},\alpha_{8}\}$; or $\beta=\alpha_3$ and $R=\{\alpha_{2},\alpha_{8}\}$, $\{\alpha_{2},\alpha_{4}\}$, $\{\alpha_{4},\alpha_{5}\}$, $\{\alpha_{5},\alpha_{6}\}$, $\{\alpha_{6},\alpha_{7}\}$ or $\{\alpha_{7},\alpha_{8}\}$; or $\beta=\alpha_4$ and $R=\{\alpha_{1},\alpha_{3}\}$,  $\{\alpha_{5},\alpha_{8}\}$, $\{\alpha_{5},\alpha_{6}\}$, $\{\alpha_{6},\alpha_{7}\}$ or $\{\alpha_{7},\alpha_{8}\}$; or $\beta=\alpha_5$ and $R=\{\alpha_{1},\alpha_{2}\}$,  $\{\alpha_{1},\alpha_{3}\}$, $\{\alpha_{3},\alpha_{4}\}$, $\{\alpha_{2},\alpha_{4}\}$, $\{\alpha_{6},\alpha_{8}\}$, $\{\alpha_{6},\alpha_{7}\}$ or $\{\alpha_{7},\alpha_{8}\}$; or $\beta=\alpha_6$ and $R=\{\alpha_{2},\alpha_{5}\}$ or $\{\alpha_7,\alpha_8\}$.
	\item $K$ is of type $F_4$. Then $\beta=\alpha_1$ and $R=\{\alpha_{3},\alpha_{4}\}$ or $\{\alpha_2,\alpha_{3}\}$; $\beta=\alpha_2$ and $R=\{\alpha_{3},\alpha_{4}\}$; $\beta=\alpha_3$ and $R=\{\alpha_{1},\alpha_{2}\}$; $\beta=\alpha_4$ and $R=\{\alpha_{2},\alpha_{3}\}$ or $\{\alpha_1,\alpha_{3}\}$.
	\end{enumerate}
	\item $R$ is empty or one of the following cases occurs.
	\begin{enumerate}[label=(\alph*)]
	\item $K$ is of type $A_m$ ($m\geq 2$). Then, $\beta=\alpha_1$ and $R$ is $\{\alpha_2\}$ or $\{\alpha_m\}$ (if $m\geq 3$); $\beta=\alpha_k$ with $2\leq k\leq \frac{m}{2}$ and $R$ is a subset of $\{\alpha_1,\alpha_{k+1}\}$,  $\{\alpha_1,\alpha_m\}$, $\alpha_{k-1},\alpha_{k+1}\}$ (if $k\geq 3$) or $\alpha_{k-1},\alpha_m\}$ (if $k\geq 3$); or $\beta=\alpha_\frac{m+1}{2}$ (if $m$ is odd) and $R$ is a subset of $\{\alpha_1,\alpha_m\}$ or $R=\{\alpha_{k-1}\}$, $\{alpha_1,\alpha_{k+1}\}$ or  $\alpha_{k-1},\alpha_{k+1}\}$ (if $m\geq 5$).
	\item $K$ is of type $B_m$ ($m\geq 3$). Then, $m=3$, $\beta=\alpha_1$ and $R$ is $\{\alpha_3\}$; $\beta=\alpha_k$ with $2\leq k\leq m-3$ and $R$ is $\{\alpha_1\}$ or $\{\alpha_{k-1}\}$ (if $k\geq 3$); or $\beta=\alpha_{m-2}$ ($m\geq 4$) and $R$ is a subset of $\{\alpha_1,\alpha_m\}$ or $\{\alpha_{m-3},\alpha_m\}$ (if $m\geq 5$); or $\beta=\alpha_m-1$ and $R$ is a subset of $\{\alpha_1,\alpha_m\}$ or $R$ is $\{\alpha_{m-2}\}$ (if $m\geq 4$) or $\{\alpha_{m-2},\alpha_m\}$ (if $m\geq 5$); or $\beta=\alpha_m$ and $R$ is $\{\alpha_1\}$ or $\{\alpha_{m-1}\}$.
	\item $K$ is of type $C_m$ ($m\geq 2$). Then, $\beta=\alpha_1$ and $R$ is $\{\alpha_2\}$; or $\beta=\alpha_k$ with $2\leq k\leq m-1$ ($m\geq 3$) and $R$ is a subset of $\{\alpha_1,\alpha_{k+1}\}$ or $\{\alpha_{k-1},\alpha_{k+1}\}$ (if $k\geq 3$ and $m\geq 4$); or $\beta=\alpha_m$ and $R=\{\alpha_{1}\}$ or $\{\alpha_{m-1}\}$ (if $m\geq 3$).
	\item $K$ is of type $D_m$ ($m\geq 4$). Then, $\beta=\alpha_k$ with $2\leq k\leq m-4$ ($m\geq 6$) and $R$ is $\{\alpha_1\}$ or $\{\alpha_{k-1}\}$ (if $k\geq 3$ and $m\geq 7$); or $\beta=\alpha_{m-3}$ and $R$ is $\{\alpha_{m-1}\}$, or a subset of $\{\alpha_1,\alpha_{m-1}\}$ (if $m\geq 5$) or $\{\alpha_{m-4},\alpha_{m-1}\}$ (if $m\geq 6$); or $\beta=\alpha_{m-2}$ and $R$ is $\{\alpha_1\}$, $\{\alpha_1,\alpha_{m-1}\}$ or $\{\alpha_1,\alpha_{m-1},\alpha_m\}$, or $R$ is a subset of  $\{\alpha_{m-3},\alpha_{m-1}\}$ (if $m\geq 5$), $R$ is $\{\alpha_{m-3},\alpha_{m-1},\alpha_m\}$ (if $m\geq 5$); or $\beta=\alpha_m$ and $R$ is $\{\alpha_1\}$ or $\{\alpha_{m-1}\}$.
	\item $K$ is of type $E_6$. Then $\beta=\alpha_2$ and $R=\{\alpha_{1}\}$; or $\beta=\alpha_3$ and $R$ is a subset of $\{\alpha_1,\alpha_{2}\}$ or $\{\alpha_1,\alpha_{6}\}$; or $\beta=\alpha_4$ and $R$ is subset of $\{\alpha_{2},\alpha_i,\alpha_j\}$ with $i=1$ or 3 and $j=5$ or 6 modulo symmetries.
	\item $K$ is of type $E_7$. Then $\beta=\alpha_2$ and $R=\{\alpha_{1}\}$ or $\{\alpha_7\}$; or $\beta=\alpha_3$ and $R$ is a subset of $\{\alpha_1,\alpha_{2}\}$ or $\{\alpha_1,\alpha_{7}\}$;
	or $\beta=\alpha_4$ and $R$ is subset of $\{\alpha_{2},\alpha_i,\alpha_j\}$ with $i=1$ or 3 and $j=5$ or 7;
	 or $\beta=\alpha_5$ and $R$ is a subset of $\{\alpha_i,\alpha_j\}$ with $i=1$ or 2 and $j=6$ or 7; 
	  or $\beta=\alpha_6$ and $R=\{\alpha_7\}$.
	\item $K$ is of type $E_8$. Then $\beta=\alpha_2$ and $R=\{\alpha_{1}\}$ or $\{\alpha_8\}$;
	 or $\beta=\alpha_3$ and $R$ is a subset of $\{\alpha_1,\alpha_2\}$ or $\{\alpha_1,\alpha_{8}\}$; 
	 or $\beta=\alpha_4$ and $R$ is subset of $\{\alpha_{2},\alpha_i,\alpha_j\}$ with $i=1$ or 3 and $j=5$ or 8;
	 or $\beta=\alpha_5$ and $R$ is a subset of $\{\alpha_i,\alpha_j\}$ with $i=1$ or 2 and $j=6$ or 8;
	 or $\beta=\alpha_6$ and $R$ is $\alpha_7$ or $\alpha_8$; 
	  or $\beta=\alpha_7$ and $R=\{\alpha_8\}$.
	\item $K$ is of type $F_4$. Then $\beta=\alpha_1$ and $R=\{\alpha_{4}\}$; $\beta=\alpha_2$ and $R$ is a subset of $\{\alpha_{1},\alpha_{3}\}$ or $\{\alpha_1,\alpha_4\}$; $\beta=\alpha_3$ and $R$ is a subset of $\{\alpha_{1},\alpha_{4}\}$ or $\{\alpha_2,\alpha_4\}$.
	\item $K$ is of type $G_2$. Then $\beta=\alpha_1$ and $R=\{\alpha_2\}$; or $\beta=\alpha_2$ and $R=\{\alpha_1\}$
	\end{enumerate}
	\end{enumerate}
	\end{prop}
The proof, which is a long but not difficult case by case verification, is left to the reader.

%%%%%%%%%%%%%%%%%%%%%
% References
%%%%%%%%%%%%%%%%%%%%%

\providecommand{\bysame}{\leavevmode\hbox to3em{\hrulefill}\thinspace}

\end{document}